\documentclass{cedram-alco}
\usepackage[all]{xy}
\usepackage[latin1]{inputenc}
\usepackage{csquotes}
\usepackage[usenames,dvipsnames]{color}
% ...

\newcommand{\rk}{\operatorname{rk}}
\newcommand{\cork}{\operatorname{cork}}
\newcommand{\im}{\operatorname{im}}
\newcommand{\KK}{\mathbb{K}}
\renewcommand{\c}{\mathrm{c}}

\equalenv{ex}{exam}
\equalenv{rem}{rema}
\equalenv{lem}{lemm}

\title
[Universal Tutte characters via combinatorial coalgebras]
{Universal Tutte characters\\via combinatorial coalgebras}

\author[\initial{C.} Dupont]{\firstname{Cl\'{e}ment} \lastname{Dupont}}
\address{IMAG\\Universit\'{e} de Montpellier\\CNRS\\Montpellier\\France}
\email{clement.dupont@umontpellier.fr}

\author[\initial{A.} Fink]{\firstname{Alex} \lastname{Fink}}
\address{School of Mathematical Sciences\\Queen Mary University of London\\UK}
\email{a.fink@qmul.ac.uk}

\author[\initial{L.} Moci]{\firstname{Luca} \lastname{Moci}}
\address{IMJ-PRG\\Universit\'{e} Paris-Diderot Paris 7\\Paris\\France}
\email{lucamoci@hotmail.com}



\begin{document}
\begin{abstract}
This work discusses the extraction of meaningful invariants of combinatorial objects from coalgebra or bialgebra structures.
The Tutte polynomial is an invariant of graphs well known for the formula which computes it recursively by deleting and contracting edges, and for its universality with respect to similar recurrence. 
We generalize this to all classes of combinatorial objects with deletion and contraction operations, 
associating to each such class a universal Tutte character by a functorial procedure.
We show that these invariants satisfy a universal property and convolution formulae similar to the Tutte polynomial. 
With this machinery we recover classical invariants for delta-matroids, matroid perspectives, relative and colored matroids, generalized permutohedra, and arithmetic matroids. 
We also produce some new invariants along with new convolution formulae.
\end{abstract}

\maketitle

\section{Introduction}
The Tutte polynomial is surely the single most appreciated invariant of matroids and graphs. 
For one, its specializations include any function which can be computed recursively
from its evaluations for a deletion and a contraction, as a weighted sum.
Many invariants of independent interest in matroid and graph theory do satisfy such a recurrence,
such as the chromatic and flow polynomials; examples occur also in knot theory and in statistical physics.
Moreover, the Tutte polynomial satisfies interesting identities like the
convolution formula of Kook--Reiner--Stanton \cite{KRS},
which also follows from work of Etienne--Las Vergnas \cite{etiennelasvergnas}.

The present work is concerned with the 
many other combinatorial objects which possess invariants with properties reminiscent of these,
such as matroid perspectives and their Las Vergnas polynomial \cite{lasvergnasextensions} or delta-matroids and their Bollob\'{a}s--Riordan polynomial \cite{BR}, which  are both matroidal frameworks for certain topological embeddings of graphs in surfaces.
Another example, which served as our initial motivation, is arithmetic matroids, introduced by D'Adderio and the third author \cite{Md1}, whose arithmetic Tutte polynomial satisfies a convolution formula (proved by Backman--Lenz \cite{backmanlenz} in a special case, and then in the present paper in greater generality).
To formulate a ``Tutte-like'' deletion-contraction recurrence for a class of combinatorial objects,
we require that each object have an underlying set, 
and that there are two ways to
create new objects
by removing elements of this set,
\emph{deletion} and \emph{contraction}, subject to some axioms.
Classes with this structure are called \emph{minors systems}. With every minor system is naturally associated a coalgebra; moreover if on the minors system is defined some sort of \emph{direct sum}, the coalgebra is endowed with a product that makes it into a \emph{bialgebra}. Bialgebras have proven to be a powerful language in combinatorics: see especially \cite{JoniRota}.

Our main contribution is to define a \emph{universal Tutte character}
for any minors system (Definition~\ref{def:universal Tutte}).
This definition unifies a great number of known Tutte-like invariants
by specializing the minors system in our universal Tutte character;
the second half of our paper is dedicated to presenting these.
Our invariant is universal in the same sense as the Tutte polynomial of matroids and graphs is:
namely, any function satisfying a deletion-contraction recurrence, suitably understood,
is an evaluation thereof (Proposition~\ref{prop:universal}). The universal Tutte character also satisfies a universal convolution formula that specializes to every convolution formula in a minors system we are aware of, including the formula of Kook--Reiner--Stanton, and gives rise to several new examples. Moreover, our construction is functorial with respect to the minors system. 

Duchamp, Hoang-Nghia, Krajewski and Tanasa \cite{duchampetal}  have recently provided similar machinery recovering the Tutte polynomial of a matroid.
This was generalized by
Krajewski, Moffatt and Tanasa \cite{KMT} who associated a polynomial Tutte invariant with every graded connected Hopf algebra. 
Our approach builds on the latter, making several improvements in generality and canonicity.

For one, we lift a restriction on minors systems in \cite{KMT} 
by allowing multiple non-isomorphic structures with empty underlying set.
As one example of how this is useful, our universal Tutte invariant for graphs
distinguishes the many different graphs with empty edge set: 
it is thus strictly more general than the Tutte polynomial.
We obtain as a specialization Tutte's \emph{dichromatic polynomial} \cite{Tuttedichromatic},
which is not an evaluation of the Tutte polynomial because it requires counting connected components.

As a consequence of that choice, with every minors system is naturally associated a bialgebra, which is not necessarily Hopf (unlike in \cite{KMT}). Classes of structures with multiple empty members cannot be naturally handled as connected (hence Hopf) bialgebras, 
as the extra degree-zero elements can frustrate the definition of an antipode (see Remark \ref{noHopf}). A Hopf algebra may be obtained as a quotient, but this comes at the price of losing a consistent amount of information, hence reducing the number of invariants that can be recovered in this way.

Lacking a Hopf algebra structure is no obstacle to our programme, which needs only a bialgebra structure.
In fact the amount of information we use about multiplication is small.
For minors systems with a unique empty member the multiplication is irrelevant, and we can work with a coalgebra.

Rather than the language of Hopf algebras,
we prefer to use the language of linear species and comonoids therein,
which is more convenient and canonical, 
not requiring us to fix a single chain of ground sets and relabel after every set operation.
This language was already present in Schmitt's article \cite{schmittspecies}
(preceding the more often cited \cite{schmitt}).

Our construction employs 
a universal bialgebra norm taking values in the monoid ring of
what we call the \emph{Grothendieck monoid} of a minors system.
This ring is not necessarily a polynomial ring, and when it is not, the relations among its generators
automatically encapsulate the relations which \cite{KMT} had to attach to their polynomial rings with the somewhat fiddly machinery of ``uniform selectors''.
In Theorem \ref{thm:pres X(S) second} we give a quadratic presentation for the Grothendieck monoid, 
making the targets of our invariants easy to compute in practice. Then we use the machinery that we built to obtain several convolution formulae for different classes of combinatorial objects. Some of these formulae are well known, while others are (to the best of our knowledge) new: see for instance Propositions \ref{prop:convolution flats}, \ref{prop:multiconvolution matroids} and \ref{prop:convolution cyclotomic} for the classical Tutte polynomial, Proposition \ref{prop:convolution LV} for the Las Vergnas polynomial, Proposition \ref{BRKung} for the Bollob\'{a}s--Riordan polynomial and Theorem \ref{thm:arith conv} for the arithmetic Tutte polynomial.

We have chosen to exemplify our results with mostly matroid-like combinatorial structures, which are arguably simpler to deal with than topological examples. Our formalism, and in particular the mechanism of twist maps, should help uncover new invariants in the latter class and produce interesting convolution formulae. This will be the subject of a subsequent article.

\subsection*{Layout}
The structure of this paper is as follows.  
In Section~\ref{sec:minors systems} we give the definitions of minors systems and comonoids in set species.
Section~\ref{sec:main} contains the main results, including the definitions and statements of our universal invariants and formulae.
Section~\ref{sec:more} presents a number of further results which are less essential for our main developments.

The remainder of the paper, Sections \ref{sec:matroids} through \ref{sec:arith}, 
comprises a sequence of applications of our theory to numerous individual minors systems.
We work out Grothendieck groups, universal Tutte characters, and in many cases universal convolution formulae,
and point out how these specialize to invariants and formulae present in the literature.
Section \ref{sec:matroids} covers the minors system of matroids, 
of which all the subsequent sections are in one way or another generalizations, together with the minors system of graphs.
Of the following sections, Section \ref{sec:dmp} is on delta-matroids, matroid perspectives, and their ilk
is called on in Section \ref{sec:relative} on relative matroids,
but there are no (or at most incidental) dependences between these and
Section \ref{sec:SF} on polymatroids and generalised permutohedra, 
Section \ref{sec:colored} on colored matroids, 
or Section \ref{sec:arith} on arithmetic matroids,
nor among the latter three, so the reader should have no trouble taking these in any order.
The last four sections are also new by comparison with \cite{KMT}.

\subsection*{Notation}
We fix a commutative ring with unit $\KK$, which will serve as a coefficient ring throughout the article. For most applications the case $\KK=\mathbb{Z}$ is enough. 

\tableofcontents

\section{Minors systems}\label{sec:minors systems}

	\subsection{Set species and minors systems}
		A set species \cite{joyalespeces} is a structure which one should think of as associating to each finite set $E$ the set of structures of some combinatorial type on~$E$, for example the set of matroids with ground set $E$.
		
		\begin{defi}
		A \emph{set species} is a functor from the category of finite sets and bijections to the category of sets. 
		\end{defi}
		
		More concretely, a set species $\mathsf{S}$ associates to every finite set $E$ a set $\mathsf{S}[E]$ and to every bijection $\sigma:E\stackrel{\sim}{\rightarrow} E'$ a map $\mathsf{S}[\sigma]:\mathsf{S}[E]\rightarrow \mathsf{S}[E']$, such that $\mathsf{S}[\sigma\circ \tau]=\mathsf{S}[\sigma]\circ \mathsf{S}[\tau]$ and $\mathsf{S}[\mathrm{id}]=\mathrm{id}$. The set $\mathsf{S}[n]\doteq \mathsf{S}[\{1,\ldots,n\}]$
		then has an action of the symmetric group $\Sigma_n$ on $n$ letters, and the set species $\mathsf{S}$ can be recovered from the \emph{symmetric sequence} $\{\mathsf{S}[n]\, , \, n\geq 0\}$. For $\mathsf{S}$ a set species and $n\geq 0$ an integer, we set
		$$\mathsf{S}_n\doteq \operatorname{colim}_{|E|=n} \mathsf{S}[E] \simeq \mathsf{S}[n] / \Sigma_n\ .$$
		We denote by the same symbol an element $X\in\mathsf{S}[E]$ and its \enquote{isomorphism class} $X\in\mathsf{S}_n$. We let $\mathsf{S}_\bullet=\bigsqcup_{n\geq 0}\mathsf{S}_n$.
		
		\begin{ex}
		A trivial example is the set species $\mathsf{Set}$ for which $\mathsf{Set}[E]=\{E\}$ is a singleton for every finite set $E$. Our prototypical examples are:
		\begin{enumerate}[--]
		    \item the set species $\mathsf{Mat}$, for which $\mathsf{Mat}[E]$ is the set of matroids with ground set $E$, and $\mathsf{Mat}_n$ is the set of isomorphism classes of matroids on an $n$-element set;
		    \item the set species $\mathsf{Gra}$, for which $\mathsf{Gra}[E]$ is the set of graphs with set of edges $E$, and $\mathsf{Gra}_n$ is the set of isomorphism classes of graphs with $n$ edges. Here graphs are allowed to have loops and multiples edges.
		\end{enumerate}
		For more on these example, see Section \ref{sec:matroids}.
		\end{ex}
		
		We are interested in set species which carry \emph{minor} operations like those of our prototypical examples,
		to wit, restriction (or equivalently deletion) and contraction.
		Many important and naturally occurring species of this form, again including $\mathsf{Mat}$,
		also have a direct sum operation. Concretely, for the set species $\mathsf{Gra}$, the deletion is obtained by removing a given edge from a graph, while the contraction by contracting that edge (that is, removing it and identifying the vertices incident to it); finally, the direct sum of two graphs is obtained by taking their disjoint union.

		For our theory we need only direct sums for which one object belongs to~$\mathsf{S}[\varnothing]$,
		which justifies the following definition of \emph{minors system}.
		Often the existence of such a limited direct sum is automatic: see Remark~\ref{rem:sum when connected}.
		Minors systems which have a direct sum of general pairs of objects we will call \emph{multiplicative minors systems} (Definition~\ref{def:mms}).

		\begin{defi}\label{def:ms}
		A \emph{minors system} is a set species $\mathsf{S}$ with the following extra structure.
		\begin{enumerate}
		\item For every decomposition $E=A\sqcup B$, a \emph{coproduct map} $\Delta_{A,B}:\mathsf{S}[E]\rightarrow \mathsf{S}[A]\times \mathsf{S}[B]$.
		\item For every finite set $E$, a \emph{product map} $\rho_E: \mathsf{S}[E]\times \mathsf{S}[\varnothing]\rightarrow\mathsf{S}[E]$.
		\end{enumerate}
		We introduce the notations $\Delta_{A,B}(X)=(X|A,X/A)$ and we call the operations $X\mapsto X|A$ and $X\mapsto X/A$ \emph{restriction} and \emph{contraction} respectively. We also set $\rho_E(X,Y)=X\oplus Y$, that we call the \emph{direct sum} operation. These operations must satisfy the following compatibilities.
		\begin{enumerate}
		\item[(M1)] The coproduct maps are functorial: for every decomposition $E=A\sqcup B$, $E'=A'\sqcup B'$, for every bijection $\sigma:E\stackrel{\sim}{\rightarrow} E'$ such that $\sigma(A)=A'$ and $\sigma(B)=B'$, for every $X\in \mathsf{S}[E]$, we have 
		$$(\mathsf{S}[\sigma](X))|A' = \mathsf{S}[\sigma_{|A}](X|A) \;\; \mbox{ and } \;\; (\mathsf{S}[\sigma](X))/A' = \mathsf{S}[\sigma_{|B}](X/A) $$
		\item[(M2)] The coproduct maps are coassociative: for every decomposition $E=A\sqcup B\sqcup C$ and every $X\in\mathsf{S}[E]$ we have
		$$(\mathrm{id}\times\Delta_{B,C}) (\Delta_{A,B\sqcup C} (X)) = (\Delta_{A,B}\times \mathrm{id}) (\Delta_{A\sqcup B,C}(X))\ .$$
		In other words:
		$$(X| A\sqcup B) | A = X|A \;\;,\;\; (X/A)|B = (X|A\sqcup B)/A\;\;,\;\; (X/A)/B = X/(A\sqcup B)\ .$$
		\item[(M3)] The coproduct maps are counital: for every finite set $E$ and every $X\in\mathsf{S}[E]$ we have 
		$$\mathrm{pr}_1(\Delta_{E,\varnothing}(X))=X=\mathrm{pr}_2(\Delta_{\varnothing,E}(X))\ .$$
		In other words:
		$$X|E=X \;\;\textnormal{ and }\;\; X/\varnothing =X\ .$$
		\item[(M4)] The product maps $\rho_E$ are functorial: for every bijection $\sigma:E\stackrel{\sim}{\rightarrow} E'$ and every $X\in \mathsf{S}[E]$, $Y\in\mathsf{S}[\varnothing]$ we have
		$$\mathsf{S}[\sigma](X\oplus Y) = \mathsf{S}[\sigma](X) \oplus Y \ .$$
		\item[(M5)] The product maps $\rho_E$ are associative: for every finite set $E$ and every $X\in\mathsf{S}[E]$, $Y,Z\in\mathsf{S}[\varnothing]$, we have
		$$X\oplus (Y\oplus Z) = (X\oplus Y)\oplus Z\ .$$
		\item[(M6)] The product maps $\rho_{\varnothing}$ are commutative: for every $X,Y\in\mathsf{S}[\varnothing]$ we have
		$$X\oplus Y= Y\oplus X\ .$$
		\item[(M7)] The direct sum operation has a neutral element: there exists a (necessarily unique) element $1_{\mathsf{S}}\in\mathsf{S}[\varnothing]$ such that for every finite set $E$ and every $X\in\mathsf{S}[E]$ we have 
		$$X\oplus 1_{\mathsf{S}}= X\ .$$
		\item[(M8)] The direct sum operation is compatible with restriction and contraction: for every finite set $E$, every subset $A\subseteq E$, and every $X\in\mathsf{S}[E]$, $Y\in\mathsf{S}[\varnothing]$, we have
		$$(X\oplus Y)|A= (X|A) \oplus Y \;\;\mbox{ and }\;\; (X\oplus Y)/A = (X/A)\oplus Y\ .$$
		\end{enumerate}
		\end{defi}

		\begin{rem}
		A set species with coproduct maps that satisfy axioms (M1), (M2), (M3) is sometimes called a \emph{comonoid in set species}.
		\end{rem}

		\begin{rem}
		Axioms (M4), (M5), (M6) and (M7) mean that $\mathsf{S}[\varnothing]$ is a commutative monoid which acts functorially on the sets $\mathsf{S}[E]$.
		\end{rem}
		
		\begin{defi}
		A minors system $\mathsf{S}$ is said to be \emph{connected} if $\mathsf{S}[\varnothing]$ consists only of the object $1_{\mathsf{S}}$.
		\end{defi}
		
		The word ``connected'' here has the force it has in ``connected Hopf algebra'', not in say ``connected graph'' or ``connected matroid''. 
		
		\begin{rem}\label{rem:sum when connected}
		In a connected minors system there is no choice of direct sum operation:
		it is forced by the axiom $X\oplus 1_{\mathsf{S}}=X$ for every $X\in\mathsf{S}[E]$, 
		for which axioms (M4), (M5), (M6), (M7), (M8) are automatically satisfied. 
		In other words, a connected minors system is a set species $\mathsf{S}$ with operations of restriction and contraction 
		that satisfy axioms (M1), (M2), (M3) and such that there is a unique element in $\mathsf{S}[\varnothing]$.
		\end{rem}

		It will be convenient to introduce the \emph{deletion} in a minors system, defined for  $X\in\mathsf{S}[E]$ and a subset $A\subseteq E$ by $X\backslash A\doteq X|A^\c$ with $A^\c\doteq (E(X)\setminus A)$. This allows to rewrite the coassociativity axiom (M2) in a more symmetric way: for disjoint subsets $A, B\subseteq E$ and for $X\in\mathsf{S}[E]$ we have
		$$(X\backslash A)\backslash B=X\backslash (A\sqcup B) \;\;,\;\; (X\backslash A)/B=(X/B)\backslash A \;\;,\;\; (X/A)/B=X/(A\sqcup B)\ .$$
		
		\begin{defi}\label{def:mms}
		A \emph{multiplicative minors system} is a minors system together with the data, for every decomposition $E=A\sqcup B$, of a map
		$$\mu_{A,B}:\mathsf{S}[A]\times\mathsf{S}[B]\rightarrow \mathsf{S}[E]\ ,$$
		such that $\mu_{E,\varnothing}=\rho_E$. We denote $X\oplus Y=\mu_{A,B}(X,Y)$. The following compatibilities, which imply (M4), (M5), (M6), (M7), (M8), must be satisfied.
		\begin{enumerate}
		\item[(M4')] The product maps $\mu_{A,B}$ are functorial: for every bijection $\sigma:E\stackrel{\sim}{\rightarrow} E'$, for every decomposition $E=A\sqcup B$, and for every $X\in \mathsf{S}[A]$, $Y\in\mathsf{S}[B]$, we have
		$$\mathsf{S}[\sigma](X\oplus Y)=\mathsf{S}[\sigma_{|A}](X)\oplus\mathsf{S}[\sigma_{|B}](Y)\ .$$
		\item[(M5')] The product maps $\mu_{A,B}$ are associative: for every decomposition $E=A\sqcup B\sqcup C$ and every $X\in\mathsf{S}[A], Y\in\mathsf{S}[B], Z\in\mathsf{S}[C]$, we have
		$$X\oplus (Y\oplus Z) = (X\oplus Y)\oplus Z\ .$$
		\item[(M6')] The product maps $\mu_{A,B}$ are commutative: for every decomposition $E=A\sqcup B$ and every $X\in\mathsf{S}[A]$, $Y\in\mathsf{S}[B]$, we have
		$$X\oplus Y = Y\oplus X\ .$$
		\item[(M7')] The direct sum operation has a neutral element: there exists a (necessarily unique) element $1_{\mathsf{S}}\in\mathsf{S}[\varnothing]$ such that for every finite set $E$ and every $X\in\mathsf{S}[E]$ we have 
		$$X\oplus 1_{\mathsf{S}}= X\ .$$
		\item[(M8')] The direct sum operation is compatible with restriction and contraction: for every decomposition $E=A\sqcup B\sqcup C\sqcup D$, for every $X\in\mathsf{S}[A\sqcup C]$, $Y\in\mathsf{S}[B\sqcup D]$, we have
		$$(X\oplus Y)|(A\sqcup B)=(X|A)\oplus (Y|B) \;\;\mbox{ and }\;\; (X\oplus Y)/(A\sqcup B) = (X/A)\oplus (Y/B)\ .$$
		\end{enumerate}
		\end{defi}	
		
		\begin{rem}\label{rem:cHmiss}
		A multiplicative minors system is sometimes called a \emph{bimonoid in set species}, and a connected multiplicative minors system is called a \emph{connected Hopf monoid in set species} in \cite{ardila-aguiar}.
		\end{rem}

		\begin{rem}
		There are two differences between our definition of a minors system and the definition of \cite{KMT}. The first and most important difference is that in [\emph{loc.\;cit.}]\ all minors systems are assumed to be connected; our setting discloses a wider range of applications. The second difference is that we use the language of set species whereas in [\emph{loc.\;cit.}]\ only isomorphism classes of combinatorial structures are considered. This allows us to retain more information and to have well-defined operations of restriction and contraction without having to choose representatives for isomorphism classes. All minors systems in [\emph{loc.\;cit.}]\ can be upgraded to minors systems in our sense.
		\end{rem}

	\subsection{Comonoids in linear species and coalgebras}
	
		We fix a commutative ring with unit $\KK$.

		\begin{defi}
		A linear species is a functor from the category of finite sets and bijections to the category of $\KK$-modules.
		\end{defi}
		
		More concretely, a linear species $\mathsf{V}$ associates to every finite set $E$ a $\KK$-module $\mathsf{V}[E]$ and to every bijection $\sigma:E\stackrel{\sim}{\rightarrow} E'$ a $\KK$-linear map $\mathsf{V}[\sigma]:\mathsf{V}[E]\rightarrow \mathsf{V}[E']$, such that $\mathsf{V}[\sigma\circ \tau]=\mathsf{V}[\sigma]\circ \mathsf{V}[\tau]$ and $\mathsf{V}[\mathrm{id}]=\mathrm{id}$. Recall from \cite{aguiarmahajanbook} that the category of linear species has the structure of a symmetric monoidal category with respect to the Cauchy product, defined by
		$$(\mathsf{V}\cdot \mathsf{W})[E] = \bigoplus_{E=A\sqcup B}\mathsf{V}[A]\otimes \mathsf{W}[B]\ .$$
		The unit $\mathsf{1}$ for this monoidal structure satisfies $\mathsf{1}[\varnothing]=\KK$ and $\mathsf{1}[E]=0$ for $E\neq\varnothing$.\medskip
		
		For a linear species $\mathsf{V}$ and an integer $n\geq 0$ we set
		$$\mathsf{V}_n\doteq \operatorname{colim}_{|E|=n}\mathsf{V}[E] \simeq \mathsf{V}[n]/\Sigma_n\ .$$
		This defines a functor from the category of linear species to the category of (non-negatively) graded $\KK$-modules, called the \emph{Fock functor} \cite[Chapter 15]{aguiarmahajanbook}. This functor is monoidal. In most of what follows, the reader may work \emph{after applying the Fock functor}, and mentally replace \enquote{monoid in linear species} by \enquote{graded algebra} and \enquote{comonoid in linear species} by \enquote{graded coalgebra}.\medskip
		
		We let $S\mapsto \KK S$ denote the linearization functor from sets to $\KK$-modules. Applying this functor to a set species $\mathsf{S}$ gives rise to a linear species that we simply denote $\KK\mathsf{S}$. For $\mathsf{S}$ a minors system, the coproduct maps gives rise to linear maps that we still denote by
		$$\Delta_{A,B}:\KK\mathsf{S}[E]\rightarrow\KK\mathsf{S}[A]\otimes\KK\mathsf{S}[B]\ .$$
		These maps assemble to a morphism $\Delta:\KK\mathsf{S}\rightarrow \KK\mathsf{S}\cdot\KK\mathsf{S}$ in the category of linear species. The map $\mathsf{S}[\varnothing]\rightarrow\{*\}$ gives rise to a morphism $\varepsilon:\KK\mathsf{S}\rightarrow\mathsf{1}$.
		
		\begin{prop}
		If $\mathsf{S}$ is a minors system then $(\KK\mathsf{S},\Delta,\varepsilon)$ is a comonoid in linear species.
		\end{prop}
		
		\begin{proof}
		The coassociativity of $\Delta$ is a consequence of axiom (M2). The compatibility between $\Delta$ and $\varepsilon$ is a consequence of axiom (M3).
		\end{proof}
		
		The Fock functor sends $\KK\mathsf{S}$ to the graded $\KK$-module with degree $n$ component
		$$\operatorname{colim}_{|E|=n}\KK\mathsf{S}[E] \simeq \KK\mathsf{S}_n\ .$$
		The collection of the maps $\Delta_{A,B}$ give rise to linear maps
		$$\Delta_{m,n}:\KK\mathsf{S}_{m+n} \rightarrow \KK\mathsf{S}_m\otimes\KK\mathsf{S}_n\ .$$
		These maps, together with the map $\varepsilon:\KK\mathsf{S}_0\rightarrow\KK$, endow the graded $\KK$-module $\KK\mathsf{S}_\bullet$ with the structure of a graded coalgebra, where the coproduct is given by the familiar formula
		$$\Delta(X)=\sum_{A\subseteq E}X|A \otimes X/A \ .$$
		
		\medskip
		
		The direct sum operation gives rise to maps $\rho_E:\KK\mathsf{S}[E]\otimes\KK\mathsf{S}[\varnothing]\rightarrow \KK\mathsf{S}[E]$, and the unit element $1_{\mathsf{S}}$ gives rise to a map $\eta:\mathsf{1}\rightarrow \KK\mathsf{S}$. Because of axioms (M4), (M5), (M6), (M7), this gives $\KK\mathsf{S}[\varnothing]$ the structure of a commutative algebra and every $\KK\mathsf{S}[E]$ the structure of a module over $\KK\mathsf{S}[\varnothing]$, which is functorial in $E$. These structures are compatible with the comonoid structure in a sense that we only make explicit in the context of a multiplicative minors system. In this case one gets maps 
		$$\mu_{A,B}:\KK\mathsf{S}[A]\otimes\KK\mathsf{S}[B]\rightarrow\KK\mathsf{S}[E]$$ 
		which assemble to a morphism $\mu:\KK\mathsf{S}\cdot\KK\mathsf{S}\rightarrow\KK\mathsf{S}$ in the category of linear species.

		\begin{prop}
		If $\mathsf{S}$ is a multiplicative minors system then $(\KK\mathsf{S},\Delta,\varepsilon,\mu,\eta)$ is a commutative bimonoid in linear species.
		\end{prop}
		
		\begin{proof}
		The associativity of $\mu$ is a consequence of axiom (M5') and its commutativity is a consequence of axiom (M6'). The compatibility between $\mu$ and $\eta$ is a consequence of axiom (M7'). The compatibility between $\Delta$ and $\mu$ is a consequence of axiom (M8'). The compatibility between $\Delta$ and $\eta$ follows from axiom (M3) since $\Delta_{\varnothing,\varnothing}(1_{\mathsf{S}})=(1_{\mathsf{S}}|\varnothing)\otimes (1_{\mathsf{S}}/\varnothing)=1_{\mathsf{S}}\otimes 1_{\mathsf{S}}$. The compatibility between $\varepsilon$ and $\mu$ is trivial, as is the compatibility between $\varepsilon$ and $\eta$.
		\end{proof}
		
		One can determine exactly when $\KK\mathsf{S}$ is a Hopf monoid in linear species, or in more classical terms when the graded bialgebra $\KK\mathsf{S}_\bullet$ is a Hopf algebra.
		
		\begin{prop}\label{prop:hopf when}
		For a multiplicative minors system $\mathsf{S}$, the bimonoid in linear species $\KK\mathsf{S}$ is a Hopf monoid in linear species if and only if the commutative monoid $\mathsf{S}[\varnothing]$ is an abelian group. This is the case in particular if $\mathsf{S}$ is connected.
		\end{prop}
		
		\begin{proof}
		By \cite[Proposition 8.10]{aguiarmahajanbook}, $\KK\mathsf{S}$ is a Hopf monoid in linear species if and only if $\KK\mathsf{S}[\varnothing]$ is a Hopf algebra. Since $\KK\mathsf{S}[\varnothing]$ is the monoid algebra of the monoid $\KK\mathsf{S}$, the claim follows.
		\end{proof}
		
		\begin{rem}\label{rem:Hopf quotient}
		For $\mathsf{S}$ a minors system and $E$ a finite set, one can perform the quotient of $\KK\mathsf{S}[E]$ by the sub-$\KK$-module spanned by elements $(U\oplus S-S)$ for $U\in\mathsf{S}[\varnothing]$ and $S\in\mathsf{S}[E]$. The collection of these quotients inherits the structure of a connected comonoid in linear species. In the context of a multiplicative minors system, this quotient becomes a connected bimonoid in linear species, and thus a Hopf monoid in linear species. 
		As will appear clearly when we introduce Tutte characters, it is more convenient not to lose the information contained in $\mathsf{S}[\varnothing]$ and work with the whole bimonoid $\KK\mathsf{S}$.
		See Remark~\ref{rem:matroids are normalized} for an informal examination of the combinatorial meaning of this quotient in one example.
		\end{rem}

\section{Norms and universal Tutte characters}\label{sec:main}
		
	\subsection{Norms}
	
		For $U$ a set, we overload the notation $U$ to refer also to
		the constant set species defined by $U[E]=U$ for every finite set $E$, and $U[\sigma]=\mathrm{id}$ for every bijection $\sigma:E\stackrel{\sim}{\rightarrow}E'$. For $\mathsf{S}$ a set species, a morphism of set species $f:\mathsf{S}\rightarrow U$ is the datum, for every finite set $E$, of a map $f[E]:\mathsf{S}[E]\rightarrow U$, such that for every bijection $\sigma:E\stackrel{\sim}{\rightarrow} F$ we have $f[F]\circ \mathsf{S}[\sigma] = f[E]$. In other words, a morphism of set species $f:\mathsf{S}\rightarrow U$ is the same thing as a collection of maps $f_n:\mathsf{S}_n\rightarrow U$ for $n\geq 0$. The same holds for linear species.
		
		In all (commutative) monoids the law is written multiplicatively, unless otherwise specified.
		For example, we write the free monoid on one generator as $u^{\mathbb N}$, where $u$ is the generator,
		rather than simply $\mathbb N$ as we could if we worked additively.
	
		\begin{defi}
		Let $\mathsf{S}$ be a minors system. A \emph{norm} for $\mathsf{S}$ is the data of a commutative monoid $U$ and a morphism of set species $N:\mathsf{S}\rightarrow U$, that satisfies the following axioms.
		\begin{enumerate}
		\item[(N1)] $N(X)=N(X|A)\,N(X/A)$ for every $X\in\mathsf{S}[E]$ and $A\subseteq E$.
		\item[(N2)] $N(X\oplus Y) = N(X)$ for every $X\in\mathsf{S}[E]$ and $Y\in \mathsf{S}[\varnothing]$.
		\item[(N3)] $N(1_{\mathsf{S}})=1$.
		\end{enumerate}
		\end{defi}
		
		We note that relations (N2) and (N3) imply that we have $N(X)=1$ for every $X\in\mathsf{S}[\varnothing]$.\medskip
	
		A case that will be of special interest for us is when $U$ is the multiplicative monoid $R^\times$ of a commutative $\KK$-algebra~$R$. In this case we will denote the $\KK$-linear extension of~$N$ again by $N:\KK\mathsf{S}\rightarrow R$, and will still call the extension a norm.

		\begin{rem}\label{rem:rank functions}
		In \cite[Theorem 1]{KMT} one considers morphisms of set species $r:\mathsf{S}\rightarrow\mathbb{N}$ that satisfy $r(X)=r(X|A)+r(X/A)$, $r(X\oplus Y)=r(X)+r(Y)$ and $r(1_{\mathsf{S}})=0$:
		these appear as the parameters $r_j(S)$ in the exponents of~\cite[Theorem~1]{KMT}.
		Since our convention for monoids is multiplicative, we view such a morphism as a norm with values in the monoid $U=u^{\mathbb{N}}$ defined by $N(X)=u^{r(X)}$. More generally, a tuple $(r_1,\ldots,r_d)$ of such morphisms give rise to a norm with values in $U=u_1^{\mathbb{N}}\cdots u_d^{\mathbb{N}}$ defined by $N(X)=u_1^{r_1(X)}\cdots u_d^{r_d(X)}$. One can extend this linearly to get a norm with values in the polynomial ring $\KK[u_1,\ldots,u_d]$.
		\end{rem}

		\begin{rem}\label{rem:norms characters}
		If $\mathsf{S}$ is a multiplicative minors system then any norm $N$ automatically satisfies $N(X\oplus Y)=N(X)N(Y)$ for every $X\in\mathsf{S}[E]$, $Y\in\mathsf{S}[F]$. To see this use axiom (N1) for $A=E$ to get
		$$N(X\oplus Y)=N(X\oplus Y|\varnothing)\, N(X/E\oplus Y)$$ 
		and use axiom (N2) to get $N(X\oplus Y|\varnothing)=N(X)$ and $N(X/E\oplus Y)=N(Y)$.
		\end{rem}
		
		\begin{rem}
		In place of the constant set species $U$ one might wish to use any commutative monoid in set species. 
		This more general framework is well suited for treating combinatorial invariants whose set of variables depends on the ground set, such as the multivariate Tutte polynomial for matroids. 
		Since the general theory is essentially the same, we choose to stick to the constant case here and refer the interested reader to Section~\ref{sec:multivariate} for more details on the general case and remark~\ref{rem:Potts} for the matroid example.
		\end{rem}
		
		\begin{defi}
		Let $\mathsf{S}$ be a minors system, $R$ be a commutative $\KK$-algebra and $N:\KK\mathsf{S}\rightarrow R$ be a norm. The \emph{inverse norm} of $N$ is the norm $\overline{N}:\KK\mathsf{S}\rightarrow R$ defined for $X\in\mathsf{S}[E]$ by $\overline{N}(X)\doteq(-1)^{|E|}N(X)$.
		\end{defi}
		
		The terminology is justified by the following proposition. Let us recall from \cite[1.2.4]{aguiarmahajanbook} that for $(\mathsf{C},\Delta,\varepsilon)$ a comonoid in linear species, $(\mathsf{A},\mu,\eta)$ a monoid in linear species, and two morphisms of linear species $f,g:\mathsf{C}\rightarrow \mathsf{A}$, the \emph{convolution} of $f$ and $g$ is the morphism of linear species $f*g:\mathsf{C}\rightarrow \mathsf{A}$ defined as
		$$f*g\doteq \mu\circ (f\cdot g)\circ \Delta\ .$$
		This gives the space of morphisms of linear species from $\mathsf{C}$ to $\mathsf{A}$ the structure of an associative algebra whose unit element is the composition $\upsilon\doteq\eta\circ\varepsilon$. After applying the Fock functor, this is nothing but the usual convolution.
		
		\begin{prop}\label{prop:inverse convolution}
		The norm $\overline{N}$ is the inverse of $N$ for the convolution: we have $N*\overline{N}=\overline{N}*N=\upsilon$.
		\end{prop}
		
		\begin{proof}
		For $X\in\mathsf{S}[\varnothing]$ we have $\Delta_{\varnothing,\varnothing}(X)=X\otimes X$ and thus $(\overline{N}*N)(X)=\overline{N}(X)N(X)=1$. For $X\in\mathsf{S}[E]$ with $E\neq\varnothing$ one computes 
		$$(\overline{N}*N)(X)=\sum_{A\subseteq E}(-1)^{|A|}N(X|A)N(X/A) =\left(\sum_{A\subseteq E}(-1)^{|A|}\right)N(X) = 0 \ ,$$
		which completes the proof of $\overline{N}*N=\upsilon$. 
		The proof of $N*\overline{N}=\upsilon$ is similar, and indeed follows by exchanging the roles of $N$ and $\overline{N}$.
		\end{proof}

	\subsection{The Grothendieck monoid}

		\begin{defi}\label{def:U}
		For a minors system $\mathsf{S}$, we define its \emph{Grothendieck monoid} $U(\mathsf{S})$ to be the commutative monoid having generators $[X]$ for all isomorphism classes $X\in\mathsf{S}_\bullet$, subject to the relations:
		\begin{enumerate}
		\item $[X]=[X|A][X/A]$ for every $X\in\mathsf{S}[E]$ and $A\subseteq E$.
		\item $[X\oplus Y] = [X]$ for every $X\in \mathsf{S}[E]$ and $Y\in\mathsf{S}[\varnothing]$.
		\item $[1_{\mathsf{S}}]=1$.
		\end{enumerate}
		\end{defi}

		\begin{defi}
		The morphism $\mathsf{S}\rightarrow U(\mathsf{S}) \; , \; X\mapsto [X]$ is called the \emph{universal norm} for the minors system $\mathsf{S}$.
		\end{defi}
		
		The universal norm is universal in the sense that any norm factors uniquely through it. 
		In other words, the datum of a norm for $\mathsf{S}$ with values in $U$ is equivalent to that of a morphism of monoids $U(\mathsf{S})\rightarrow U$.
		
		\medskip
		
		In the next proposition we start simplifying the presentation of the Grothendieck monoid. For $X\in\mathsf{S}[E]$ and $\sigma:\{1,\ldots,n\}\stackrel{\sim}{\rightarrow} E$ a linear order on $E$, we denote by $X^\sigma_i\in\mathsf{S}[\{\sigma(i)\}]$ the object obtained from $X$ by contracting $\sigma(1),\ldots,\sigma(i-1)$ and deleting $\sigma(i+1),\ldots,\sigma(n)$.
		
		\begin{prop}\label{prop:pres X(S) first}
		The Grothendieck monoid $U(\mathsf{S})$ is generated by the elements $[X]$ for $X\in\mathsf{S}_1$, with relations:
		\begin{enumerate}
		\item for $X\in\mathsf{S}[E]$ and two linear orders $\sigma,\sigma':\{1,\ldots,n\}\stackrel{\sim}{\rightarrow}E$,
		$$\prod_{i=1}^n \, [X^\sigma_i] = \prod_{i=1}^n \, [X^{\sigma'}_i]\ .$$
		\item $[X\oplus Y]=[X]$ for every $X\in \mathsf{S}_1$ and $Y\in\mathsf{S}_0$.
		\end{enumerate}
		\end{prop}
		
		\begin{proof}
		Let us denote by $V(\mathsf{S})$ the monoid defined by these generators and relations and show that it is isomorphic to $U(\mathsf{S})$. 
		\begin{enumerate}[--]
		\item Let $\varphi:U(\mathsf{S})\rightarrow V(\mathsf{S})$ be the morphism of monoids defined on the generators by $\varphi([X])=\prod_{i=1}^n[X^\sigma_i]$ for $X\in\mathsf{S}[E]$ and any choice of a linear order $\sigma:\{1,\ldots,n\}\stackrel{\sim}{\rightarrow}E$. For $X\in\mathsf{S}[E]$ and $A\subseteq E$ of cardinality $r$, let us choose a linear order $\sigma:\{1,\ldots,n\}\stackrel{\sim}{\rightarrow} E$ such that $\sigma(\{1,\ldots,r\})=A$. Then one has 
		$$\prod_{i=1}^r[X^\sigma_i]=\varphi([X|A])) \;\; \mbox{ and } \;\; \prod_{i=r+1}^n[X^\sigma_i]=\varphi([X/A])\ .$$
		Thus, we have $\varphi([X])=\varphi([X|A])\varphi([X/A])$ and $\varphi$ is compatible with relation (1) defining $U(\mathsf{S})$. For $X\in\mathsf{S}[E]$ and $Y\in\mathsf{S}[\varnothing]$, for $\sigma:\{1,\ldots,n\}\stackrel{\sim}{\rightarrow}E$, and for every $i=1,\ldots,n$, we have $(X\oplus Y)^\sigma_i= X^\sigma_i\oplus Y$. This shows that $\varphi$ is compatible with relation (2) defining $U(\mathsf{S})$. It is obviously compatible with relation (3) as well, and thus well-defined.
		\item Let $\psi:V(\mathsf{S})\rightarrow U(\mathsf{S})$ be the morphism of monoids defined on the generators by $\psi([X])=[X]$ for $X\in\mathsf{S}_1$. Let $X\in\mathsf{S}[E]$ and choose a linear order $\sigma:\{1,\ldots,n\}\stackrel{\sim}{\rightarrow} E$. By using relation (1) defining $U(\mathsf{S})$, an easy induction on $n$ shows that we have an equality in $U(\mathsf{S})$: $\prod_{i=1}^n[X^\sigma_i]=[X]$. This proves that $\psi$ is compatible with relation (1) defining $V(\mathsf{S})$. Since $\psi$ is obiously compatible with relation $(2)$ as well, it is well-defined. We also proved that we have $\psi(\prod_{i=1}^n [X^\sigma_i])=[X]$.
		\item It is now clear that $\psi\circ\varphi$ is the identity of $U(\mathsf{S})$ and that $\varphi\circ\psi$ is the identity of $V(\mathsf{S})$, which completes the proof.\qedhere
		\end{enumerate}
		\end{proof}
		
		We refine Proposition \ref{prop:pres X(S) first} further and show that the relations (1) in degree~$2$ are enough. 
		
		\begin{thm}\label{thm:pres X(S) second}
		The Grothendieck monoid $U(\mathsf{S})$ is generated by the classes $[X]$ for $X\in\mathsf{S}_1$, with relations:
		\begin{enumerate}
		\item $[X|e][X/e]=[X|f][X/f]$ for $X\in\mathsf{S}[\{e,f\}]$.
		\item $[X\oplus Y]=[X]$ for every $X\in \mathsf{S}_1$ and $Y\in\mathsf{S}_0$.
		\end{enumerate}
		\end{thm}
		
		\begin{proof}
		Let us show that the relations (1) from Proposition \ref{prop:pres X(S) first} are consequences of the case $n=2$. 
		Since the group of permutations of $\{1,\ldots,n\}$ is generated by the transpositions $\tau_r=(r \, ,\, r+1)$ for $r=1,\ldots,n-1$, 
		it is enough to prove that for $X\in\mathsf{S}[E]$ and for any linear order $\sigma:\{1,\ldots,n\}\stackrel{\sim}{\rightarrow} E$ and any $r=1,\ldots,n-1$ we have 
		$$\prod_{i=1}^n \, [X^\sigma_i] = \prod_{i=1}^n \, [X^{\sigma \tau_r}_i] \ .$$
		Since $X^\sigma_i=X^{\sigma\tau_r}_i$ for $i\notin\{r, r+1\}$, it is enough to prove that 
		$$[X^\sigma_r][X^\sigma_{r+1}] = [X^{\sigma\tau_r}_r][X^{\sigma\tau_r}_{r+1}]\ .$$
		Let $X'\in\mathsf{S}[\{\sigma(r),\sigma(r+1)\}]$ be obtained from $X$ by contracting $\sigma(1),\ldots,\sigma(r-1)$ and deleting $\sigma(r+2),\ldots,\sigma(n)$. The above equality reads
		$$[X'|\sigma(r)] [X'/\sigma(r)] = [X'|\sigma(r+1)][X'/\sigma(r+1)] \ .$$
		This is a special case of relation (1) from the statement of the theorem, and the proof is complete.
		\end{proof}
		
		Theorem \ref{thm:pres X(S) second} shows that the Grothendieck monoid is a rather crude invariant of a minors system, in the sense that it only sees structures with ground sets of cardinality $\leq 2$. 
		However, it is the main ingredient in our imminent definition of the universal Tutte character,
		in which it serves as our replacement for the technology of \enquote{uniform selectors} of \cite{KMT}.
		Indeed, the definition of ``uniform'' in [\emph{loc.\;cit.}]\ says, in light of Proposition \ref{prop:pres X(S) first}(1),
 		that a uniform selector is exactly the restriction of a norm (valued in their target ring $\mathbb K[\{x_j\}_{j\in J}]$)
 		to structures with singleton ground sets.
				
		\begin{ex}
		The Grothendieck monoid of the minors system $\mathsf{Set}$ is the free monoid on one generator corresponding to a set with one element, i.e.\ we have an isomorphism $U(\mathsf{Set})\simeq u^{\mathbb{N}}$ which maps a finite set $E$ to the monomial $u^{|E|}$. As we will see in Section \ref{sec:matroids}, the  minors systems $\mathsf{Mat}$ and $\mathsf{Gra}$ have the same Grothendieck monoid, which is the free commutative monoid on two generators, namely the classes of a loop and a coloop.
		\end{ex}

	\subsection{Tutte characters}
	
		So far our constructions haven't taken advantage of the fact that our minors systems are not necessarily connected. The next ingredient is the key.
		
		\begin{defi}\label{def:twist}
		Let $\mathsf{S}$ be a minors system. A \emph{twist map} for $\mathsf{S}$ with values in a commutative monoid $U$ is a morphism of monoids $\tau:\mathsf{S}[\varnothing]\rightarrow U$.
		\end{defi}
		
		\begin{ex}
		A graph with an empty set of edges consists of a finite number of isolated vertices. A natural twist map for the minors system $\mathsf{Gra}$ is thus the morphism of monoids $\tau:\mathsf{Gra}[\varnothing]\rightarrow a^{\mathbb{N}}$ which maps the graph with $k$ isolated vertices to the monomial $a^k$. This example will be developed in more detail in Section \ref{ssec:graphs}.
		\end{ex}
		
		If $U$ is the multiplicative monoid of a $\KK$-algebra $R$, then we still denote by $\tau:\KK\mathsf{S}\rightarrow R$ the morphism of linear species extending $\tau$ linearly to $\KK\mathsf{S}[\varnothing]$ and satisfying $\tau(X)=0$ for $X\in\mathsf{S}[E]$, $E\neq \varnothing$, and still call this extension a twist map.
		
		\begin{defi}
		Let $R$ be a commutative $\KK$-algebra, $N_1,N_2:\KK\mathsf{S}\rightarrow R$ be two norms, and $\tau:\KK\mathsf{S}\rightarrow R$ be a twist map. The \emph{Tutte character} associated to the triple $(N_1,\tau,N_2)$ is the convolution product
		$$T_{N_1,\tau,N_2}\doteq N_1*\tau*N_2:\KK\mathsf{S} \rightarrow R\ .$$ 
		In other words, it is defined, for $X\in\mathsf{S}[E]$, by the formula
		$$T_{N_1,\tau,N_2}(X)=\sum_{A\subseteq E} N_1(X|A)\, \tau(X|A/A)\, N_2(X/A)\ .$$
		\end{defi}
		
		\begin{rem}
		For $X\in\mathsf{S}[\varnothing]$ we have $T_{N_1,\tau,N_2}(X)=\tau(X)$. For $X\in\mathsf{S}[E]$ and $Y\in\mathsf{S}[\varnothing]$ one easily shows that we have $T_{N_1,\tau,N_2}(X\oplus Y)=T_{N_1,\tau,N_2}(X)\tau(Y)$. 
		\end{rem}
		
		\begin{rem}
		If $\mathsf{S}$ is a multiplicative minors system then by Remark \ref{rem:norms characters} the norms $N_1$ and $N_2$ are compatible with the direct sum operation. Since $\tau$ is compatible with the direct sum operation, it follows that the convolution $T_{N_1,\tau,N_2}$ is a morphism of monoids in linear species: $T_{N_1,\tau,N_2}(X\oplus Y)=T_{N_1,\tau,N_2}(X)\,T_{N_1,\tau,N_2}(Y)$. This justifies our choice of the terminology \enquote{character}.
		\end{rem}
		
		\begin{rem}
		If the minors system $\mathsf{S}$ is connected then we necessarily have $\tau=\upsilon$ and the Tutte character is simply the convolution product $N_1*N_2$. When $N_1$ and $N_2$ are defined as in Remark \ref{rem:rank functions}, these are the characters that appear in \cite[Theorem 1]{KMT}.
		\end{rem}
		
		The Tutte character can be computed recursively thanks to the following deletion-contraction formula, with the base case $T_{N_1,\tau,N_2}(X)=\tau(X)$ for $X\in\mathsf{S}[\varnothing]$.
		
		\begin{prop}\label{prop: deletion contraction general}
		For every object $X\in\mathsf{S}[E]$ and every $e\in E$ we have the deletion-contraction recurrence formula:
		$$T_{N_1,\tau,N_2}(X)=N_1(X\backslash e^\c)\, T_{N_1,\tau,N_2}(X/e) + N_2(X/e^\c) \, T_{N_1,\tau,N_2}(X\backslash e)\ .$$
		\end{prop}
		
		\begin{proof}
		We compute
		$$T_{N_1,\tau,N_2}(X) = \sum_{\substack{A\subseteq E\\ e\in A}} N_1(X|A)\tau(X|A/A)N_2(X/A) + \sum_{\substack{A\subseteq E\\ e\notin A}} N_1(X|A)\tau(X|A/A)N_2(X/A)\ .$$
		Using the fact that $N_1$ is a norm we can rewrite the first sum as a sum over $A'=A\setminus e\subseteq E\setminus e$, of terms 
		\begin{eqnarray*}
		&& N_1(X| A'\cup e) \,\tau(X|A'\cup e/A'\cup e) \,N_2(X/A'\cup e) \\
		 & = & N_1(X|A'\cup e|e)\,N_1(X|A'\cup e/e)\,\tau(X/e|A'/A')\,{N_2}(X/e/A')\\
		& = & N_1(X| e)\,N_1(X/e|A')\,\tau(X/e|A'/A')\,N_2(X/e/A')\ .
		\end{eqnarray*}
		Thus, the first sum equals $N_1(X\backslash e^\c)\, T_{N_1,\tau,N_2}(X/e)$. The second sum is treated in the same way.
		\end{proof}
		
	\subsection{The universal Tutte character}
		
		\begin{defi}\label{def:universal Tutte}
		The \emph{universal Tutte character} of the minors system $\mathsf{S}$ is the Tutte character with values in the monoid ring $\KK[U(\mathsf{S})\times\mathsf{S}[\varnothing]\times U(\mathsf{S})]$ associated to $N_1$ (resp.\ $N_2$) the universal norm on the first (resp.\ last) factor of $U(\mathsf{S})\times\mathsf{S}[\varnothing]\times U(\mathsf{S})$, and $\tau$ the twist map corresponding to the embedding of $\mathsf{S}[\varnothing]$ as the middle factor. It is denoted by 
		$$T^{\mathsf{S}}: \KK\mathsf{S} \rightarrow \KK\left[U(\mathsf{S})\times\mathsf{S}[\varnothing]\times U(\mathsf{S})\right]\ .$$
		\end{defi}
				
		It is universal in the following sense.		
		
		\begin{prop}\label{prop:universal}
		Let $R$ be a commutative $\KK$-algebra and $\Phi:\KK\mathsf{S}\rightarrow R$ be a morphism of linear species. Assume that $\Phi:\KK\mathsf{S}[\varnothing]\rightarrow R$ is a morphism of $\KK$-algebras and that $\Phi$ satisfies a deletion-contraction recurrence formula
		$$\Phi(X)=N_1(X\backslash e^\c)\, \Phi(X/e) + N_2(X/e^\c) \, \Phi(X\backslash e)$$
		for every $X\in\mathsf{S}[E]$ and $e\in E$, where $N_1,N_2:\KK\mathsf{S}\rightarrow R$ are two norms. 
		There then exists a morphism of $\KK$-algebras $\overline{\Phi}:\KK\left[U(\mathsf{S})\times\mathsf{S}[\varnothing]\times U(\mathsf{S})\right] \rightarrow R$ such that $\Phi=\overline{\Phi}\circ T^{\mathsf{S}}$.
		\end{prop}	
		
		\begin{proof}
		Let us define $\overline{\Phi}$ by the formula
		$$\overline{\Phi}([X_1],Y,[X_2])=N_1(X_1)\,\Phi(Y)\,N_2(X_2)\ .$$
		By applying $\overline{\Phi}$ to the deletion-contraction formula for $T^{\mathsf{S}}$ one gets the deletion-contraction formula
		$$(\overline{\Phi}\circ T^{\mathsf{S}})(X)=N_1(X\backslash e^\c)\, (\overline{\Phi}\circ T^{\mathsf{S}})(X/e) + N_2(X/e^\c) \, (\overline{\Phi}\circ T^{\mathsf{S}})(X\backslash e)\ .$$
		Besides, for $Y\in\mathsf{S}[\varnothing]$ we have $\overline{\Phi}\circ T^{\mathsf{S}}(Y)=\Phi(Y)$, thus an easy induction on the cardinality of $E$ proves that $\overline{\Phi}\circ T^{\mathsf{S}}(X)=\Phi(X)$ for every $X\in\mathsf{S}[E]$.
		\end{proof}			
		
		\begin{ex}
		The universal Tutte character for the minors system $\mathsf{Set}$ has values in the ring $\KK[u_1,u_2]$ and is defined, for a finite set $E$, by 
		$$T^{\mathsf{Set}}(E)= \sum_{A\subseteq E} u_1^{|A|}u_2^{|E\setminus A|} = (u_1+u_2)^{|E|}\ .$$
		The deletion-contraction recurrence formula reads:
		$$(u_1+u_2)^{|E|}=u_1(u_1+u_2)^{|E\setminus\{e\}|} + u_2(u_1+u_2)^{|E\setminus\{e\}|}\ .$$
		
		As we will see in Section 5, the universal Tutte character for the minors systems $\mathsf{Mat}$ and $ \mathsf{Gra}$ is a polynomial which specializes to the classical Tutte polynomial.
		\end{ex}

	\subsection{The convolution formula}
	
		An important feature of Tutte characters is the following general convolution formula.
		
		\begin{thm}\label{thm:general convolution formula}
		Let $R$ be a commutative $\KK$-algebra, $N_0,N_1,N_2:\KK\mathsf{S}\rightarrow R$ be three norms and $\tau_1,\tau_2:\KK\mathsf{S}\rightarrow R$ be two twist maps. Then we have the convolution formula between Tutte characters:
		$$T_{\overline{N_0},\tau_1\tau_2,N_2}=T_{\overline{N_0},\tau_1,N_1}*T_{\overline{N_1},\tau_2,N_2} \ .$$
		In other words, for any $X\in\mathsf{S}[E]$ we have
		$$T_{\overline{N_0},\tau_1\tau_2,N_2}(X)= \sum_{A\subseteq E} T_{\overline{N_0},\tau_1,N_1}(X|A)\, T_{\overline{N_1},\tau_2,N_2}(X/A) \ .$$
		\end{thm}

		\begin{proof}
		One simply computes, using Proposition \ref{prop:inverse convolution}:
		\begin{eqnarray*}
		T_{\overline{N_0},\tau_1,N_1}*T_{\overline{N_1},\tau_2,N_2} & = & \overline{N_0} * \tau_1 *(N_1 * \overline{N_1}) * \tau_2 *N_2 \\
		& = &\overline{N_0} * (\tau_1 *\tau_2) * N_2\ .
		\end{eqnarray*}
		Now $\tau_1*\tau_2$ is simply the twist map corresponding to the pointwise product of $\tau_1$ and $\tau_2$, hence the result.
		\end{proof}
		
		If $N_0,N_1,N_2$ are three copies of the universal norm $\mathsf{S}\rightarrow U(\mathsf{S})$ and $\tau_1,\tau_2$ are two copies of the universal twist map $\mathsf{S}[\varnothing]\stackrel{\mathrm{id}}{\rightarrow} \mathsf{S}[\varnothing]$ then we get a \emph{universal convolution formula} with values in the algebra 
		$$\KK[U(\mathsf{S})\times \mathsf{S}[\varnothing]\times U(\mathsf{S})\times \mathsf{S}[\varnothing] \times U(\mathsf{S})]\ .$$
		 Every convolution formula in the sense of Theorem \ref{thm:general convolution formula} appears as a specialization of the universal convolution formula.
		
		\begin{rem}\label{rem:inverse U}
		The inversion $N\mapsto \overline{N}$ manifests itself as the automorphism of $\KK[U(\mathsf{S})]$ that sends a generator $[X]$, for $X\in\mathsf{S}[E]$, to $(-1)^{|E|}[X]$, or equivalently that sends a generator $[X]$, for $X\in\mathsf{S}_1$, to $-[X]$.
		\end{rem}
		
		\begin{rem}\label{rem:multiconvolution}
		One can iterate the convolution formula from Theorem \ref{thm:general convolution formula} and get the following \emph{iterated convolution formulae}, for a choice of norms $N_0,\ldots,N_n$ and twist maps $\tau_1,\ldots,\tau_n$:
		$$T_{\overline{N_0},\tau_1\cdots \tau_n,N_n} = T_{\overline{N_0},\tau_1,N_1} * T_{\overline{N_1},\tau_2,N_2} * \cdots * T_{\overline{N_{n-1}},\tau_n,N_n}\ .$$
		More explicitly, for $X\in\mathsf{S}[E]$:
		$$T_{\overline{N_0},\tau_1\cdots \tau_n,N_n}(X)=\sum_{\varnothing=A_0\subseteq A_1\subseteq \cdots \subseteq A_n=E} \left(\prod_{i=1}^n \,T_{\overline{N_{i-1}},\tau_i,N_i}(X|A_i/A_{i-1})\right)\ .$$
		\end{rem}
		
		\begin{ex}
		The universal convolution formula for the minors system $\mathsf{Set}$ lives in the ring $\KK[u_0,u_1,u_2]$ and reads:
		$$(-u_0+u_2)^{|E|}=\sum_{A\subseteq E}(-u_0+u_1)^{|A|}(-u_1+u_2)^{|E\setminus A|}\ .$$
		\end{ex}

\section{More on norms and Tutte characters}\label{sec:more}
	In this section we gather some longer remarks on and extensions of the foregoing theory.
	It can be skipped on a first reading.
	
	\subsection{Reduced Tutte characters}\label{ssec:reduced}
	
		One notices in the example of $\mathsf{Set}$ that the universal Tutte invariant $T^{\mathsf{Set}}(E)\in\KK[u_1,u_2]$ is a homogeneous polynomial of degree $|E|$. This is a general phenomenon that is easily stated in terms of rings graded by a monoid. 
		
		Let $U$ be a commutative monoid.  
		A \emph{$U$-graded commutative ring} is a commutative ring $R$ with a direct sum decomposition $R=\bigoplus_{x\in U}R_x$ (called a \emph{$U$-grading}\/)
		such that $R_xR_y\subseteq R_{xy}$ for any elements $x,y\in U$. 
		If $U=u^{\mathbb{N}}$ then a $U$-graded commutative ring is nothing but a non-negatively graded commutative ring $R=\bigoplus_{n\in\mathbb{N}} R_n$, and if $U=u_1^{\mathbb{N}}\cdots u_d^{\mathbb{N}}$ then a $U$-graded commutative ring is nothing but a non-negatively $d$-graded commutative ring $R=\bigoplus_{(n_1,\ldots,n_d)\in\mathbb{N}^d} R_{n_1,\ldots,n_d}$. 
		
		We give the ring $\KK[U(\mathsf{S})\times\mathsf{S}[\varnothing]\times U(\mathsf{S})]$ the $U(\mathsf{S})$-grading
		where the degree of a basis element $([X_1],Y,[X_2])$ is the product $[X_1][X_2]\in U(\mathsf{S})$. 
		
		\begin{prop}\label{prop:homogeneous}
		For $\mathsf{S}$ a minors system and $X\in\mathsf{S}[E]$, the universal Tutte character $T^{\mathsf{S}}(X)\in\KK[U(\mathsf{S})\times\mathsf{S}[\varnothing]\times U(\mathsf{S})]$ is $U(\mathsf{S})$-homogeneous of degree $[X]$.
		\end{prop}				
		
		\begin{proof}
		Every monomial $([X|A],[X|A/A],[X/A])$ has degree $[X|A][X/A]=[X]$ by definition of the Grothendieck ring.
		\end{proof}
		
		This result says, roughly speaking, that there are twice as many variables coming from $U(\mathsf{S})$ in the universal Tutte character $T^{\mathsf S}$ as there ``should'' be,
		and that the information in $T^{\mathsf S}$ can be wholly recovered from an invariant where half of these variables are specialized away.
		As we will now explain, the ``unnecessary'' twins of the variables exist to allow for prefactors,
		which are ubiquitous when working with deletion-contraction recurrences.
		
		For concreteness, to make the explanation easier, let us work after applying a norm $U(\mathsf{S})\rightarrow u_1^{\mathbb{N}}\cdots u_d^{\mathbb{N}}$.
		To choose such a norm amounts to choosing maps $r_1,\ldots,r_d:\mathsf{S}\rightarrow \mathbb{N}$ as in Remark \ref{rem:rank functions}. 
		We also assume for simplicity that we are applying the trivial twist map $S[\varnothing]\rightarrow \{*\}$. Then the image of the universal Tutte invariant is an element $T(X)\in\KK[u_{i,1},u_{i,2} \; , \; i=1,\ldots, d]$.
		This polynomial ring is $d$-graded, where the degree of $u_{i,1}$ and $u_{i,2}$ is the $i$\/th standard basis vector $e_i$. 
		What the above proposition is saying is that $T(X)$ is $d$-homogeneous of degree $(r_1(X),\ldots,r_d(X))$. 
		
		For a choice $\varepsilon=(\varepsilon_1,\ldots,\varepsilon_d)\in\{1,2\}^n$, one can define a \emph{reduced} Tutte character 
		$$T_\varepsilon:\KK\mathsf{S}\rightarrow\KK[u_1,\ldots,u_d]$$ 
		by setting $u_{i,\varepsilon_i}$ to $u_i$ and $u_{i,3-\varepsilon_i}$ to $1$. The homogeneity property means that one can reconstruct $T(X)$ from $T_\varepsilon(X)$ up to a prefactor:
		$$T(X)(u_{1,1},\ldots,u_{d,2})
		= u_{1,3-\varepsilon_1}^{r_1(X)}\ldots u_{d,3-\varepsilon_d}^{r_d(X)} \cdot T_\varepsilon(X){\left(\frac{u_{1,\varepsilon_1}}{u_{1,3-\varepsilon_1}} ,\ldots, \frac{u_{d,\varepsilon_d}}{u_{d,3-\varepsilon_d}}\right)}\ .$$
		Among these $2^d$ natural choices of reduced Tutte characters, there does not seem to be a reason to declare any one of them canonical in general, 
		though in certain special cases some choices are more relevant than others. 
		For instance, in the case of matroids, where $d=2$ (see Section \ref{sec:matroids} for more details),
		the classical corank-nullity polynomial is a preferred choice;
		this is because after a translation of the variables it becomes the Tutte polynomial, whose coefficients remain non-negative
		and gain a new combinatorial interpretation in terms of basis activities.
		
		The disadvantage to working with reduced Tutte characters is that it leads to formulae  which still contain many prefactors, and the source of these becomes more obscure. 
		For instance, it is easier to write the convolution formulae  in the unreduced version before specializing to the reduced version.
	
	\subsection{Functoriality and duality}\label{sec:functoriality duality}
	
		As already observed in \cite[Theorems 2 and 17]{KMT}, one can explain certain identities between Tutte characters by morphisms between minors systems and considerations of duality.
		
		\begin{defi}
		Let $\mathsf{S}$ and $\mathsf{S}'$ be two minors systems (resp.\ multiplicative minors systems). A \emph{morphism of minors systems} (resp.\ a \emph{morphism of multiplicative minors systems}) from $\mathsf{S}$ to $\mathsf{S}'$ is a morphism of set species $f:\mathsf{S}\rightarrow \mathsf{S}'$ that is compatible with the maps $\Delta_{A,B}$ and $\rho_E$ (resp.\ with the maps $\Delta_{A,B}$ and $\mu_{A,B}$) and sends $1_{\mathsf{S}}$ to~$1_{\mathsf{S}'}$.
		\end{defi}
		
		This makes minors systems into a category whose final object is the minors system $\mathsf{Set}$.\medskip
		
		A morphism $f:\mathsf{S}\rightarrow\mathsf{S}'$ between two minors systems induces a morphism of comonoids in linear species $\KK\mathsf{S}\rightarrow \KK\mathsf{S}'$, which is a morphism of bimonoids in linear species in the multiplicative context. Every norm $N'$ for $\mathsf{S}'$ induces a norm $N=N'\circ f$ for $\mathsf{S}$ and this is reflected in a morphism of monoids between the Grothendieck monoids $U(\mathsf{S})\rightarrow U(\mathsf{S}')$. We then have an induced map $f:\KK[U(\mathsf{S})\times \mathsf{S}[\varnothing]\times U(\mathsf{S})] \rightarrow \KK[U(\mathsf{S}')\times \mathsf{S}'[\varnothing]\times U(\mathsf{S}')]$ and the universal Tutte characters satisfy the functoriality identity:
		$$T^{\mathsf{S}'}\circ f = f\circ T^{\mathsf{S}}\ .$$
		
		For $N'_1,N'_2$ two norms for $\mathsf{S}'$ and $\tau'$ a twist map, this gives an equality of Tutte characters
		$$T_{N'_1,\tau',N'_2} \circ f = T_{N'_1\circ f,\tau'\circ f,N'_2\circ f}\ .$$
	
		\begin{defi}
		For a minors system $\mathsf{S}$, the \emph{opposite} minors system $\mathsf{S}^{\mathrm{op}}$ has the same underlying set species and has structural maps defined in the following way.
		\begin{enumerate}
		\item The coproduct map $\Delta^{\mathrm{op}}_{A,B}$ is the composition of $\Delta_{B,A}$ and the exchange of the factors $\mathsf{S}[A]$ and $\mathsf{S}[B]$.
		\item The map $\rho^{\mathrm{op}}_E$ equals the map $\rho_E$.
		\end{enumerate}
		\end{defi}
		
		The comonoid in linear species $\KK\mathsf{S}^{\mathrm{op}}$ has the same underlying linear species as that of $\mathsf{S}$, with the opposite comonoid structure (i.e., exchange the two sides of the tensor product). The Grothendieck monoid of $\mathsf{S}^{\mathrm{op}}$ is naturally isomorphic to that of $\mathsf{S}$. The universal Tutte character of $\mathsf{S}^{\mathrm{op}}$ is obtained from that of $\mathsf{S}$ by exchanging the two factors $U(\mathsf{S})$ in $\KK[U(\mathsf{S})\times \mathsf{S}[\varnothing]\times U(\mathsf{S})]$. This explains duality properties for Tutte characters when there is an isomorphism $\mathsf{S}\stackrel{\sim}{\rightarrow} \mathsf{S}^{\mathrm{op}}$, e.g.\ in the case of $\mathsf{S}=\mathsf{Mat}$ (see Section \ref{sec:matroids} for more details).
		
	\subsection{Multivariate Tutte characters}\label{sec:multivariate}
	
		In our formalism, norms and Tutte characters have values in $\KK$-algebras, i.e.\ constant monoids in linear species. It is natural to extend this to general (commutative) monoids in linear species. 
		In the present work, rather than developing a general theory, we explain only how one can build multivariate Tutte characters from norms. 
		For simplicity we only treat the case of trivial twist maps. The motivating example is that of the multivariate Tutte polynomial for matroids: see remark~\ref{rem:Potts} below.\medskip
	
	    For $V_1$ and $V_2$ two $\KK$-modules, viewed as constant linear species, the Cauchy tensor product $V_1\cdot V_2$ is no longer constant and is given, for a finite set $E$, by
	    $$(V_1\cdot V_2)[E] = \bigoplus_{E=A\sqcup B} V_1\otimes V_2\ .$$
	    In particular, a vector in $(V_1\cdot V_2)[E]$ remembers the datum of a decomposition $E=A\sqcup B$. For notational purposes, we artificially keep track of this datum by introducing dummy variables $\alpha_{e,1}$ and $\alpha_{e,2}$, for $e\in E$, and define a linear species $(V_1\cdot V_2)'$ given by
	    $$(V_1\cdot V_2)'[E]=(V_1\otimes V_2)[\{\alpha_{e,1},\alpha_{e,2} \; : \; e\in E\}]\ .$$
	    There is an embedding $(V_1\cdot V_2)\hookrightarrow (V_1\cdot V_2)'$ defined on the summand indexed by a decomposition $E=A\sqcup B$ by
	    $$v_1\otimes v_2 \mapsto (v_1\otimes v_2)\left(\prod_{e\in A}\alpha_{e,1} \right)\left(\prod_{e\in B}\alpha_{e,2}\right)\ .$$
	    
	    Let $\mathsf{S}$ be a minors system, $R_1$ and $R_2$ be two commutative $\KK$-algebras, and $N_1:\KK\mathsf{S}\rightarrow R_1$, $N_2:\KK\mathsf{S}\rightarrow R_2$ be two norms. We let $\widetilde{N}_1:\KK\mathsf{S}\rightarrow R_1\cdot R_2$ and $\widetilde{N}_2:\KK\mathsf{S}\rightarrow R_1\cdot R_2$ denote the norms that they induce. The \emph{multivariate Tutte character} associated to $N_1$ and $N_2$ is then the convolution product
	    $$\widetilde{T}_{N_1,N_2} \doteq \widetilde{N_1} * \widetilde{N_2} \, : \KK\mathsf{S}\rightarrow R_1\cdot R_2\ .$$
	    
	    If we view it as a morphism of linear species $\KK\mathsf{S}\rightarrow (R_1\cdot R_2)'$ then it is simply given, for $X\in\mathsf{S}[E]$, by the formula
	    $$\widetilde{T}_{N_1,N_2}(X) = \sum_{A\subseteq E} (N_1(X|A) \otimes N_2(X/A))\left(\prod_{e\in A}\alpha_{e,1}\right)\left(\prod_{e\notin A}\alpha_{e,2}\right)\ .$$
		
		These multivariate Tutte characters satisfy a deletion-contraction formula and convolution formulae  similar to the constant case, for which the proofs are exactly the same as those above.
		
		The observations of Section~\ref{ssec:reduced} apply here just as in the constant case.
		That is, half of the variables $\alpha_{e,i}$ are redundant and one does not lose information by setting $\alpha_{e,1}=\alpha_e$ and $\alpha_{e,2}=1$.
		
	\subsection{More general deletion-contraction recurrences}
	
		Let $\mathsf{S}$ be a minors system. In Theorem \ref{thm:general convolution formula} we proved that the universal Tutte character $T^{\mathsf{S}}$ is universal with respect to morphisms $\Phi:\KK\mathsf{S}\rightarrow R$ which are multiplicative on $\mathsf{S}[\varnothing]$ and satisfy a deletion-contraction recurrence formula of the form
		\begin{equation}\label{eq:delcont recurrence}
		\Phi(X)=N_1(X\backslash e^\c)\,\Phi(X/e) + N_2(X/e^\c)\,\Phi(X\backslash e)\ ,
		\end{equation}
		where $N_1,N_2:\KK\mathsf{S}\rightarrow R$ are norms. 
		
		If one wishes to make a general study of deletion-contraction recurrence formulae,
		one might not wish to adopt \emph{a priori} the assumption
		that one's families of coefficients are norms, or indeed have any particular predefined structure.
		In Sections \ref{sec:OW} and~\ref{sec:BR} we will see two examples of previous research
		which started with such an expansive notion of the formulae of interest,
		but wound up discovering that, under some non-zero-divisor assumptions, 
		the coefficients had to satisfy exactly the restrictions that our norms impose on them.
		We regard this as a strong argument that restricting to norms in 
		investigating formula \eqref{eq:delcont recurrence} is a \emph{natural} thing to do.
\medskip
		
		To treat the general case,
		let $N_1,N_2:\KK\mathsf{S}_1\rightarrow R$ and $\tau:\KK\mathsf{S}_0\rightarrow R$ be any linear maps,
		i.e.\ linear extensions of arbitrary set functions 
		$\mathsf{S}_1\rightarrow R$ and $\mathsf{S}_0\rightarrow R$ respectively.
		Suppose that we want to define a morphism $\Phi:\KK\mathsf{S}\rightarrow R$ 
		which satisfies the recurrence formula \eqref{eq:delcont recurrence} with the base case $\Phi(X)=\tau(X)$ if $X\in\mathsf{S}[\varnothing]$. Such a $\Phi$ may be ill-defined because applying the recurrence formula in different orders may lead to different results. 
		We say that $\Phi$ is \emph{well-defined up to cardinality $n$} if the recurrence formula unambiguously defines $\Phi(X)$ for $X\in\mathsf{S}[E]$ whenever $|E|\leq n$. 
		Note that $\Phi$ is automatically well-defined up to cardinality~$1$.\medskip
		
		For $X\in\mathsf{S}[E]$ and $\{e,f\}\subseteq E$, we use the abbreviations $X_{e,f}=X\backslash\{e,f\}^\c$ and $X^{e,f}=X/\{e,f\}^\c$, which are both elements of $\mathsf{S}[\{e,f\}]$.
		
		\begin{prop}\label{prop:general recurrence}
		Assume that $\Phi$ is well-defined up to cardinality $n-1$. Then it is well-defined up to cardinality $n$ if and only if 
		for every choice of a finite set $E$ of cardinality $n$, two elements $e,f\in E$ and an element $X\in\mathsf{S}[E]$ we have
		\begin{equation}\label{eq:general recurrence}
		\begin{split}
		\big(N_1&(X_{e,f}|e)\,N_1(X_{e,f}/e)-N_1(X_{e,f}|f)\,N_1(X_{e,f}/f)\big)\,\Phi(X/e,f) \\
		& =\big(N_2(X^{e,f}|e)\,N_2(X^{e,f}/e)-N_2(X^{e,f}|f)\,N_2(X^{e,f}/f)\big)\,\Phi(X\backslash e,f)\ .
		\end{split}
		\end{equation}
		\end{prop}
		
		\begin{proof}
		Since $\Phi$ is well-defined up to cardinality $n-1$, it is well-defined up to cardinality $n$ if and only if for every finite set $E$ of cardinality $n$ and every $e,f\in E$ we have
		$$N_1(X\backslash e^\c)\,\Phi(X/e)+N_2(X/e^\c)\,\Phi(X\backslash e) = N_1(X\backslash f^\c)\,\Phi(X/f)+N_2(X/f^\c)\,\Phi(X\backslash f)\ .$$
		Using the recurrence formula, one may compute $\Phi(X/e)$, $\Phi(X\backslash e)$, $\Phi(X/f)$ and $\Phi(X\backslash f)$ as
		$$\Phi(X/e)=N_1(X/e|f)\,\Phi(X/e,f)+N_2(X/f^\c)\,\Phi(X/e\backslash f) \ ,$$
		$$\Phi(X\backslash e)=N_1(X|f)\,\Phi(X\backslash e/f)+N_2(X\backslash e/f^\c)\,\Phi(X\backslash e,f)\ ,$$
		$$\Phi(X/f)=N_1(X/f|e)\,\Phi(X/e,f)+N_2(X/e^\c)\,\Phi(X/f\backslash e) \ ,$$
		$$\Phi(X\backslash f)=N_1(X|e)\,\Phi(X\backslash f/e)+N_2(X\backslash f/e^\c)\,\Phi(X\backslash e,f)\ .$$
		By replacing in our first equation one sees that the terms involving $X/e\backslash f=X\backslash f/e$ and $X\backslash e/f=X/f\backslash e$ cancel. 
		The conclusion follows by rearranging the remaining terms.
		\end{proof}
		
		Of course, we recover the fact that if $N_1$ and $N_2$ are norms then $\Phi$ is well-defined.
		
	\subsection{Norms as exponentials}
	
		Following \cite{KMT} we show that norms arise as exponentials of certain linear maps.
		We will not use this fact in the rest of the article. 
		
		Assume that our coefficient ring $\KK$ contains $\mathbb{Q}$.  For $(\mathsf{C},\Delta,\varepsilon)$ a comonoid in linear species, $(\mathsf{A},\mu,\eta)$ a monoid in linear species and $f:\mathsf{C}\rightarrow \mathsf{A}$ a linear map, 
		we define the exponential of $f$ by the sum
		$$\exp_*(f)=\sum_{n\geq 0}\frac{1}{n!}f^{*n}\ ,$$
		where $f^{*n}$ is the iterated convolution product of $f$ with itself $n$ times and $f^{*0}=\upsilon=\eta\circ\varepsilon$. To make sense of this infinite sum we assume that $f$ vanishes on $\mathsf{C}[\varnothing]$. Then $\exp_*(f):\mathsf{C}\rightarrow \mathsf{A}$ is a well-defined morphism of linear species. If we further assume that $\mathsf{C}$ has the structure of a bimonoid in linear species and that $f$ is an infinitesimal character (i.e.\ satisfies $f\circ\mu=\mu\circ (\varepsilon\cdot f+f\cdot\varepsilon)$) then $\exp_*(f)$ is a morphism of monoids in linear species.
		
		\begin{prop}\label{prop:exponential}
		Assume that $\KK$ contains $\mathbb{Q}$. For a norm $N:\KK\mathsf{S}\rightarrow R$ with values in a commutative $\KK$-algebra $R$, we denote by $\nu:\KK\mathsf{S}\rightarrow R$ the linear map defined for $X\in\mathsf{S}[E]$ by
		$$
		\nu(X)=
		\begin{cases}
		N(X) & \textnormal{ if } |E|=1 ; \\
		0 & \textnormal{ otherwise. }
		\end{cases}
		$$
		Then we have 
		$$\exp_*(\nu)=N\ .$$
		\end{prop}
		
		\begin{proof}
		Since $\nu$ vanishes on $\KK\mathsf{S}[\varnothing]$, its exponential makes sense. For $X\in\mathsf{S}[\varnothing]$ we have $\exp_*(\nu)(X)=\upsilon(X)=1=N(X)$. Let $X\in\mathsf{S}[E]$ with $|E|=n>0$. For degree reasons the only term that survives in $\exp_*(\nu)(X)$ is
		$$\exp_*(\nu)(X)=\frac{1}{n!}\nu^{*n}(X)=\frac{1}{n!} \, \mu^{(n-1)}\circ\nu^{\cdot (n)}\circ \Delta^{(n-1)}(X)\ .$$
		The iterated coproduct $\Delta^{(n-1)}(X)$ is a sum of tensors, some of which contain one element of $\mathsf{S}[\varnothing]$; since $\nu$ vanishes on $\KK\mathsf{S}[\varnothing]$, these terms do not contribute. An easy induction on $n$ shows that the remaining terms are
		$$\overline{\Delta}^{(n-1)}(X) = \sum_{\sigma:\{1,\ldots,n\}\stackrel{\sim}{\rightarrow} E} \left(\bigotimes_{i=1}^n X^\sigma_i\right)\ .$$
		By using axiom (N1) we get, for every $\sigma$,
		$$N(X)=\prod_{i=1}^n N(X^\sigma_i) = \prod_{i=1}^n \nu(X^\sigma_i)\ .$$
		This implies that we have
		$$\exp_*(\nu)(X)=\frac{1}{n!}\sum_{\sigma:\{1,\ldots,n\}\stackrel{\sim}{\rightarrow} E} N(X)=N(X)\ ,$$
		which concludes the proof.
		\end{proof}
		
		We note that if $\mathsf{S}$ is a multiplicative minors system then $\nu$ is an infinitesimal character.\medskip

\section{Matroids and graphs}\label{sec:matroids}

	\subsection{The minors system of matroids}

		We start by recalling basic definitions on matroids in order to set some notation. For background on matroid theory we refer, e.g., to Oxley's textbook \cite{Oxley}.
		
		\begin{defi}\label{def:mat}
		A {\em matroid} is a pair $M=(E,\rk)$, where $E$ is a finite set (the \emph{ground set}) and $\rk : 2^E \to \mathbb N$ is a function (the \emph{rank function}) such that, for all $X,Y\subseteq E$,
	    \begin{itemize}
	    \item[(R1)] $\rk(X)\leq \vert X \vert$,
	    \item[(R2)] $X\subseteq Y$ implies $\rk(X)\leq \rk(Y)$,
	    \item[(R3)] $\rk(X\cup Y) + \rk(X\cap Y) \leq \rk(X) + \rk (Y)$.
	    \end{itemize}
		\end{defi}
	
		We write $(E,\rk)=(E(M),\rk_M)$ when the context is not clear. Matroids form a set species $\mathsf{Mat}$ for which $\mathsf{Mat}[E]$ is the set of matroids with ground set $E$.
		
		\medskip
		
		Given $A\subseteq E$, the \emph{restriction} of $M$ to $A$ is the matroid $M|A$ on the set $A$ whose rank function is the restriction of the rank function of $M$. Equivalently, the \emph{deletion} of $M$ by $A$ is the defined as the restriction to $A^\c=E\setminus A$.	
		The \emph{contraction} of $M$ by a subset $A\subseteq E$ is the matroid $M/A$ on the set $ E \setminus A$, with rank function $\overline\rk$ given by $\overline\rk(B) \doteq \rk(B \cup A) - \rk(A)$ for $B \subseteq  E \setminus A$.
		
		Recall that the \emph{direct sum} of two matroids $M=(E,\rk)$ and $M'=(E',\rk')$ is the matroid $M\oplus M'\doteq (E\sqcup E', \rk\oplus\rk')$, where for $A\subseteq E$ and $A'\subseteq E'$,
		$$(\rk\oplus\rk')(A\sqcup A') = \rk(A)+\rk'(A').$$
		
		\medskip
	
		These notions of restriction, contraction and direct sum endow the set species $\mathsf{Mat}$ with the structure of a multiplicative minors system. It is connected since there is only one matroid with empty ground set, given by $\rk(\varnothing)=0$.
			
		\begin{rem}
		The linearization $\KK\mathsf{Mat}$ is a connected Hopf monoid in linear species whose image by the Fock functor is the classical Hopf algebra of matroids introduced by Schmitt \cite{schmitt}.
		\end{rem}		
		
		The \emph{rank} (resp. \emph{corank}) of a matroid $M$ is defined as $\rk(M)=\rk(E(M))$ (resp. $\cork(M)=|E(M)|-\rk(E(M))$). The \emph{coloop} (resp. \emph{loop}) is the only matroid on a set of cardinality $1$ and rank $1$ (resp. rank $0$) and is denoted by the letter $c$ (resp. $l$). We say that an element $e\in E(M)$ is a coloop in $M$ (resp. a loop in $M$) if $\rk(M\backslash\{e\})=\rk(M)-1$ (resp. $\rk(M/\{e\})=\rk(M)$).
		
		\begin{rem}
		For every integer $n\geq 1$, the set species of matroids $M$ such that $|E(M)|\leq n$ and the set species of matroids $M$ such that $\rk(M)\leq n$ are examples of minors systems that are not multiplicative.
		Further examples may be generated by performing similar ``truncations'' on other minors systems we treat below.
		\end{rem}
		
	\subsection{The universal Tutte character and the Tutte polynomial}
		The computation of the Grothendieck monoid of $\mathsf{Mat}$ is easy.
		The computations in later sections will not be as easy,
		so we work out the present example slowly and pains\-tak\-ingly in order that it can serve as a template for the examples to come.
		
		\begin{prop}\label{prop:UMat}
		Let us denote by $u$ and $v$ the classes in $U(\mathsf{Mat})$ of a coloop and a loop, respectively. We have an isomorphism of monoids 
		$$U(\mathsf{Mat})\simeq u^{\mathbb{N}}v^{\mathbb{N}}$$
		which maps the class of a matroid $M$ to the monomial 
		$$u^{\rk(M)}v^{\cork(M)}\ .$$
		\end{prop}
		
		\begin{proof}
		We use the presentation of $U(\mathsf{Mat})$ given by Theorem \ref{thm:pres X(S) second}. The classes $u$ and $v$ generate $U(\mathsf{Mat})$ and the relations that they satisfy come from matroids on two-element sets. There are four such matroids: the direct sums $c\oplus c$, $c\oplus l$, $l\oplus l$, and the uniform matroid $U_{1,2}$, which give rise to the relations $u^2=u^2$, $uv=vu$, $v^2=v^2$, $uv=uv$, respectively. These relations are all trivial, which implies that we have an isomorphism $\varphi:u^{\mathbb{N}}v^{\mathbb{N}}\stackrel{\sim}{\rightarrow}U(\mathsf{Mat})$. For a matroid $M$ and a subset $A\subseteq E(M)$ one easily checks the identities $\rk(M)=\rk(M|A)+\rk(M/A)$ and $\cork(M)=\cork(M|A)+\cork(M/A)$, which imply that the morphism of monoids $\psi:U(\mathsf{Mat})\rightarrow u^{\mathbb{N}}v^{\mathbb{N}}$ defined by $\psi([M])=u^{\rk(M)}v^{\cork(M)}$ is well-defined. Since $\psi\circ\varphi(u)=u$ and $\psi\circ\varphi(v)=v$, the composite $\psi\circ\varphi$ is the identity of $u^{\mathbb{N}}v^{\mathbb{N}}$ and $\psi$ is the inverse of $\varphi$.
		\end{proof}

		\begin{prop}
		The universal Tutte character of the minors system $\mathsf{Mat}$ is the character
		$$T^{\mathsf{Mat}} : \KK\mathsf{Mat} \rightarrow \KK[u_1,v_1,u_2,v_2]$$
		defined, for a matroid $M$, by
		\begin{eqnarray*}
		T^{\mathsf{Mat}} (M) & = & \sum_{A\subseteq E(M)} u_1^{\rk(M|A)} \, v_1^{\cork(M|A)}\, u_2^{\rk(M/A)} \, v_2^{\cork(M/A)} \notag\\
		& = & \sum_{A\subseteq E(M)} u_1^{\rk(A)} \, v_1^{|A|-\rk(A)} \, u_2^{\rk(M)-\rk(A)} \, v_2^{|E(M)|-|A|-\rk(M)+\rk(A)} \ .\notag\\
		\end{eqnarray*}
		\end{prop}
	
		\begin{proof}
		This follows directly from the definition of the universal Tutte character and Proposition \ref{prop:UMat}.
		\end{proof}
		
		Let us view $\KK[u_1,v_1,u_2,v_2]$ as a bigraded ring, where $u_1$, $u_2$ have degree $(1,0)$ and $v_1$, $v_2$ have degree $(0,1)$. We remark, as a special case of the discussion of Section \ref{ssec:reduced}, that $T^{\mathsf{Mat}}(M)$ is a bihomogeneous polynomial of degree $(\rk(M),\cork(M))$. Out of the possible reduced Tutte characters, the most popular is the \emph{corank-nullity polynomial}, obtained by setting $u_1$ and $v_2$ to $1$. It is customary to shift the remaining two variables by $1$ to obtain the following classical invariant.
		
		\begin{defi}
		The \emph{Tutte polynomial} of a matroid $M$ is the bivariate polynomial
		$$\mathfrak{T}_M(x,y)=\sum_{A\subseteq E(M)} (x-1)^{\rk(M)-\rk(A)}(y-1)^{|A|-\rk(A)}\ .$$
		\end{defi}
		
		The Tutte polynomial $\mathfrak{T}_M(x,y)$ is obtained from $T^{\mathsf{Mat}}(M)$ by specializing the variables to 
		$$(u_1,v_1,u_2,v_2)=(1,y-1,x-1,1)\ .$$ 
		In other words, it is the Tutte character with values in the polynomial ring $\KK[x,y]$ associated to the norms $N_1(M)=(y-1)^{\cork(M)}$ and $N_2(M)=(x-1)^{\rk(M)}$. Conversely, the universal Tutte character can be recovered from the Tutte polynomial up to a pre-factor: 
		
		\begin{equation}\label{eq:TMat}
		T^{\mathsf{Mat}}(M)=u_1^{\rk(M)}v_2^{\cork(M)}\, \mathfrak{T}_M(1+\tfrac{u_2}{u_1},1+\tfrac{v_1}{v_2}) \ .
		\end{equation}

		Since $\rk(M)$ is the degree in $x$ of the polynomial $\mathfrak{T}_M(x,1)$ and $\cork(M)$ is the degree in $y$ of the polynomial $\mathfrak{T}_M(1,y)$, 
		even the pre-factor can be recovered from $\mathfrak{T}_M(x,y)$ and there really is no loss of information in this specialization.
		
		\begin{prop}
		The universal Tutte character for $\mathsf{Mat}$ satisfies the following deletion-contraction recurrence formula:
		$$T^{\mathsf{Mat}}(M)=u_1^{\rk(M\backslash e^\c)}v_1^{\cork(M\backslash e^\c)} T^{\mathsf{Mat}}(M/e) + u_2^{\rk(M/e^\c)}v_2^{\cork(M/e^\c)} T^{\mathsf{Mat}}(M\backslash e)\ .$$
		In other words, we have 
		$$T^{\mathsf{Mat}}(M) = \begin{cases}
		(u_1+u_2) \,T^{\mathsf{Mat}}(M/e) & \mbox{if } e \mbox{ is a coloop in } M; \\
		(v_1+v_2) \,T^{\mathsf{Mat}}(M\backslash e) &  \mbox{if } e \mbox{ is a loop in } M; \\
		u_1 \,T^{\mathsf{Mat}}(M/ e) + v_2\, T^{\mathsf{Mat}}(M\backslash e) & \mbox{otherwise.}
		\end{cases}$$
		\end{prop}		
		
		\begin{proof}
		The first formula is a direct application of Proposition \ref{prop: deletion contraction general}. The pair $(M\backslash e^\c, M/e^\c)$ is $(c,c)$ if $e$ is a coloop in $M$, $(l,l)$ if $e$ is a loop in $M$, and $(c,l)$ otherwise. This implies the second formula.
		\end{proof}
		
		After specialization we recover the classical deletion-contraction recurrence formula for the Tutte polynomial:
		$$\mathfrak{T}_M(x,y)=(y-1)^{\cork(M\backslash e^\c)}\mathfrak{T}_{M/e}(x,y) + (x-1)^{\rk(M/e^\c)}\mathfrak{T}_{M\backslash e}(x,y)\ ,$$
		or equivalently:
		$$\mathfrak{T}_M(x,y)=\begin{cases} 
		x\, \mathfrak{T}_{M/ e}(x,y) & \textnormal{if } e \textnormal{ is a coloop in } M; \\ 
		y\, \mathfrak{T}_{M\backslash e}(x,y) & \textnormal{if } e \textnormal{ is a loop in } M; \\ 
		\mathfrak{T}_{M/e}(x,y)+\mathfrak{T}_{M\backslash e}(x,y) & \textnormal{otherwise.} \end{cases}$$
		
		\begin{rem}
		For $M$ a matroid with ground set $E$, the \emph{dual matroid} $M^\vee$ is the matroid on $E$ whose rank function is defined by $\rk_{M^\vee}(A)=\rk_M(A^\c) + |A|-\rk_M(E)$. It is a standard fact that we have $(M^\vee)^\vee=M$ and for every $A\subseteq E$, $(M\backslash A)^\vee = M^\vee/A^\c$ and $(M/A)^\vee=M^\vee\backslash A^\c$. Thus, the assignment $M\mapsto M^\vee$ is an isomorphism of minors systems $\mathsf{Mat}\stackrel{\sim}{\rightarrow} \mathsf{Mat}^{\mathrm{op}}$. The corresponding involution of the Grothendieck monoid $U(\mathsf{Mat})$ is given by $u\leftrightarrow v$. In view of the remarks of Section \ref{sec:functoriality duality}, this implies that we have a duality property:
		$$T^{\mathsf{Mat}}(M^\vee)_{(u_1,v_1,u_2,v_2)}=T^{\mathsf{Mat}}(M)_{(v_2,u_2,v_1,u_1)}\ ,$$
		or more classically:
		$$\mathfrak{T}_{M^\vee}(x,y)=\mathfrak{T}_M(y,x)\ .$$		
		\end{rem}
		 
		 \begin{rem}\label{rem:Potts}
		By applying the recipe given in Section \ref{sec:multivariate} one can recover the multivariate version of the Tutte polynomial a.k.a.\ the Potts model partition function \cite{fortuinkasteleyn,traldi,zaslavskystrong,Sokal}, whose preferred specialization is
		
		\begin{equation}\label{eq:multivariateTutte}
		\widetilde{\mathfrak{T}}_M(x,y)=\sum_{A\subseteq E(M)} \left(\prod_{e\in A}\alpha_e\right) (x-1)^{\rk(M)-\rk(A)}(y-1)^{|A|-\rk(A)}\ .
		\end{equation}
		\end{rem}		
		
	\subsection{Convolution formulae}
	
		We apply our general convolution formula from Theorem \ref{thm:general convolution formula} to the case of matroids.
		
		\begin{prop}\label{prop:convolution matroids universal}
		The universal Tutte character for $\mathsf{Mat}$ satisfies a universal convolution formula in the polynomial algebra $\KK[u_0,v_0,u_1,v_1,u_2,v_2]$:
		$$T^{\mathsf{Mat}}(M)_{(-u_0,-v_0,u_2,v_2)}= \sum_{A\subseteq E(M)} T^{\mathsf{Mat}}(M|A)_{(-u_0,-v_0,u_1,v_1)} \,T^{\mathsf{Mat}}(M/A)_{(-u_1,-v_1,u_2,v_2)}\ ,$$
		where the subscripts indicate the specialization of the variables.
		\end{prop}
		
		\begin{proof}
		This is a direct application of Theorem \ref{thm:general convolution formula}, by noting as in Remark \ref{rem:inverse U} that the inverse $\overline{N}$ of the universal norm $N:\KK\mathsf{Mat}\rightarrow \KK[u,v]$ is obtained by composing $N$ with the automorphism of $\KK[u,v]$ that sends $(u,v)$ to $(-u,-v)$.
		\end{proof} 
		
		This six-variable convolution formula specializes to a four-variable convolution formula for the Tutte polynomial that was proved by Kung \cite[Identity 1]{kungconvolution} in the context of multivariate Tutte polynomials (see also Wang \cite[Theorem 5.3]{wangconvolution}).
		
		\begin{prop}\label{prop:convolution matroids}
		The Tutte polynomial satisfies the following convolution formula in the polynomial algebra $\KK[a,b,c,d]$:
		\begin{equation}\label{eq: general convolution formula tutte matroids}
		\begin{split}
		\mathfrak{T}_M (1-ab,&1-cd)  \\ 
		=\sum_{A\subseteq E(M)} &a^{\rk(M)-\rk(A)}d^{|A|-\rk(A)}  \,\mathfrak{T}_{M|A}(1-a,1-c)\,\mathfrak{T}_{M/A}(1-b,1-d)\ .
		\end{split}
		\end{equation}
		\end{prop}
		
		\begin{proof}
		This is obtained from Proposition \ref{prop:convolution matroids universal} by setting $(u_0,v_0,u_1,v_1,u_2,v_2)=(-1,cd,-a,d,-ab,1)$ and using \eqref{eq:TMat}.
		\end{proof}
		
		One can further specialize to $(a,b,c,d)=(1,1-x,1-y,1)$ and get the classical convolution formula proved by Kook--Reiner--Stanton \cite{KRS} and Etienne--Las Vergnas \cite{etiennelasvergnas}:
		$$\mathfrak{T}_M(x,y)=\sum_{A\subseteq E(M)} \mathfrak{T}_{M|A}(0,y)\,\mathfrak{T}_{M/A}(x,0)\ .$$
		One can also specialize to less classical convolution formulae , e.g.\ with $(a,b,c,d)=(1-x,1,1,1-y)$:
		$$\mathfrak{T}_M(x,y)=\sum_{A\subseteq E(M)} (1-x)^{\rk(M)-\rk(A)}(1-y)^{|A|-\rk(A)}\mathfrak{T}_{M|A}(x,0)\,\mathfrak{T}_{M/A}(0,y)\ ,$$
		or with $(a,b,c,d)=(x-1,-1,-1,y-1)$:
		$$\mathfrak{T}_M(x,y)=\sum_{A\subseteq E(M)} (x-1)^{\rk(M)-\rk(A)}(y-1)^{|A|-\rk(A)}\mathfrak{T}_{M|A}(2-x,2)\,\mathfrak{T}_{M/A}(2,2-y)\ .$$
		Some other specializations of Proposition \ref{prop:convolution matroids universal} appear as sums over flats rather than over all subsets of the ground set\footnote{This was suggested by Spencer Backman.}, for example the following.
		\begin{prop}\label{prop:convolution flats}
		The Tutte polynomial satifies the following identity in $\KK[x,y]$:
		$$\mathfrak{T}_M(x,y) = \sum_{F \textnormal{ flat of } M} (x-1)^{\rk(M)-\rk(F)} \,\mathfrak{T}_{M|F}(1,y)\ .$$
		\end{prop}
		\begin{proof}
        Specializing Proposition \ref{prop:convolution matroids universal} at $(u_0,v_0,u_1,v_1,u_2,v_2) = (-1,1-y,0,1,x-1,1)$ gives the formula:
        $$\mathfrak{T}_M(x,y) = \sum_{F\subseteq E(M)} \mathfrak{T}_{M|F}(1,y)\, T^{\mathsf{Mat}}(M/F)_{(0,-1,x-1,1)}\ .$$
        Now notice that $$T^{\mathsf{Mat}}(M)_{(0,-1,x-1,1)}=(x-1)^{\rk(M)}\sum_{\substack{A\subseteq E(M)\\\rk(A)=0}}(-1)^{|A|}\ ,$$ 
        which equals $0$ if $M$ has at least one loop, and $(x-1)^{\rk(M)}$ otherwise. Thus, the only terms that survive in the sum are those such that $M/F$ has no loop, i.e. such that $F$ is a flat of $M$. 
		\end{proof}
		
		We note the following \emph{iterated convolution formula} for the Tutte polynomial, in the spirit of Remark \ref{rem:multiconvolution}.
		
		\begin{prop}\label{prop:multiconvolution matroids}
		The Tutte polynomial satisfies the following iterated convolution formula in the polynomial algebra $\KK[a_1,\ldots,a_n,b_1,\ldots,b_n]$:
		\begin{equation*}
		\begin{split}
		\mathfrak{T}&_M (1-a_1\cdots a_n,1-b_1\cdots b_n)  \\ 
		=&\sum_{\varnothing=A_0\subseteq A_1\subseteq\cdots\subseteq A_n= E(M)} \left(\prod_{i=1}^n \,a_i^{\rk(M/A_i)}\,b_i^{\cork(M|A_{i-1})}\,\mathfrak{T}_{M|A_i/A_{i-1}}(1-a_i,1-b_i)\right)\ .
		\end{split}
		\end{equation*}
		\end{prop}
		
		\begin{proof}
		For $i=0,\ldots,n$, let $N_i$ be the norm for the minors system $\mathsf{Mat}$ with values in $\KK[a_1,\ldots,a_n,b_1,\ldots,b_n]$ defined by
		$$N_i(M)=(b_{i+1}\cdots b_n)^{\rk(M)}(-a_1\cdots a_i)^{\cork(M)}\ .$$
		We note that we have $(\overline{N_0}*N_n)(M)=\mathfrak{T}_M(1-a_1\cdots a_n,1-b_1\cdots b_n)$. For $i=1,\ldots,n$ we have
		$$\overline{N_{i-1}}(M)=(-b_i\cdots b_n)^{\rk(M)}(a_1\cdots a_{i-1})^{\cork(M)}$$
		and thus
		$$(\overline{N_{i-1}}*N_i)(M)=(a_{i+1}\cdots a_n)^{\rk(M)}\,(b_1\cdots b_{i-1})^{\cork(M)}\,\mathfrak{T}_M(1-a_i,1-b_i)\ .$$
		According to Proposition \ref{prop:inverse convolution} every $N_i*\overline{N_i}$ is the identity for convolution and we have the formula:
		$$\overline{N_0}*\overline{N}_n = (\overline{N_0}*N_1)*(\overline{N_1}*N_2)*\cdots *(\overline{N}_{n-1}*N_n)$$
		The result then follows after collecting the powers of $a_i$ and $b_i$.
		\end{proof}
		
		The case $n=2$ of the above formula is Proposition \ref{prop:convolution matroids}. A 3-variable specialization of the case $n=3$ already appeared in the work of Reiner \cite[Theorem 3]{reinerinterpretation},
		where $(a_1,a_2,a_3,b_1,b_2,b_3)$ are set to $(1,-\frac{a}{b},1-u,1-v,-\frac{b}{a},1)$, with $a+b=1$. Another notable specialization is the following.
		
		\begin{prop}\label{prop:convolution cyclotomic}
		Assume that $\KK$ contains a primitive $n$-th root of unity $\xi$. Then the Tutte polynomial satisfies the following iterated convolution formula in the polynomial algebra $\KK[x,y]$:
		\begin{equation*}
		\begin{split}
		\mathfrak{T}_M & (x^n,y^n)  =\sum_{\varnothing=A_0\subseteq A_1\subseteq\cdots\subseteq A_n= E(M)}\\ 
		&\left(\prod_{i=1}^n \,(1-\xi^{i-1}x)^{\rk(M/A_i)}\,(1-\xi^{i-1} y)^{\cork(M|A_{i-1})}\,\mathfrak{T}_{M|A_i/A_{i-1}}(\xi^{i-1}x, \xi^{i-1}y)\right)\ .
		\end{split}
		\end{equation*}
		\end{prop}
		
		\begin{proof}
		This follows from applying Proposition \ref{prop:multiconvolution matroids} to $a_i=1-\xi^{i-1} x$ and $b_i=1-\xi^{i-1} y$.
		\end{proof}
		
		The case $n=2$ of Proposition \ref{prop:convolution cyclotomic} is simply:
		$$\mathfrak{T}_M(x^2,y^2)=\sum_{A\subseteq E(M)} (1-x)^{\rk(M)-\rk(A)}(1+y)^{|A|-\rk(A)}\mathfrak{T}_{M|A}(x,y)\,\mathfrak{T}_{M/A}(-x,-y)\ .$$

	\subsection{The minors systems of graphs and its universal Tutte character}\label{ssec:graphs}
	
		For a finite set $E$, let $\mathsf{Gra}[E]$ denote the set of graphs $G$ with edge set $E(G)=E$. This forms a multiplicative minors system $\mathsf{Gra}$ for which restriction and contraction of edges are defined in the usual way and the direct sum of two graphs is their disjoint union.
		Since a graph with zero edge is nothing but a finite number of isolated vertices, 
		one has an isomorphism of monoids $\mathsf{Gra}[\varnothing]\simeq a^{\mathbb{N}}$ which maps the graph with $k$ isolated vertices to the monomial $a^k$. This means that the multiplicative minors system $\mathsf{Gra}$ is not connected, and is thus different from the one appearing in \cite[3.2]{KMT}.
				
		\begin{rem}\label{noHopf}
		In view of Proposition \ref{prop:hopf when}, the linearization $\KK\mathsf{Gra}$ is a bimonoid in linear species but not a Hopf monoid; in other words its image by the Fock functor is a bialgebra but not a Hopf algebra. This is because its degree $0$ component $\KK\mathsf{Gra}_0$ is the polynomial algebra $\KK[a]$ equipped with the coproduct $\Delta(a)=a\otimes a$ and the counit $\varepsilon(a)=1$. This is a sub-bialgebra of  $\KK[a,a^{-1}]\simeq \KK[\mathbb Z]$, the group (Hopf) algebra of the group $\mathbb Z$, whose antipode map is induced by $a\mapsto a^{-1}$. Hence $\KK\mathsf{Gra}_0$ does not have an antipode.
		\end{rem}
		
		For $G$ a graph we denote by $V(G)$ the set of its vertices and by $k(G)$ the number of its connected components, and set $\rk(G)=|V(G)|-k(G)$. For a subset $A\subseteq E(G)$ we set $k(A)\doteq k(G|A)$ and $\rk(A)\doteq \rk(G|A)$. This notion of rank defines a morphism of minors system $\mathsf{Gra}\rightarrow \mathsf{Mat}\; , \; G\mapsto M(G)=(E(G),\rk)$.

		\begin{prop}
		The morphism of minors systems $\mathsf{Gra}\rightarrow\mathsf{Mat}$ induces an isomorphism of monoids
		$$U(\mathsf{Gra}) \stackrel{\sim}{\rightarrow} U(\mathsf{Mat})\simeq u^{\mathbb{N}}v^{\mathbb{N}}$$
		which maps the class of a graph $G$ to the monomial
		$$u^{|V(G)|-k(G)}v^{|E(G)|-|V(G)|+k(G)}\ .$$ 
		\end{prop}
		
		\begin{proof}
		One can prove this proposition by working through the presentation of $U(\mathsf{Gra})$ from Theorem \ref{thm:pres X(S) second}. Alternatively, let us denote by $\gamma:U(\mathsf{Gra})\rightarrow U(\mathsf{Mat})\simeq u^{\mathbb{N}}v^{\mathbb{N}}$ the morphism induced on the level of Grothendieck monoids. Let $\delta:u^{\mathbb{N}}v^{\mathbb{N}}\rightarrow U(\mathsf{Gra})$ be the morphism of monoids that sends $u$ to the class of a coloop (bridge) and $v$ to the class of a loop. Then clearly the composite $\gamma\circ\delta$ is the identity of $u^{\mathbb{N}}v^{\mathbb{N}}$. Since a graph with one edge is either a coloop or a loop together with additional isolated vertices, the composite $\delta\circ\gamma$ is the identity on the generators of $U(\mathsf{Gra})$ and the result follows.
		\end{proof}

		This implies that the universal Tutte character for the minors system $\mathsf{Gra}$ is the character
		$$T^{\mathsf{Gra}} : \KK\mathsf{Gra} \rightarrow \KK[u_1,v_1,a,u_2,v_2]$$
		defined, for a graph $G$, by
		\begin{equation*}
		\begin{split}
		T^{\mathsf{Gra}}(G) & = \sum_{A\subseteq E(G)} u_1^{\rk(G|A)} \, v_1^{\cork(G|A)}\,a^{k(G|A/A)}\, u_2^{\rk(G/A)} \, v_2^{\cork(G/A)} \\
		& = \sum_{A\subseteq E(G)} u_1^{|V(G)|-k(A)}v_1^{|A|-|V(G)|+k(A)}a^{k(A)} u_2^{k(A)-k(G)}v_2^{|E(G)|-|A|+k(G)-k(A)}\ .
		\end{split}
		\end{equation*}
		We mention two notable specializations of this universal invariant.
		\begin{enumerate}[--]
		\item After specialization to $(u_1,v_1,a,u_2,v_2)=(1,y-1,1,x-1,1)$, one recovers the classical Tutte polynomial 
		$$\mathfrak{T}_G(x,y)=\mathfrak{T}_{M(G)}(x,y)\ .$$
		\item After specialization to $(u_1,v_1,a,u_2,v_2)=(1,b,a,1,1)$, one recovers Tutte's dichromatic polynomial 
		$$Q_G(a,b)=\sum_{A\subseteq E(G)} a^{k(A)}b^{|A|-|V(G)|+k(A)}\ ,$$ 
		which is not an invariant of the matroid $M(G)$ \cite{Tuttedichromatic}. This can be computed recursively with the base case $Q_G(a,b)=a^k$ if $G$ consists of $k$ isolated vertices and the deletion-contraction recurrence formula:
		$$Q_G(a,b)=b^{\cork(G\backslash e^\c)}Q_{G/e}(a,b)+Q_{G\backslash e}(a,b)\ ,$$
		or equivalently:
		$$Q_G(a,b)=\begin{cases} (b+1)\, Q_{G\backslash e}(a,b) & \textnormal{if } e \textnormal{ is a loop in } G; \\ Q_{G/e}(a,b)+Q_{G\backslash e}(a,b) & \textnormal{otherwise.} \end{cases}$$
		One can apply the general convolution formula of Theorem \ref{thm:general convolution formula} and get a four-variable convolution formula for the dichromatic polynomial. A specialization of that formula is particularly simple: by applying Theorem \ref{thm:general convolution formula} with the norms $N_0(G)=1$, $N_1(G)=b^{\cork(G)}$, $N_2(G)=1$ and twist maps $\tau_1$ and $\tau_2$ corresponding to variables $a_1$ and $a_2$ we get a convolution formula in the polynomial algebra $\KK[a_1,a_2,b]$:
		$$Q_G(a_1a_2,b)=\sum_{A\subseteq E(G)}Q_{G|A}(a_1,b)\,\chi_{G/A}(a_2)\ ,$$
		where $\chi$ denotes the chromatic polynomial.
		\end{enumerate}
		
\section{Delta-matroids and perspectives}\label{sec:dmp}
	The examples of minors systems which \cite{KMT} focus on 
	are various notions of embedded graphs in surfaces and the matroid-like objects which correspond to them.
	Recently, Moffatt and Smith \cite{MoffattSmith}
	proposed the framework of \emph{delta-matroid perspectives} to unify these structures.
	Delta-matroid perspectives themselves yield as their Tutte invariant the Krushkal polynomial of an embedded graph,
	whereas forgetting various pieces of the structure yields
	the other embedded graph invariants and matroidal objects of~\cite{KMT}.
	We add to the examples appearing in these works 
	the \emph{saturated delta-matroids} of Tardos and Bouchet, which fit cozily in between.  
	The reader lost among the objects being introduced in this section may want to keep at hand
	diagram \eqref{eq:diagram minors systems}, which summarizes their relationships.

	The examples of delta-matroids and of their perspectives are of note because
	our universal Tutte character arrives directly at the correct target ring,
	which unlike our previous examples is not a polynomial ring.
	This is a significant feature which tends to be glossed over in the ``graph polynomials'' literature,
	where monoid rings are implicitly frequent but are handled as
	the subrings of polynomial rings spanned by monomials 
	whose exponents satisfy certain, often unstated, congruences.
	
	\subsection{Delta-matroids and matroid perspectives}
		Delta-matroids, defined by Bouchet \cite{Bouchet}, are a generalization of matroids
		whose most familiar axiom system, the \emph{feasible set} axioms,
		takes off from not the rank axioms of Section~\ref{sec:matroids} but the matroid basis axioms.  
		A basis of a matroid $M$ is a maximal set $X\subseteq E(M)$ with $\rk(X)=|X|$,
		and it is a standard fact that in this case $|X|=\rk(M)$.
		Let $\bigtriangleup$ denote symmetric difference of sets.
		\begin{defi}
		A \emph{delta-matroid} is a pair $D=(E,\mathcal B)$ where $E$ is a finite set
		and $\mathcal B\subseteq {2^E}$ is such that
	    \begin{itemize}
	    \item[($\Delta$1)] $\mathcal B\neq\varnothing$,
	    \item[($\Delta$2)] For any $X,Y\in\mathcal B$ and $x\in X\bigtriangleup Y$,
	    there exists $y\in X\bigtriangleup Y$, possibly equal to~$x$, such that $X\bigtriangleup\{x,y\}\in\mathcal B$.
	    \end{itemize}
		\end{defi}
		
		If the feasible sets of a delta-matroid all have equal cardinality,
		then they form the set of bases of some matroid on the same ground set, and conversely.
		The unfamiliar reader may take this as a statement of the basis axiom system for matroids
		(though one stated with the slight obfuscation of the use of~$\bigtriangleup$).
		Given an arbitrary delta-matroid $D$, 
		its feasible sets of minimum, resp.\ maximum, cardinality
		are the sets of bases of matroids on $E(D)$, 
		which are called the \emph{lower matroid} $D_{\rm min}$ of~$D$,
		resp.\ the \emph{upper matroid} $D_{\rm max}$ of~$D$.
		
		The \emph{deletion} of $M=(E,\mathcal B)$ by a singleton $\{a\}$ such that $a\not\in B$ for some $B\in\mathcal B$ (i.e.\ $a$ is not a \emph{coloop})
		is $M\backslash\{a\} = (E\setminus\{a\}, \{B:B\in\mathcal B,a\not\in B\})$.
		The \emph{contraction} of~$M$ by $\{a\}$, if $a\in B$ for some $B\in\mathcal B$ (i.e.\ $a$ is not a \emph{loop})
		is $M/\{a\} = (E\setminus\{a\}, \{B\setminus\{a\}:B\in\mathcal B,a\in B\})$.
		If $a$ is a coloop or a loop, then we set $M\backslash\{a\}$ and $M/\{a\}$ equal.  (It cannot be both a coloop and a loop.)
		For $A\subseteq E$ arbitrary, $M\backslash A$ and $M/A$ are defined as they must be according to the coassociativity axiom (M2) of Definition~\ref{def:ms}.
		These operations make delta-matroids into a connected minors system $\Delta\mathsf{Mat}$.
		It is multiplicative with the direct sum 
		$(E,\mathcal B)\oplus(E',\mathcal B')=(E\sqcup E',\{B\sqcup B':B\in\mathcal B,B'\in\mathcal B'\})$.
		\medskip

		Matroid perspectives, due to Las Vergnas in~\cite{lasvergnasextensions},
		are also known as \emph{strong maps} of matroids whose underlying ground set map is $\mathrm{id}:E\to E$.

		\begin{defi}
		A \emph{matroid perspective} is a pair $(M,M')$ of matroids on the same ground set $E=E(M)=E(M')$ such that for subsets $A\subseteq B\subseteq E$ we have the inequality
		$$\rk_M(B)-\rk_M(A)\geq \rk_{M'}(B)-\rk_{M'}(A)\ .$$
		\end{defi}
		
		By defining restriction, contraction and direct sum componentwise, one makes matroid perspectives into a connected multiplicative minors system $\mathsf{MatPer}$. It is naturally a minors subsystem of the product $\mathsf{Mat}\times\mathsf{Mat}$.\medskip

		We now explain how $\mathsf{MatPer}$ also forms a minors subsystem of the minors system $\mathsf{\Delta Mat}$ of delta-matroids.
		
		\begin{defi}
		A delta-matroid $D$ with ground set $E$ is said to be \emph{saturated} if for subsets $X\subseteq Y\subseteq Z$ of $E$, if $X$ and $Z$ are feasible then $Y$ is feasible.
		\end{defi}
		
		Saturated delta-matroids were introduced by Tardos \cite{tardosgeneralized} under the name \emph{generalized matroids} or \emph{g-matroids}, see also \cite{Bouchet,Bouchet-maps}. 
		They form a minors subsystem $\mathsf{Sat\Delta Mat}\subseteq \mathsf{\Delta Mat}$ of the minors system of delta-matroids.
		
		All delta-matroids on a one-element set are saturated. There are three non-saturated delta-matroids on a two-element set $\{e,f\}$, whose collections of feasible sets are
		$\{\varnothing,\{e,f\}\}$, $\{\varnothing,\{e\},\{e,f\}\}$ and $\{\varnothing,\{f\},\{e,f\}\}$.
		
		Saturated delta-matroids arise naturally from matroid perspectives. The next proposition is due to Tardos \cite{tardosgeneralized}.
		
		\begin{prop}\label{prop:tardos}
		Let $(M,M')$ be a matroid perspective with ground set $E$. Let us say that a subset $X\subseteq E$ is feasible if it is an independent subset for $M$ and a spanning subset for $M'$. Then this defines a saturated delta-matroid on $E$.
		\end{prop}
		
		\begin{proof}
 		Let us first remark that an independent set for $M'$ is independent for $M$ and that a spanning set for $M$ is spanning for $M'$. This implies that a basis $B$ of $M$ (resp. a basis $B'$ of $M'$) is feasible and that the set of feasible sets is not empty. 
 		
 		Next we have to prove the symmetric basis exchange axiom. There are two (dual) cases to consider.
 		\begin{enumerate}[(i)]
 		\item Let $X, Y$ be feasible sets and choose an element $x\in X\setminus Y$. If $X\setminus \{x\}$ is a spanning set for $M'$ then $X\setminus \{x\}=X\bigtriangleup \{x,x\}$ is feasible and we are done. Assume now that $X\setminus \{x\}$ is not a spanning set for $M'$. Then since $Y$ is a spanning set for $M'$ there exists an element $y\in Y\setminus X$ such that $X\bigtriangleup \{x,y\}=X\setminus\{x\}\cup\{y\}$ is a spanning set for $M'$. This implies that we have $\rk_{M'}(X\setminus\{x\}\cup\{y\})-\rk_{M'}(X\setminus \{x\})=1$. Since $(M,M')$ is a matroid perspective we get that $\rk_M(X\setminus\{x\}\cup\{y\})-\rk_M(X\setminus\{x\})=1$. Since $X\setminus\{x\}$ is an independent set of $M$, this implies that $X\setminus\{x\}\cup\{y\}$ is also an independent set of $M$, and is thus a feasible set.
 		\item Let $X, Y$ be feasible sets and choose an element $y\in Y\setminus X$. If $X\cup\{y\}$ is an independent set of $M$ then $X\cup\{y\}=X\bigtriangleup\{y,y\}$ is feasible and we are done. Assume now that $X\cup\{y\}$ is not an independent set of $M$. Then since $X$ is an independent set of $M$ there exists an element $x\in X\setminus Y$ such that $X\bigtriangleup\{y,x\}=X\cup\{y\}\setminus\{x\}$ is an independent set of $M$. This implies that $\rk_M(X\cup\{y\})-\rk_M(X\cup\{y\}\setminus \{x\})=0$. Since $(M,M')$ is a matroid perspective we get that $\rk_{M'}(X\cup\{y\})-\rk_{M'}(X\cup\{y\}\setminus \{x\})=0$. Since $X\cup\{y\}$ is a spanning set of $M'$ this implies that $X\cup\{y\}\setminus\{x\}$ is also a spanning set of $M'$, and is thus a feasible set.
 		\end{enumerate}
 		Finally, for $X\subseteq Y\subseteq Z$ subsets of $E$, if $X$ is spanning for $M'$ and $Z$ is independent for $M$, then $Y$ is spanning for $M'$ and independent for $M$, i.e.\ is a feasible set. We have thus proved that the delta-matroid that we have just defined is saturated. 
		\end{proof}
		
		Let us denote by $D(M,M')$ the saturated delta-matroid defined in the above proposition.
		
		\begin{prop}\label{prop:iso matper satdeltamat}
		The assignments $(M,M')\mapsto D(M,M')$ and $D\mapsto (D_{\max},D_{\min})$ induce an isomorphism of minors systems:
		$$\mathsf{MatPer} \simeq \mathsf{Sat\Delta Mat}\ .$$
		\end{prop}
		
		\begin{proof}
		The fact that we get functorial isomorphisms $\mathsf{MatPer}[E]\simeq \mathsf{Sat\Delta Mat}[E]$ is clear and one simply has to prove that they are compatible with restriction and contraction. Let us prove that $D(M,M')\backslash\{a\}=D(M\backslash\{a\},M'\backslash\{a\})$ for a matroid perspective $(M,M')$ with ground set $E$ and $a\in E$. If $a$ is a coloop in $M'$ (and thus also in $M$) then $a$ is a coloop in $D(M,M')$ and the feasible sets of $D(M\backslash \{a\},M'\backslash\{a\})$ are exactly the sets $X\setminus\{a\}$ for $X$ feasible in $D(M,M')$. Otherwise the feasible sets of $D(M\backslash\{a\},M\backslash\{a\})$ are those sets $X\subseteq E\setminus\{a\}$ that are spanning sets for $M'\backslash\{a\}$ and independent sets for $M\backslash \{a\}$, i.e.\ the feasible sets in $D(M,M')$ that do not contain $a$. Let us now prove that $D(M,M')/\{a\}=D(M/\{a\},M'/\{a\})$. If $a$ is a loop in $M$ (and thus also in $M'$) then $a$ is a loop in $D(M,M')$ and the feasible sets of $D(M/\{a\},M'/\{a\})$ are exactly the feasible sets in $D(M,M')$. Otherwise the feasible sets of $D(M/\{a\},M'/\{a\})$ are those sets $X\subseteq E\setminus\{a\}$ that are spanning sets for $M'/\{a\}$ and independent sets for $M/\{a\}$, i.e.\ the sets $X=Y\setminus\{a\}$ for $Y$ a feasible set of $D(M,M')$ containing $a$. 
		\end{proof}
		
		\begin{rem}\label{rem:delta non-morphisms}
		The assigment $D\mapsto (D_{\max},D_{\min})$ does \emph{not} induce a morphism of minors systems from the full minors system of delta-matroids $\mathsf{\Delta Mat}$ to $\mathsf{MatPer}$. For example, if $E(D)=\{e,f\}$ and $D$ has feasible sets $\{\varnothing,\{e\},\{e,f\}\}$,
		then the feasible sets of $D_{\rm min}/f$ are~$\{\varnothing\}$ 
		while those of $(D/f)_{\rm min}$ are~$\{\{e\}\}$.
		\end{rem}
		
		\begin{rem}
		By examining the proof of Proposition \ref{prop:tardos} one sees that saturated delta-matroids are those delta-matroids that satisfy the following strong form of the symmetric exchange axiom \cite{tardosgeneralized}. For feasible sets $X, Y$ one has:
		\begin{enumerate}[(i)]
		\item for every $x\in X\setminus Y$, either $X\setminus\{x\}$ is feasible or there exists $y\in Y\setminus X$ such that $X\setminus\{x\}\cup\{y\}$ is feasible;
		\item for every $y\in Y\setminus X$, either $X\cup\{y\}$ is feasible or there exists $x\in X\setminus Y$ such that $X\cup\{y\}\setminus\{x\}$ is feasible.
		\end{enumerate}
		\end{rem}
		
	\subsection{Delta-matroid perpectives}
		
		For delta-matroid perspectives our definitions follow \cite{MoffattSmith}.
		
		\begin{defi}
		A \emph{delta-matroid perspective} is a triple $(M,D,M')$ of two matroids $M$ and~$M'$
		and a delta-matroid $D$, all on the same ground set $E$,
		such that $(M,D_{\rm max})$ and $(D_{\rm min},M')$ are matroid perspectives.
		\end{defi}
		
		By Proposition~32 of \cite{MoffattSmith}, delta-matroid perspectives become a connected multiplicative minors systems $\mathsf{\Delta MatPer}$ with restriction, contraction, and direct sum operations defined componentwise. It is naturally a minors subsystem of $\mathsf{\Delta Mat}\times \mathsf{MatPer} \subseteq \mathsf{Mat}\times \mathsf{\Delta Mat}\times\mathsf{Mat}$.

		In what follows, we will compute our universal Tutte polynomial
		for~$\Delta\mathsf{MatPer}$,
		and then specialize it to $\mathsf{MatPer}$ and~$\Delta\mathsf{Mat}$.
		The specialization is based on parts (2), (3) and (4) of the following proposition,
		part (1) of which similarly recovers the results of Section~\ref{sec:matroids}.
		\begin{prop}\label{prop:dmp specialization}\mbox{}
		\begin{enumerate}
		\item[(1)] $M$ is a matroid if and only if $(M,M,M)$ is a delta-matroid perspective.
		\item[(2)] The following are equivalent: 
			$(M,M')$ is a matroid perspective;
			$(M,M,M')$ is a delta-matroid perspective;
			$(M,M',M')$ is a delta-matroid perspective;
			$(M,D(M,M'),M')$ is a delta-matroid perspective.
		\item[(3)] If $(M,D,M')$ is a delta-matroid perspective then $(M,M')$ is a matroid perspective.
		\item[(4)] $D$ is a delta-matroid if and only if $(D_{\rm max},D,D_{\rm min})$ is a delta-matroid perspective.
		\end{enumerate}
		\end{prop}
		\begin{proof}
		With the exception of the last equivalent condition of~(2), 
		this is proved as Proposition~4 of \cite{MoffattSmith}.
		Sufficiency of this equivalent condition is assured by~(3).
		Necessity follows from the definition of delta-matroid perspective in view of Proposition~\ref{prop:iso matper satdeltamat},
		which says that $D(M,M')_{\rm max}=M$ and $D(M,M')_{\rm min}=M'$.
		\end{proof}
		\begin{rem}
		The only nontrivial fact called on in the proof in~\cite{MoffattSmith} is that if $D$ is a delta-matroid
		then $(D_{\rm max},D_{\rm min})$ is a matroid perspective \cite{Bouchet-maps}.
		This, together with transitivity of the relation of being a matroid perspective, for instance proves~(3).		\end{rem}
		
		\medskip
		
		The following diagram, where all arrows are morphisms of minors systems, summarizes the links between the minors systems $\mathsf{Mat}$, $\mathsf{MatPer}$, $\mathsf{Sat\Delta Mat}$, $\mathsf{\Delta Mat}$ and $\mathsf{\Delta MatPer}$ explained in Proposition \ref{prop:dmp specialization}.\medskip
		\vspace{.6cm}
		\begin{equation}\label{eq:diagram minors systems}
		\xymatrixcolsep{5pc}\xymatrix{
		\mathsf{Mat} \ar@{^{(}->}[r] &\mathsf{MatPer} \simeq \mathsf{Sat\Delta Mat} \ar@{^(->}[d]\ar@{-->}@/_1.5pc/[l]\ar@{-->}@/^1.5pc/[l] \ar@{^{(}->}[r]& \Delta\mathsf{MatPer} \ar@{-->}@/_1.5pc/[l] \ar@{-->}@/^1pc/[ld]\\
		& \mathsf{\Delta Mat}& 
		}
		\end{equation}
		\vspace{.2cm}
		
		The inclusion $\mathsf{Mat}\hookrightarrow \mathsf{MatPer}$ is $M\mapsto (M,M)$ and it has two sections (pictured as dashed arrows) given by $(M,M')\mapsto M$ and $(M,M')\mapsto M'$. The inclusion $\mathsf{MatPer}\hookrightarrow \mathsf{\Delta MatPer}$ is $(M,M')\mapsto (M,D(M,M'),M)$, and it corresponds via the isomorphism $\mathsf{MatPer}\simeq \mathsf{Sat\Delta Mat}$ to the inclusion $\mathsf{Sat\Delta Mat}\hookrightarrow \mathsf{\Delta MatPer}$ given by $D\mapsto (D_{\max}, D,D_{\min})$. It has a section (pictured as a dashed arrow) given by $(M,D,M')\mapsto (M,M')$. Finally, the map $\mathsf{\Delta MatPer}\rightarrow \mathsf{\Delta Mat}$ pictured as a dashed arrow is given by $(M,D,M')\mapsto D$. 
		
		\begin{rem}
		We emphasize (see Remark \ref{rem:delta non-morphisms}) the fact that the natural morphisms of set species $\mathsf{\Delta Mat}\rightarrow \mathsf{Mat}$ given by $D\mapsto D_{\max}$ or $D_{\min}$ are \emph{not} morphisms of minors systems. The same applies to the morphisms of set species $\mathsf{\Delta Mat}\rightarrow \mathsf{MatPer}$ and $\mathsf{\Delta Mat}\rightarrow \mathsf{\Delta MatPer}$ given by $D\mapsto (D_{\max},D_{\min})$ and $D\mapsto (D_{\max},D,D_{\min})$, respectively.
		\end{rem}
		
		\begin{rem}
		Other natural morphisms of minors systems are not on this diagram. 
		For instance, there are two natural embeddings $\mathsf{Mat}\hookrightarrow\mathsf{MatPer}$ given by $M\mapsto (M,Z(E(M)))$ and $M\mapsto (F(E(M)),M)$, 
		where $Z(E)$ (resp.\ $F(E)$) denotes the zero matroid (resp.\ the free matroid) on a finite set $E$. 
		There are also the two embeddings $\mathsf{MatPer}\hookrightarrow\mathsf{\Delta MatPer}$ of Proposition~\ref{prop:dmp specialization}(2),
		given by $(M,M')\mapsto (M,M,M')$ and $(M,M')\mapsto (M,M',M')$.
		\end{rem}

	\subsection{The universal Tutte character for delta-matroid perspectives and the Krushkal polynomial}
		
		We now turn to the description of the Grothendieck monoid of $\mathsf{\Delta MatPer}$. The matroids on a one-element set are a coloop $c$ and a loop $l$.
		There is one delta-matroid on a single element which is not a matroid, the remaining set system
		$n = (\{e\},\{\varnothing,\{e\}\})$.
		From these, five delta-matroid perspectives can be assembled:
		$(c,c,l)$, $(c,l,l)$, $(c,c,c)$, $(l,l,l)$ and $(c,n,l)$.
		
		\begin{prop}
		Let us denote by $s$, $t$, $u$, $v$ and $w$ the classes in $U(\mathsf{\Delta MatPer})$ of the delta-matroid perspectives $(c,c,l)$, $(c,l,l)$, $(c,c,c)$, $(l,l,l)$ and $(c,n,l)$ respectively. Then we have an isomorphism of monoids
		\begin{equation}\label{eq:UDeltaMatPer}
		U(\mathsf{\Delta MatPer})\simeq (s^{\mathbb N}t^{\mathbb N}u^{\mathbb N}v^{\mathbb N}w^{\mathbb N})/\langle w^2=st\rangle\ .
		\end{equation}
		\end{prop}
		
		\begin{proof}
		By Theorem \ref{thm:pres X(S) second}, the classes $s$, $t$, $u$, $v$ and $w$ generate $U(\mathsf{\Delta MatPer})$ and the relations that they satisfy come from delta-matroid perspectives on two-element sets, of which there are 38 in all. Only two of these yield nontrivial relations: these are the perspective $(M,D,M')$ on ground set $\{e,f\}$ where $D$ has feasible sets $\{\varnothing,\{e\},\{e,f\}\}$, $M=D_{\rm max}$, and $M'=D_{\rm min}$, together with its isomorphic image under exchanging $e$ and~$f$. They give the relation $w^2=st$, hence the claim.
		\end{proof}
		
		In order to give a useful description of the isomorphism \eqref{eq:UDeltaMatPer} we will need a central fact about delta-matroids.
		
		\begin{lem}\label{lem:discrepancy dm}
		For a delta-matroid $D$ and a subset $A\subseteq E(D)$ we have an equality between non-negative integers:
		$$\rk(D_{\max})-\rk((D|A)_{\max}) - \rk((D/A)_{\max}) =\rk((D|A)_{\min}) + \rk((D/A)_{\min}) - \rk(D_{\min})\ .$$
		\end{lem}
		
		\begin{proof}
		If $A$ has cardinality $1$ then this is proved in \cite[Lemma 10]{KMT}. The general case follows by an easy induction on the cardinality of $A$.
		\end{proof}

		We note that in general both sides of the above equality are non-zero, even though they are if $D$ is a saturated delta-matroid. This is related to the fact, already noted in Remark \ref{rem:delta non-morphisms}, that $D\mapsto D_{\max}$ and $D\mapsto D_{\min}$ are not compatible with deletion and restriction, e.g. $(D|A)_{\max}$ differs from $D_{\max}|A$ in general.
		
		\begin{prop}
		The isomorphism \eqref{eq:UDeltaMatPer} maps the class of a delta-matroid perspective $(M,D,M')$ to the monomial
		\begin{equation}\label{eq:dmp class} 
		s^{\rk(D_{\rm min})-\rk(M')}\,t^{\rk(M)-\rk(D_{\rm max})}\,u^{\rk(M')}\,v^{\cork(M)}\,w^{\rk(D_{\rm max})-\rk(D_{\rm min})}\ .
		\end{equation}
		\end{prop}
		
		\begin{proof}
		Let $\psi:U(\mathsf{MatPer})\rightarrow (s^{\mathbb N}t^{\mathbb N}u^{\mathbb N}v^{\mathbb N}w^{\mathbb N})/\langle w^2=st\rangle$ denote the morphism of monoids mapping the class of a delta-matroid perspective $(M,D,M')$ to the monomial \eqref{eq:dmp class}. If we prove that $\psi$ is well-defined then we are done since on one-element delta-matroid perspectives we have $\psi(c,c,l)=s$, $\psi(c,l,l)=t$, $\psi(c,c,c)=u$, $\psi(l,l,l)=v$ and $\psi(c,n,l)=w$. Let us fix a delta-matroid perspective $(M,D,M')$ on a set $E$ and a subset $A\subseteq E$. We need to prove that we have the equality $\psi(M|A,D|A,M'|A)\,\psi(M/A,D/A,M'/A)=\psi(M,D,M')$. The fact that the powers of $u$ and $v$ agree simply follows from the equalities $\rk(M'|A)+\rk(M'/A)=\rk(M')$ and $\cork(M|A)+\cork(M/A)=\cork(M)$, so we only need to care about the powers of $s$, $t$ and $w$. We denote by $\delta$ the non-negative quantity of Lemma \ref{lem:discrepancy dm}. Up to the powers of $u$ and $v$, the product $\psi(M|A,D|A,M'|A)\,\psi(M/A,D/A,M'/A)$ equals
		$$s^{\delta+\rk(D_{\min})-\rk(M'|A)+\rk(M'/A)}t^{\delta+\rk(M|A)+\rk(M/A)-\rk(D_{\max})}w^{\rk(D_{\max})-\rk(D_{\min})-2\delta}\ .$$
		By using the relations $\rk(M'|A)+\rk(M'/A)=\rk(M')$, $\rk(M|A)+\rk(M/A)=\rk(M)$ and $s^\delta t^\delta=w^{2\delta}$ we see that this equals
		$$s^{\rk(D_{\min})-\rk(M')}t^{\rk(M)-\rk(D_{\max})}w^{\rk(D_{\max})-\rk(D_{\min})}\ ,$$
		which is $\psi(M,D,M')$ up to the powers of $u$ and $v$, and we are done.
		\end{proof}

		The universal Tutte character 
		$$T^{\Delta\mathsf{MatPer}}:\KK\Delta\mathsf{MatPer}\rightarrow 
		\KK[s_1,t_1,u_1,v_1,w_1,s_2,t_2,u_2,v_2,w_2]$$		
		is thus given by
		\begin{align*}
		T^{\Delta\mathsf{MatPer}}(M,D,M') &= \sum_{A\subseteq E} 
		s_1^{\rk((D|A)_{\rm min})-\rk_{M'}(A)}\,
		t_1^{\rk_M(A)-\rk((D|A)_{\rm max})}\,\\ &\cdot
		u_1^{\rk_{M'}(A)}\,
		v_1^{|A|-\rk_M(A)}\,
		w_1^{\rk((D|A)_{\rm max})-\rk((D|A)_{\rm min})} \\&\cdot
		s_2^{\rk((D/A)_{\rm min})-\cork_{M'}(A)}
		t_2^{\cork_M(A)-\rk((D/A)_{\rm max})}\,\\ & \cdot
		u_2^{\rk(M')-\rk_{M'}(A)}\,
		v_2^{|E\setminus A|-\cork_M(A)}\,
		w_2^{\rk((D/A)_{\rm max})-\rk((D/A)_{\rm min})} \ ,
		\end{align*}
		where if $M$ is a matroid and $A\subseteq E(M)$, then by $\cork_M(A)$ we mean $\rk(M)-\rk_M(A)$.
		
		We compare this to the Krushkal polynomial,
		an invariant of graphs in surfaces introduced by \cite{Krushkal} for orientable surfaces and Butler \cite{Butler} in general,
		and framed as an invariant of delta-matroid perspectives by \cite{MoffattSmith}.
		The name ``polynomial'' is arguably inapt:
		the invariant is naturally an element of an algebra isomorphic to~$\KK[U(\Delta\mathsf{MatPer})]$.
		The usual practice is to write its values after embedding in a polynomial algebra by
	    \begin{equation}\label{eq:dmp embedding}
	    \KK[U(\Delta\mathsf{MatPer})]\simeq \KK[s,t,u,v,w]/\langle w^2-st\rangle \hookrightarrow \KK[s^{\frac12},t^{\frac12},u,v],
	    \end{equation}
	    where $w$ maps to $s^{\frac12}t^{\frac12}$ and the other variables map to identically named variables.
		To relate our universal Tutte character to the Krushkal invariant
		we will use this embedding and write $w$ as $s^{\frac12}\,t^{\frac12}$,
		and similarly for subscripted variables.
		Then the monomial of \eqref{eq:dmp class} can be written 
		\[s^{\sigma(D)-\rk(M')}\,t^{\rk(M)-\sigma(D)}\,u^{\rk(M')}\,v^{\cork(M)}\ ,\]
		where $\sigma(D)$ denotes the mean $\frac12(\rk(D_{\rm max})+\rk(D_{\rm min}))$. We note that Lemma \ref{lem:discrepancy dm} means that we have $\sigma(D)=\sigma(D|A)+\sigma(D/A)$ for every delta-matroid $D$ and every $A\subseteq E(D)$.
		We can rewrite
		\begin{equation*}
		\begin{split}
		T^{\Delta\mathsf{MatPer}}&(M,D,M') = \sum_{A\subseteq E} 
		s_1^{\sigma(D|A)-\rk_{M'}(A)}\,
		t_1^{\rk_M(A)-\sigma(D|A)}\,
		u_1^{\rk_{M'}(A)}\,
		v_1^{|A|-\rk_M(A)}\,\\&\cdot
		s_2^{\sigma(D/A)-\cork_{M'}(A)}
		t_2^{\cork_M(A)-\sigma(D/A)}\,
		u_2^{\rk(M')-\rk_{M'}(A)}\,
		v_2^{|E\setminus A|-\cork_M(A)}\ .
		\end{split}
		\end{equation*}

		Specialization to 
		$$(s_1,t_1,u_1,v_1,s_2,t_2,u_2,v_2)=(a,b,1,y,1,1,x,1)$$
		recovers the Krushkal invariant in Moffatt and Smith's convenient presentation,
		\begin{equation*}
		\begin{split}
		K_{(M,D,M')}&(x,y,a,b) \\=
		 &\sum_{A\subseteq E} 
		x^{\rk(M')-\rk_{M'}(A)}\,
		y^{|A|-\rk_M(A)}\, 
		a^{\sigma(D|A)-\rk_{M'}(A)}\,
		b^{\rk_M(A)-\sigma(D|A)}\ .
		\end{split}
		\end{equation*}

		Again this is a reduced Tutte polynomial in the sense of Section~\ref{ssec:reduced}, so nothing was really lost in the specialization.
		The universal Tutte character is an evaluation of the Krushkal invariant 
		up to a prefactor which bears no further information itself:
		\begin{equation*}
		\begin{split}
		T^{\mathsf{\Delta MatPer}}&(M,D,M') \\
		& = s_2^{\sigma(D)-\rk(M')}\,t_2^{\rk(M)-\sigma(D)}\,u_1^{\rk(M')}\,v_2^{\cork(M)}\, K_{(M,D,M')}(\tfrac{u_2}{u_1},\tfrac{v_1}{v_2},\tfrac{s_1}{s_2},\tfrac{t_1}{t_2})\ .
		\end{split}
		\end{equation*}
		
		We leave the general eight-variable convolution formula for the Krushkal polynomial to the assiduous reader to write down,
		though in the two subsections to follow, we will present its specializations
		to the Las Vergnas and bivariate Bollob\'{a}s--Riordan polynomials.
		Note that the complications we discuss in the latter case (Section~\ref{sec:deltamatroids})
		brought about by the embedding \eqref{eq:dmp embedding}, and the consequent need for square roots of~$-1$,
		will occur for the Krushkal polynomial as well.

\subsection{Matroid perspectives and the Las Vergnas polynomial}\label{ssec:LV}

		We use the natural morphism of minors systems $\mathsf{\Delta MatPer}\rightarrow \mathsf{MatPer}\; , \; (M,D,M')\mapsto (M,M')$ to compute the Grothendieck monoid of $\mathsf{MatPer}$. We note that this morphism has a natural section $\mathsf{MatPer}\rightarrow \mathsf{\Delta MatPer}\;,\;(M,M')\mapsto (M,D(M,M'),M')$, which implies that the morphism $U(\mathsf{\Delta MatPer})\rightarrow U(\mathsf{MatPer})$ is a split surjection of monoids.
		
		\begin{prop}
		The morphism of minors systems $\mathsf{\Delta MatPer}\rightarrow \mathsf{MatPer}$ induces an isomorphism of monoids
		$$U(\mathsf{MatPer}) \simeq U(\Delta\mathsf{MatPer})/\langle s=t=w\rangle \simeq u^{\mathbb N}v^{\mathbb N}w^{\mathbb N}$$
		which maps the class of a matroid perspective $(M,M')$ to the monomial
		\[u^{\rk(M')}\,v^{\cork(M)}\,w^{\rk(M)-\rk(M')}\ .\]
		\end{prop}
		
		\begin{proof}
		We use the notations  $f:\mathsf{\Delta MatPer}\rightarrow \mathsf{MatPer}$ and $i:\mathsf{MatPer}\rightarrow\mathsf{\Delta MatPer}$. On structures on a one-element set, they act as follows:
		\[f(c,c,c) = (c,c),\quad f(l,l,l) = (l,l),\quad f(c,c,l)=f(c,l,l)=f(c,n,l)=(c,l),\]
		\[i(c,c) = (c,c,c),\quad i(l,l) = (l,l,l),\quad i(c,l)=(c,n,l)\ .\]
		This implies that the maps induced by $f$ and $i$ satisfy $f(u)=u$, $f(v)=v$, $f(s)=f(t)=f(w)=w$ and $i(u)=u$, $i(v)=v$, $i(w)=w$ where we denote by $u$, $v$, $w$ the classes in $U(\mathsf{MatPer})$ of the matroid perspectives $(c,c)$, $(l,l)$ and $(c,l)$ respectively. The first claim follows. The second claim follows by considering the monomial \eqref{eq:dmp class} for $(M,D(M,M'),M')$ and setting $s=t=w$.
		\end{proof}

		The universal Tutte character
		$$T^{\mathsf{MatPer}}:\KK\mathsf{MatPer}\rightarrow \KK[u_1,v_1,w_1,u_2,v_2,w_2]$$
		is thus given by
		\begin{align*}
		T^{\mathsf{MatPer}}(M,M') &= \sum_{A\subseteq E} 
		u_1^{\rk_{M'}(A)}\,v_1^{|A|-\rk_M(A)}\,w_1^{\rk_M(A)-\rk_{M'}(A)}\,u_2^{\rk(M')-\rk_{M'}(A)}
		\\&\cdot
		v_2^{|E|-|A|-(\rk(M)-\rk_M(A))}\,w_2^{(\rk(M)-\rk_M(A))-(\rk(M')-\rk_{M'}(A))} \ .
		\end{align*}

		The classical comparandum is the Tutte polynomial of a matroid perspective, a $3$-variable polynomial defined by Las Vergnas \cite{lasvergnasextensions}:
		\begin{equation*}
		\begin{split}
		\mathfrak{T}&_{M,M'}(x,y,z) \\
		&= \sum_{A\subseteq E}(x-1)^{\rk(M')-\rk_{M'}(A)}(y-1)^{|A|-\rk_M(A)}z^{(\rk(M)-\rk_M(A))-(\rk(M')-\rk_{M'}(A))}\ .
		\end{split}
		\end{equation*}
		Up to the shifting of the variables $x$ and $y$, it is a reduced Tutte character in the sense of Section \ref{ssec:reduced} obtained by specializing the variables to
		$$(u_1,v_1,w_1,u_2,v_2,w_2)=(1,y-1,1,x-1,1,z)\ .$$
		One can thus recover the universal Tutte character from the Las Vergnas polynomial up to a prefactor:
		\begin{equation}\label{eq:TMatPer}
		T^{\mathsf{MatPer}}(M,M') = u_1^{\rk(M')}v_2^{\cork(M)}w_1^{\rk(M)-\rk(M')} \,\mathfrak{T}_{M,M'}(1+\tfrac{u_2}{u_1}, 1+\tfrac{v_1}{v_2},\tfrac{w_2}{w_1}).
		\end{equation}
		
		One easily derives from our formalism a $6$-variable convolution formula for the Las Vergnas polynomial that we couldn't find in the literature.		
		
		\begin{prop}\label{prop:convolution LV}
		The Las Vergnas polynomial satisfies the following convolution formula in the polynomial algebra $\KK[a,b,c,d,e,f]$:
		\begin{equation*}
		\begin{split}
		\mathfrak{T}&_{M, M'}(1-ab,1-cd,-ef) \hspace{7cm}\\
		& = \sum_{A\subseteq E} a^{\rk(M')-\rk_{M'}(A)}\,d^{|A|-\rk_M(A)}\,e^{(\rk(M)-\rk_M(A))-(\rk(M')-\rk_{M'}(A))} \hspace{1cm}\\
		& \hspace{1.5cm}\cdot\mathfrak{T}_{M|A, M'|A}(1-a,1-c,-e) \, \mathfrak{T}_{M/A, M'/A}(1-b,1-d,-f)\ .
		\end{split}
		\end{equation*}
		\end{prop}	
		
		\begin{proof}
		 Theorem \ref{thm:general convolution formula} implies that the universal Tutte character satisfies a universal convolution formula in the polynomial algebra $\KK[u_0,v_0,w_0,u_1,v_1,w_1,u_2,v_2,w_2]$. One concludes by specializing the variables to
		$$(u_0,v_0,w_0,u_1,v_1,w_1,u_2,v_2,w_2)=(1,-cd,1,-a,d,-e,-ab,1,-ef)$$
		and using \eqref{eq:TMatPer}.
		\end{proof}

		By further specializing to $(a,b,c,d,e,f)=(1,1-x,1-y,1,1,-z)$ one gets a $3$-variable convolution formula that already appeared in \cite{Kayibi,kayibipirzada,KMT}:
		$$\mathfrak{T}_{M, M'}(x,y,z)=\sum_{A\subseteq E} \mathfrak{T}_{M|A, M'|A}(0,y,-1)\,\mathfrak{T}_{M/A, M'/A}(x,0,z)\ .$$
		
		\begin{rem}
		As was noted in \cite{KMT}, the classical relations between the Tutte polynomial of matroid perspectives and the Tutte polynomial of matroids can be explained in a fashion parallel to this subsection,
		by morphisms between the minors systems $\mathsf{MatPer}$ and $\mathsf{Mat}$.
		\end{rem}
		
\subsection{Delta-matroids and the bivariate Bollob\'{a}s--Riordan polynomial}\label{sec:deltamatroids}

		We use the morphism of minors systems $\mathsf{\Delta MatPer}\rightarrow \mathsf{\Delta Mat}\;,\; (M,D,M')\mapsto D$ to compute the Grothendieck monoid of $\mathsf{\Delta Mat}$. The assignment $D\mapsto (D_{\max}, D, D_{\min})$ gives a section of this morphism in the category of set species (but not in the category of minors systems: see Remark \ref{rem:delta non-morphisms}). This implies that the induced morphism $U(\mathsf{\Delta MatPer})\rightarrow U(\mathsf{\Delta Mat})$ is surjective.

		\begin{prop}
		The morphism of minors systems $\mathsf{\Delta MatPer}\rightarrow \mathsf{\Delta Mat}$ induces an isomorphism of monoids
		$$U(\mathsf{\Delta Mat}) \simeq U(\Delta\mathsf{MatPer})/\langle s=u, t=v\rangle \simeq (u^{\mathbb N}v^{\mathbb N}w^{\mathbb N})/\langle w^2=uv\rangle$$
		which maps the class of a delta-matroid $D$ to the monomial
		\begin{equation}\label{eq:delta class}
		u^{\rk(D_{\rm min})}\,v^{|E(D)|-\rk(D_{\rm max})}\,w^{\rk(D_{\rm max})-\rk(D_{\rm min})}\ .
		\end{equation}
		\end{prop}
		
		\begin{proof}
		We use the notation $g:\mathsf{\Delta MatPer}\rightarrow \mathsf{\Delta Mat}$. On delta-matroid perspectives on a one-element set, it acts as follows:
		\[g(c,c,c) = g(c,c,l) = c,\quad g(l,l,l) = g(c,l,l) = l,\quad g(c,n,l) = n\]
		This implies that the map induced by $g$ satisfies $g(s)=g(u)=u$, $g(t)=g(v)=v$, $g(w)=w$ where we denote by $u$, $v$, $w$ the classes in $U(\mathsf{\Delta Mat})$ of the delta-matroids $c$, $l$, $n$ respectively. We note that the relation $w^2=uv$ in $U(\mathsf{\Delta Mat})$ is induced by the delta-matroids on $\{e,f\}$ with feasible sets $\{\varnothing,\{e\},\{e,f\}\}$ and $\{\varnothing,\{f\},\{e,f\}\}$. By direct inspection, one checks that the remaining 13 delta-matroids on $\{e,f\}$ yield trivial relations in $U(\mathsf{\Delta Mat})$ and the first claim follows. The second claim follows by considering the monomial \eqref{eq:dmp class} for $(D_{\max},D,D_{\min})$ and setting $s=u$, $t=v$.
		\end{proof}

		The universal Tutte character
		\[T^{\Delta\mathsf{Mat}} : \mathbb K\Delta\mathsf{Mat}\to 
		\mathbb K[u_1,v_1,w_1,u_2,v_2,w_2]/\langle w_1^2-u_1v_1, w_2^2-u_2v_2\rangle\]
		is thus given by
		\begin{align*}
		T^{\Delta\mathsf{Mat}}(D) &= \sum_{A\subseteq E(D)} 
		u_1^{\rk((D|A)_{\rm min})}\, v_1^{|A|-\rk((D|A)_{\rm max})}\, w_1^{\rk((D|A)_{\rm max})-\rk((D|A)_{\rm min})}
		\\&\qquad\qquad\cdot
		u_2^{\rk((D/A)_{\rm min})}\, v_2^{|E(D)|-\rk((D/A)_{\rm max})}\, w_2^{\rk((D/A)_{\rm max})-\rk((D/A)_{\rm min})}\ .
		\end{align*}
		
		It is customary to write its values after embedding in a polynomial algebra by
	    \begin{equation}
	    \KK[U(\Delta\mathsf{Mat})]\simeq \KK[u,v,w]/\langle w^2-uv\rangle \hookrightarrow \KK[u^{\frac12},v^{\frac12}],
	    \end{equation}
	    where $w$ maps to $u^{\frac12}v^{\frac12}$ and the other variables map to identically named variables. Then the monomial \eqref{eq:delta class} equals $u^{\sigma(D)}\,v^{|E(D)|-\sigma(D)}$. In what follows we use the notation $\sigma(A)\doteq\sigma(D|A)$. The universal Tutte character, written in $\KK[u_1^{\frac12}, v_1^{\frac12}, u_2^{\frac12}, v_2^{\frac12}]$, then becomes
		$$T^{\Delta\mathsf{Mat}}(D) = \sum_{A\subseteq E(D)} u_1^{\sigma(A)} \, v_1^{|A|-\sigma(A)} \, u_2^{\sigma(D)-\sigma(A)} \, v_2^{|E(D)|-|A|-\sigma(D)+\sigma(A)}\ .$$

		Specialization to $(u_1,v_1,u_2,v_2)=(1,y-1,x-1,1)$ recovers the bivariate Bollob\'{a}s--Riordan polynomial, which lives in $\KK[(x-1)^{\frac12},(y-1)^{\frac12}]$:
		\[\widetilde{\mathfrak{R}}_D(x,y) = \sum_{A\subseteq E(D)} 
		(x-1)^{\sigma(D)-\sigma(A)} (y-1)^{|A|-\sigma(A)}.\]
		Up to the shifting of the variables, it is a reduced Tutte character in the sense of Section \ref{ssec:reduced}. One can thus recover the universal Tutte character from the bivariate Bollob\'{a}s--Riordan polynomial up to a prefactor:
		\begin{equation}\label{eq:TDeltaMat}
		T^{\mathsf{\Delta Mat}}(D)=u_1^{\sigma(D)}v_2^{|E(D)|-\sigma(D)}\,\widetilde{\mathfrak{R}}_D(1+\tfrac{u_2}{u_1},1+\tfrac{v_1}{v_2})\ .
		\end{equation}\medskip
		
		When dealing with convolution formulae , the embedding $w\mapsto u^{\frac12}v^{\frac12}$
		turns out to be inconvenient. Indeed, the natural involution of $\KK[U(\Delta\mathsf{Mat})]$ which maps $(u,v,w)$ to $(-u,-v,-w)$ does not extend to $\KK[u^{\frac12},v^{\frac12}]$ unless $\KK$ contains a square root of $-1$ (compare with the discussion before Theorem 16 in \cite{KMT}). Rather, our formalism produces a universal convolution formula that lives in the quotient ring $\KK[u_i,v_i,w_i \, , \, i=0,1,2]/\langle w_i^2-u_iv_i\rangle$. At the cost of introducing a square root of $-1$, we can produce a 4-variable convolution formula for the bivariate Bollob\'{a}s--Riordan polynomial that we couldn't find in the literature.
		
		\begin{prop}\label{BRKung}
		The bivariate Bollob\'{a}s--Riordan polynomial satisfies the following convolution formula in the algebra $\KK[(-1)^{\frac12},a^{\frac12},b^{\frac12},c^{\frac12},d^{\frac12}]$:
		\begin{equation*}
		\begin{split}
		\widetilde{\mathfrak{R}}_D(1&-ab,1-cd)\\
		&=\sum_{A\subseteq E(D)} a^{\sigma(D)-\sigma(A)}\,d^{|A|-\sigma(A)}\, \widetilde{\mathfrak{R}}_{D|A}(1-a,1-c)\,\widetilde{\mathfrak{R}}_{D/A}(1-b,1-d)\ .
		\end{split}
		\end{equation*}
		\end{prop}	
		
		\begin{proof}
		 This follows from Theorem \ref{thm:general convolution formula} after specializing the variables to
		$$(u_0,v_0,u_1,v_1,u_2,v_2)=(1,-cd,-a,d,-ab,1)$$
		and using \eqref{eq:TDeltaMat}.
		\end{proof}	
	
		By further specializing to $(a,b,c,d)=(1,1-x,1-y,1)$ 
		one gets a two-variable convolution formula in the algebra $\KK[(-1)^{\frac12},(x-1)^{\frac12},(y-1)^{\frac12}]$:
		$$\widetilde{\mathfrak{R}}_D(x,y)=\sum_{A\subseteq E(D)} \widetilde{\mathfrak{R}}_{D|A}(0,y)\,\widetilde{\mathfrak{R}}_{D/A}(x,0)\ .$$
		This formula was proved in \cite{KMT} for \emph{even} delta-matroids, in which case it lives in the algebra $\KK[(x-1)^{\frac12},(y-1)^{\frac12}]$.\medskip

		We conclude this section with the diagram obtained from diagram \eqref{eq:diagram minors systems} after applying the Grothendieck monoid functor. The dashed arrows are surjections and we indicate the additional relations that are imposed on the generators of their domains.
		
		$$\xymatrixcolsep{5pc}
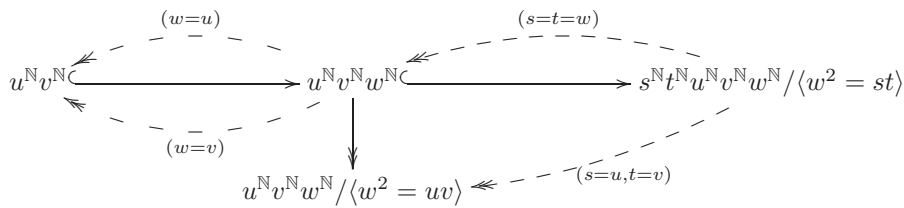
\xymatrix{
		u^{\mathbb{N}}v^{\mathbb{N}} \ar@{^{(}->}[r] &u^{\mathbb{N}}v^{\mathbb{N}}w^{\mathbb{N}} \ar@{->>}[d]\ar@{-->>}@/_1.5pc/[l]_-{(w=u)} \ar@{-->>}@/^1.5pc/[l]^-{(w=v)} \ar@{^{(}->}[r]& s^{\mathbb{N}}t^{\mathbb{N}}u^{\mathbb{N}}v^{\mathbb{N}}w^{\mathbb{N}}/\langle w^2=st \rangle \ar@{-->>}@/_1.5pc/[l]_-{(s=t=w)} \ar@{-->>}@/^1pc/[ld]^-{(s=u,t=v)}\\
		& u^{\mathbb{N}}v^{\mathbb{N}}w^{\mathbb{N}}/\langle w^2=uv \rangle& 
		}$$
		\vspace{.2cm}
		
		All these morphisms account for identities relating the different Tutte characters of the minors systems appearing in the diagram.
				
\section{Relative matroids and relative Tutte polynomials}\label{sec:relative}

	\subsection{The minors system of relative matroids}
	
		Starting with a minors system $\mathsf{S}$ one can construct a new minors system $\mathsf{RelS}$, where objects have the same structure as in $\mathsf{S}$ 
		but where a subset of the ground set, the \emph{zero set},
		is excluded from the argument of~$\mathsf{RelS}$ and is therefore not available to the minor operations.
		
		\begin{defi}
		Let $\mathsf{S}$ be a set species. For a finite set $E$, a \emph{relative structure} of type $\mathsf{S}$ on $E$ is an element of $\mathsf{S}[E\sqcup E_0]$ for some finite set $E_0$, called the \emph{zero set}. We always consider relative structures \enquote{up to isomorphism on the zero set}: we identify $X\in\mathsf{S}[E\sqcup E_0]$ and $X'\in\mathsf{S}[E\sqcup E'_0]$ if there is a bijection $\sigma:E\sqcup E_0\stackrel{\sim}{\rightarrow} E\sqcup E'_0$ extending the identity of $E$ such that $\mathsf{S}[\sigma](X)=X'$.
		\end{defi}

		This definition allows us to talk about the \emph{sets} $\mathsf{RelS}[E]$ of relative structures of type $\mathsf{S}$ on $E$, which form a set species $\mathsf{RelS}$. It contains $\mathsf{S}$ as the subspecies consisting of the structures with empty zero set. If $\mathsf{S}$ has the structure of a (multiplicative) minors system, then this structures extends in the obvious way to $\mathsf{RelS}$.\medskip
		
		We focus here on the case $\mathsf{S}=\mathsf{Mat}$. The multiplicative minors system $\mathsf{RelMat}$ of \emph{relative matroids} has been studied under different names by many authors \cite{lasvergnasextensions,brylawskidecomposition,chaikenported,lasvergnassetpointed,diaohetyeirelative}. It is not connected since $\mathsf{RelMat}[\varnothing]$ is the set of isomorphism classes of all matroids; as a monoid, it is the free monoid on the set of isomorphism classes of connected matroids.

		\begin{rem}
		We could avoid talking about isomorphism classes on the complement of $E$ by switching from the framework of set species to the more convenient framework of set species in two sorts, i.e.\ functors from the square of the category of finite sets and bijections to the category of sets. In other words, a set species in two sorts is the datum, for every pair $(E,E_0)$ of finite sets, of a set $\mathsf{T}[E,E_0]$ and, for every bijections $\sigma:E\stackrel{\sim}{\rightarrow} F$ and $\sigma_0:E_0\stackrel{\sim}{\rightarrow} F_0$, of a map $\mathsf{T}[\sigma,\sigma_0]:\mathsf{T}[E,E_0]\rightarrow \mathsf{T}[F,F_0]$, which satisfy $\mathsf{T}[\sigma\circ \tau,\sigma_0\circ \tau_0]=\mathsf{T}[\sigma,\sigma_0]\circ\mathsf{T}[\tau,\tau_0]$. From a set species $\mathsf{S}$ one defines a set species in two sorts $\mathsf{T}$ by $\mathsf{T}[E,E_0]\doteq \mathsf{S}[E\sqcup E_0]$. If $\mathsf{S}$ has the structure of a minors system then this structure is transferred to $\mathsf{T}$, where restriction and contraction are only taken with respect to subsets of the first factor $E$.
		\end{rem}
		
	\subsection{The universal Tutte character for relative matroids and the relative Tutte polynomial}
	
		It is natural to relate relative matroids and matroid perspectives as in the next proposition.
		
		\begin{prop}\label{prop:relmat matper satdeltamat}
		Let $M$ be a relative matroid on $E$ represented by a matroid on $E\sqcup E_0$. Then the following holds.
		\begin{enumerate}
		\item The pair $(M\backslash E_0, M/E_0)$ is a matroid perspective.
		\item If we declare that a subset $A\subseteq E$ is feasible if there exists a subset $A_0\subseteq E_0$ such that $A\sqcup A_0$ is a basis of $M$, then this defines a saturated delta-matroid.
		\item These assignments define morphisms of minors systems $\mathsf{RelMat}\rightarrow \mathsf{MatPer}$ and $\mathsf{RelMat}\rightarrow \mathsf{Sat\Delta Mat}$ that are compatible with the isomorphism $\mathsf{MatPer}\simeq\mathsf{Sat\Delta Mat}$ from Proposition \ref{prop:iso matper satdeltamat}.
		\end{enumerate}
		\end{prop}
		
		\begin{proof}
		It is clear that $(M\backslash E_0,M/E_0)$ is a matroid perspective and that $M\mapsto (M\backslash E_0,M/E_0)$ induces a morphism of minors systems $\mathsf{RelMat}\rightarrow\mathsf{MatPer}$. Let $D$ be the saturated delta-matroid corresponding to $(M\backslash E_0,M/E_0)$ through the bijection $\mathsf{MatPer}[E]\simeq \mathsf{Sat\Delta Mat}[E]$. Let $A\subseteq E$ be a feasible subset in $D$, i.e.\ $A$ is an independent set of $M\backslash E_0$ and a spanning set of $M/E_0$. Then $A\sqcup E_0$ is a spanning set in $M$ and there exists $A_0\subseteq E_0$ such that $A\sqcup A_0$ is a basis of $M$. Conversely, if $A\subseteq E$ and $A_0\subseteq E_0$ are such that $A\sqcup A_0$ is a basis of $M$, then $A$ is an independent set of $M\backslash E_0$ and a spanning set of $M/E_0$. This completes the proof.
		\end{proof}
		
		\begin{rem}\label{rem:higgs}
		By a result of Edmonds and Higgs \cite{higgsstrong}, the morphism $\mathsf{RelMat}[E]\rightarrow\mathsf{MatPer}[E]$ is surjective for any finite set $E$. There is a natural section $\mathsf{MatPer}[E]\rightarrow \mathsf{RelMat}[E]$ given by Higgs, called the \emph{Higgs major} operation by Crapo \cite{CrapoBowdoin}. It defines a morphism of set species $\mathsf{MatPer}\rightarrow\mathsf{RelMat}$ which is not a morphism of minors systems. To see this, note that the Higgs major of the matroid perspective $(c,l)$ on a one-element set $\{e\}$ is the relative matroid represented by the uniform matroid $U_{1,2}$ on a two-element set $\{e,e_0\}$. Contracting $e$ in this relative matroid leads to the relative matroid on the empty set represented by a loop on $\{e_0\}$, while the Higgs major of $(c,l)/\{e\}=(U_{0,0},U_{0,0})$ is the relative matroid on the empty set represented by the empty matroid.
		
		A setting in which the corresponding section is a morphism of minors systems is provided by the elementary factorizations and majors of Kung's Exercise~8.14 in~\cite{KungStrong}. Here, in place of $\mathsf{RelMat}$, one uses a minors system of relative matroids whose zero set is totally ordered, considered up to order-preserving isomorphism on the zero set.
		\end{rem}

		The Grothendieck monoid $U(\mathsf{RelMat})$ is enormous, being generated by the infinite set of all relative matroids on a ground set of cardinality $1$. In order to work with a tractable invariant, is is natural to work with norms which factor through the morphism $\mathsf{RelMat}\rightarrow \mathsf{MatPer}\simeq\mathsf{Sat\Delta Mat}$, i.e.\ replace the universal norm with the norm 
		$$N:\mathsf{RelMat}\rightarrow u^{\mathbb{N}}v^{\mathbb{N}}w^{\mathbb{N}}\ .$$
		For a relative matroid $M$ on $E$ represented by a matroid on $E\sqcup E_0$, its image by $N$ is the monomial
		\begin{eqnarray*}
		& u^{\rk(M/E_0)}v^{\cork(M\backslash E_0)}w^{\rk(M\backslash E_0)-\rk(M/E_0)}\\
		= & u^{\rk(E\sqcup E_0)-\rk(E_0)}v^{|E|-\rk(E)}w^{\rk(E)+\rk(E_0)-\rk(E\sqcup E_0)}
		\end{eqnarray*}
		
		Let us denote by $R_0\doteq \KK[\mathsf{RelMat}[\varnothing]]$ the free commutative $\KK$-algebra on the set of isomorphism classes of connected matroids and work with the universal twist map $\tau:\KK\mathsf{RelMat}\rightarrow R_0$. \medskip
		
		By a slight abuse of notation, we still denote by $T^{\mathsf{RelMat}}$ the Tutte character associated to two copies of the norm $N$ and the twist map $\tau$:
		$$T^{\mathsf{RelMat}} : \KK\mathsf{RelMat} \rightarrow R_0[u_1,v_1,w_1,u_2,v_2,w_2]\ .$$
		For a relative matroid $M$ on $E$ represented by a matroid on $E\sqcup E_0$ it is computed by the formula:
		\begin{equation*}
		\begin{split}
		T&^{\mathsf{RelMat}}(M) \\
		& = \sum_{A\subseteq E}  \, \,\tau(M/A|E_0) \,\, u_1^{\rk(A\sqcup E_0)-\rk(E_0)} v_1^{|A|-\rk(A)} w_1^{\rk(A)+\rk(E_0)-\rk(A\sqcup E_0)} \\
		& \hspace{1cm} \cdot u_2^{\rk(E\sqcup E_0)-\rk(A\sqcup E_0)}  v_2^{|E|-|A|-\rk(E)+\rk(A)} w_2^{\rk(E)+\rk(A\sqcup E_0)-\rk(A)-\rk(E\sqcup E_0)} \ .
		\end{split}
		\end{equation*}
		
		Specialization of the variables to $(u_1,v_1,w_1,u_2,v_2,w_2)=(1,y-1,1,x-1,1,z)$ yields a reduced Tutte character in $\KK[x,y,z]$ that deserves the name \emph{relative Tutte polynomial}:
		\begin{equation*}
		\begin{split}
		\mathfrak{T}^{\mathrm{rel}}_M(x,y,z) = \sum_{A\subseteq E}  \, \,\tau(M/A|E_0) \,\, &(x-1)^{\rk_M(E\sqcup E_0)-\rk_M(A\sqcup E_0)}(y-1)^{|A|-\rk_M(A)} \\
		& \;\;\cdot z^{\rk_M(E)+\rk_M(A\sqcup E_0)-\rk_M(A)-\rk_M(E\sqcup E_0)} \ .
		\end{split}
		\end{equation*}
		
		If we apply the morphism $R_0\rightarrow \KK$ which maps every matroid on $E_0$ to $1$, and multiply by $z^{\rk_M(E\sqcup E_0)-\rk_M(E)}$, then we find the \enquote{Tutte polynomial of $M$ pointed by $E_0$} defined by Las Vergnas \cite{lasvergnasextensions,lasvergnassetpointed}.\medskip
		
		The relative Tutte polynomial satisfies a convolution formula similar to that of the Las Vergnas polynomial (Proposition \ref{prop:convolution LV}) that we leave to the assiduous reader to write down.
		
		\begin{rem}
		For a general minors system $\mathsf{S}$ and the corresponding relative minors system $\mathsf{RelS}$, it would be natural to only look at norms that factor through the image of the morphism $\mathsf{RelS}\rightarrow \mathsf{S}\times\mathsf{S} \; , \; X\mapsto (X\backslash E_0,X/E_0)$. In the case $\mathsf{S}=\mathsf{Mat}$, the image of this morphism is indeed the minors system $\mathsf{MatPer}\subseteq \mathsf{Mat}\times\mathsf{Mat}$ by Remark \ref{rem:higgs}.
		\end{rem}
		
\section{Generalised permutohedra, a.k.a.\ submodular functions}\label{sec:SF}
		Submodular functions are a natural weakening of matroid rank functions.
		Having been introduced to combinatorics along multiple channels,
		they go by many names: there are no essential differences (though there are minor ones)
		between submodular functions, generalised permutohedra, and polymatroids.

		\subsection{The minors system of submodular functions}
		\begin{defi}
		A {\em submodular function} is a pair $M=(E,\rk)$, where $E$ is a finite set and $\rk : 2^E \to \mathbb Z$ is a function such that, for all $X,Y\subseteq E$,
	    \begin{itemize}
	    \item[(Sub1)] $\rk(\varnothing) = 0$,
	    \item[(Sub2)] $\rk(X\cup Y) + \rk(X\cap Y) \leq \rk(X) + \rk (Y)$.
	    \end{itemize}
		\end{defi}
        Again we will write $(E,\rk)$ as $(E(M),\rk_M)$ when helpful for clarity.
		Taking the codomain of $\rk$ to be some other totally ordered group than~$\mathbb Z$ would require no unexpected adjustments to the theory.
		
		Submodular functions form a connected multiplicative minors system, that we denote $\mathsf{SF}$, with operations formally identical to those for matroids.
		We reprise the definitions quickly.
		Given a submodular function $M=(E,\rk)$ and a set $A\subseteq E$, 
		the restriction $M| A$ has ground set $A$ and rank function the restriction of $\rk$,
		whereas the contraction $M/A$ has ground set $E(M) \setminus A$ and rank function 
		$\overline\rk$ given by $\overline\rk(B) \doteq \rk(B \cup A) - \rk(A)$ for $B \subseteq  E(M) \setminus A$.
		The direct sum of submodular functions $M=(E,\rk)$ and $M'=(E',\rk')$ is defined to be $M\oplus M'\doteq (E\sqcup E', \rk\oplus\rk')$, where for $A\subseteq E$ and $A'\subseteq E'$,
		$$(\rk\oplus\rk')(A\sqcup A') = \rk(A)+\rk'(A').$$
		
		A \emph{polymatroid} is a nondecreasing submodular function, i.e.\ is such that $X\subseteq Y$ implies $\rk(X)\leq \rk(Y)$.
		We let $\mathsf{PMat}\subset\mathsf{SF}$ denote the minors system of polymatroids.
		Polymatroids were introduced by Edmonds \cite{Edmonds}, who construed them as polytopes,
		to provide a context in which a greedy algorithm would correctly optimise any linear functional with nonnegative coefficients.
		
		Later, Postnikov introduced \emph{generalised permutohedra} \cite{Postnikov}, a class of polytopes which form a minors system $\mathsf{GP}$ isomorphic to $\mathsf{SF}$, 
		if we insist as we will here that their vertices have integer coordinates 
		(or more generally, coordinates in the ordered group that we use as the codomain of $\rk$).
		A generalised permutohedron is a polytope of the form
		\[P(M) \doteq \{x\in\mathbb R^E : \sum_{e\in A}x_e\leq\rk(A)\mbox{ for all }A\subseteq E, \ 
		\sum_{e\in E}x_e=\rk(E)\}\]
		for some submodular function $M=(E,\rk)$.
		Generalised permutohedra subsume many well-appreciated combinatorial families of polyhedra;
		we refer the reader to \cite{Postnikov} for details.
		The name of the class reflects the inclusion of the permutohedra among them.
		For instance, if $\rk_M(A) = -\binom{|A|}2$,
	 	then $P(M)$ is the convex hull of the orbit of the symmetric group $\Sigma_E$
		consisting of points whose coordinates in some order are $0,-1,\ldots,-|E|+1$.
		Generalised permutohedra are the polytopes whose normal fan coarsen the normal fan of this permutohedron.
		That is, informally, the permutohedron is being generalised by allowing its facets to be translated around and even pinched down to lower-dimensional faces, so long as no faces in new directions are created.

		We state the minor operations on~$\mathsf{GP}$.
		If $P\subseteq\mathbb R^E$ is a generalised permutohedron and $A\subseteq E$,
		then the face of $P$ on which the linear functional $\sum_{e\in A}x_e$ is maximised
		can be written as the Cartesian product of a polytope in $\mathbb R^A$ and a polytope in $\mathbb R^{E\setminus A}$.
		These factors are respectively $P|A$ and $P/A$.
		Direct sum of generalised permutohedra is Cartesian product.
		Aguiar and Ardila prove that $\mathsf{GP}$ is the largest connected multiplicative minors system
		whose elements are polytopes (with integer vertices) and whose operations are as just prescribed \cite[Theorem~6.1]{ardila-aguiar}.

	\subsection{The universal Tutte character for submodular functions}
	
		\begin{prop}\label{prop:X(SF)}
		There is an embedding of monoids
		\[
			\iota:U(\mathsf{SF})\hookrightarrow x^{\mathbb N}y^{\mathbb Z}, \quad
			[(E,\rk)] \mapsto x^{|E|} y^{\rk(E)},
		\]
		whose image is $\{1\}\cup\{x^ay^b : a>0\}$.
		\end{prop}
			
		\begin{proof}
		That $\iota$ is well-defined is easily checked from the presentation of~$U(\mathsf{SF})$.
		
		If $M=(E,\rk)$ is a submodular function such that $\iota([M])=y^b$ for some $b\in\mathbb Z$,
		then $|E|=0$ so $\rk(E)=0$ and thus $b=0$.
		Since all factorizations of~$y^b$ within the codomain of~$\iota$ are into other elements of form $y^{b'}$,
		this shows that $y^b\not\in\im \iota$ when $b\neq0$.
		Conversely, for each $b\in\mathbb Z$, the submodular function $s_b=(E,\rk_b)$
		where $E$ is a singleton and $\rk_b(E)=b$ satisfies $\iota(s_b)=xy^b$,
		and these elements generate the remainder of the claimed image.
		
		It remains to be shown that $\iota$ is injective, for which it's enough to show 
		that the generators $xy^b$ of its image have unique preimages which generate $X(\mathsf{SF})$ --
		indeed $\iota^{-1}(xy^b) = \{s_b\}$ is clear, and the generation is assured by Theorem~\ref{thm:pres X(S) second} --
		and that the relations between these in the codomain hold also in the domain.
		These relations are generated by the relations
		\[xy^a\cdot xy^b = xy^c\cdot xy^d\]
		for integers $a$, $b$, $c$, $d$ with $a+b=c+d$.
		We may assume without loss of generality that $a$ is the greatest of these integers, so that $a\geq d$.
		This makes the function $\rk:2^{\{e,f\}}\to\mathbb Z$ given by
		\[
			\rk(\varnothing)=0,\quad
			\rk(\{e\})=a,\quad
			\rk(\{f\})=c,\quad
			\rk(\{e,f\})=c+d=a+b
		\]
		submodular.  Restricting and contracting this function on $\{e\}$ and on~$\{f\}$ provides the needed relation
		\[[(\{e,f\},\rk)] = [s_a][s_b] = [s_c][s_d]\]
		in~$U(\mathsf{SF})$.
		\end{proof}
		
		Mutatis mutandis, the above proof shows that for polymatroids the Grothendieck monoid $U(\mathsf{PMat})$
		injects in the same way into $x^{\mathbb N}y^{\mathbb N}$.
		The Grothendieck monoid functor takes the inclusion $\mathsf{PMat}\subset\mathsf{SF}$ to the obvious inclusion $x^ay^b\mapsto x^ay^b$ between these monoids.
		
		Moreover, every matroid is a polymatroid (and therefore a submodular function), 
		and there is a natural inclusion of minors systems $j:\mathsf{Mat}\hookrightarrow\mathsf{PMat}$.
		The resulting morphism $U(j):U(\mathsf{Mat})\hookrightarrow U(\mathsf{PMat})$ 
		is given on generators by $u\mapsto xy$, $v\mapsto x$.

		The next proposition follows in the familiar way:
		\begin{prop}\label{prop:T^SF}
		The universal Tutte character for submodular functions is
		\[T^{\mathsf{SF}}:\KK\mathsf{SF}\to\KK[x_1y_1^i,x_2y_2^i : i\in\mathbb Z] \hookrightarrow \KK[x_1,y_1^{\pm 1},x_2,y_2^{\pm 1}]\]
		defined for a submodular function $M$ by
		\begin{align}\label{eq:T^SF}
			T^{\mathsf{SF}}(M) 
			&= \sum_{A\subseteq E(M)} x_1^{|E(M|A)|}\,y_1^{\rk(M|A)}\,x_2^{|E(M/A)|}\,y_2^{\rk(M/A)}
			\notag\\&= x_2^{|E(M)|}\,y_2^{\rk(M)} \sum_{A\subseteq E(M)} 
			\left(\frac{x_1}{x_2}\right)^{|A|}\,\left(\frac{y_1}{y_2}\right)^{\rk_M(A)}.
		\end{align}
		\end{prop}
		The universal Tutte character for polymatroids is formally identical,
		except that in the codomain we take $i\in\mathbb N$.
		
		As expected by Section~\ref{ssec:reduced}, the invariant $T^{\mathsf{SF}}(M)$ is ``essentially'' bivariate, with a bivariate prefactor.
		Also, by applying the inverse of the above morphism $U(j)$ and recollecting like variables,
		we could make the exponents in formula~\eqref{eq:T^SF} identical to those in the universal Tutte character for matroids, the corank-nullity generating function \eqref{eq:TMat}.
		However, this rearrangement is not especially motivated for submodular functions and polymatroids,
		since $|A|-\rk(A)$ is not guaranteed to be nonnegative.
		Indeed, we are not aware of any previous interest in $T^{\mathsf{SF}}(M)$ in the literature,
		and thus we have no authorities to follow in selecting a preferred reduced Tutte character.
		
		Again, we leave to the reader the easy task of using Theorem~\ref{thm:general convolution formula}
		to write down a convolution formula for $T^{\mathsf{SF}}(M)$.
				
		\subsection{Comparison to Oxley-Whittle}\label{sec:OW}
		An \emph{$r$-polymatroid} is a polymatroid $M$ such that $\rk_M(\{e\})\leq r$ for all $e\in E(M)$.
		These form a sub-minors system of $\mathsf{PMat}$ which we'll denote $\mathsf{PMat}_r$.
		For example, matroids are exactly 1-polymatroids.
		We continue the notation $s_i$ for single-element polymatroids from the proof of Proposition~\ref{prop:X(SF)}.
		
		In \cite{oxley-whittle}, Oxley and Whittle answer the question of finding a universal deletion-contraction invariant for 2-polymatroids.
		They adopt a more general framing of this question than ours,
		allowing as a deletion-contraction invariant any function $\Phi:\mathsf{PMat}_r\to R$
		to a commutative $\KK$-algebra $R$ satisfying
		\begin{equation}\label{eq:OW rec}		
		\Phi(M)=C_1(M\backslash e^\c,M/e^\c)\, \Phi(M/e) + C_2(M\backslash e^\c,M/e^\c) \, \Phi(M\backslash e)
		\end{equation}
		where $C_1$ and $C_2$ are two unrestricted families of coefficients indexed by pairs of single-element $r$-polymatroids.
		More exactly, since $M/e = M\backslash e$ when $M\backslash e^\c = M/e^\c$,
		\cite{oxley-whittle} only writes one coefficient in this case, $C(s_i,s_i)\doteq C_1(s_i,s_i)+C_2(s_i,s_i)$;
		and since $\rk(M\backslash e^\c)\geq\rk(M/e^\c)$ always, 
		it omits from consideration the coefficients with indices contrary to this.
		This leaves $(r+1)^2$ parameters in all.
		Recalling Proposition~\ref{prop:universal}, our own notion of deletion-contraction invariant demands in \eqref{eq:OW rec} 
		that $C_1$ be independent of~$M/e^\c$ and $C_2$ of~$M\backslash e^\c$,
		and that both be evaluations of norms in the remaining index, $N_1$ and $N_2$ respectively.
		Note that equation~\eqref{eq:OW rec} is a more general setup than the one considered in Proposition~\ref{prop:general recurrence},
		because although that proposition drops the requirement of normhood, 
		it still assumes that $C_1$ is independent of~$M/e^\c$ and $C_2$ of $M\backslash e^\c$.
		
		For matroids, the extra generality of formulation \eqref{eq:OW rec} is benign.
		The condition that $N:\KK\mathsf{Mat}\to R$ be a norm imposes no relation on its values $N(s_0)$ and~$N(s_1)$ at a loop and a coloop,
		so the equations
		\begin{align*}
		N_1(s_1)+N_2(s_1) &= C(s_1,s_1), &
		N_1(s_0)+N_2(s_0) &= C(s_0,s_0), \\
		N_1(s_1) &= C_1(s_1,s_0), &
		N_2(s_0) &= C_2(s_1,s_0)
		\end{align*}
		can be solved for $N_i(s_j)$ whatever their right hand sides,
		and a deletion-contraction invariant in the sense of Oxley and Whittle is also a deletion-contraction invariant in our sense.
		
		For 2-polymatroids, by contrast, \eqref{eq:OW rec} is strictly more general than our formulation.
		The nine equations in Oxley and Whittle's parameters to be solved for our norm evaluations $N_i(s_j)$ are
		\[
		\begin{array}{r@{}lr@{}l}
		\multicolumn{2}{r@{}}{N_1(s_2)+N_2(s_2)} & \multicolumn{2}{@{}l}{{}= C(s_2,s_2) \doteq x,}\\
		\multicolumn{2}{r@{}}{N_1(s_1)+N_2(s_1)} & \multicolumn{2}{@{}l}{{}= C(s_1,s_1) \doteq z,}\\
		\multicolumn{2}{r@{}}{N_1(s_0)+N_2(s_0)} & \multicolumn{2}{@{}l}{{}= C(s_0,s_0) \doteq y,}\\
		N_1(s_2) & {}= C_1(s_2,s_1) \doteq d, & \quad N_2(s_1) & {}= C_2(s_2,s_1) \doteq c, \\
		N_1(s_1) & {}= C_1(s_1,s_0) \doteq b, &       N_2(s_0) & {}= C_2(s_1,s_0) \doteq a, \\
		N_1(s_2) & {}= C_1(s_2,s_0) \doteq n, &       N_2(s_0) & {}= C_2(s_2,s_0) \doteq m, 	\end{array}
		\]
		where for ease of comparison we have inserted the single-letter notation of \cite{oxley-whittle} for the parameters. 
		(The $x$ and $y$ here are unrelated to our $x$ and $y$ in Proposition~\ref{prop:X(SF)}.)

		The above equations are consistent only if $a=m$, $d=n$, and $z=b+c$.
		Moreover, if $N:\KK\mathsf{PMat}_2\to R$ is a norm, 
		its evaluations at single-element polymatroids satisfy the relation $N(s_0)N(s_2) = N(s_1)^2$,
		both of these evaluating to $u^2v^2$ under the map $\iota$ of Proposition~\ref{prop:X(SF)}.
		This imposes for consistency the further conditions $m(x-n) = c^2$ and $(y-m)n = b^2$.
		
		The main theorem of~\cite{oxley-whittle}, Corollary~3.15, 
		classifies the 2-polymatroid invariants satisfying \eqref{eq:OW rec} with codomain $\KK$ into five families, each with a universal invariant, and each imposing some equations on the nine parameters. 		However, directly after stating the theorem, the authors dismiss four of the families as uninteresting.
		The one which retains their approbation is exactly the invariant $T^{\mathsf{PMat}_2}$,
		and the equations it imposes are just those derived above:
		$a=m$, $d=n$, $m(x-n) = c^2$, $(y-m)n = b^2$, $z=b+c$.
		(They also list some inequations for this family, but these are an artificial expedient to make the five families disjoint.) 
		
		\begin{rem}
		One can check that of the five families of~\cite{oxley-whittle},
		the only ones satisfying the assumptions of Proposition~\ref{prop:general recurrence}
		are $T^{\mathsf{PMat}_2}$ and the scalar multiples of the counit of~$\mathsf{PMat}_2$.
		\end{rem}
		
		\begin{rem}
		In \cite[Section 1]{Kayibi} Kayibi observed a relationship between the Las Vergnas polynomial (Section~\ref{ssec:LV})
		and Oxley and Whittle's invariant, our $T^{\mathsf{PMat}_2}$.
		In our language, this relationship arises functorially from the injection of minors systems
		$\mathsf{MatPer}\to\mathsf{PMat}_2$ which sends the matroid perspective $(M,M')$
		to the 2-polymatroid $N$ on the same ground set such that $\rk_N = \rk_M + \rk_{M'}$.
		The induced map of Grothendieck monoids is $(u, v, w) \mapsto (xy^2, x, xy)$.
		\end{rem}
		
\section{Colored matroids}\label{sec:colored}

	\subsection{The minors system of colored matroids}
    
		Let $\Lambda$ be a fixed set. For $\mathsf{S}$ a set species, one can form the set species $\mathsf{S}_\Lambda$ of $\Lambda$-colored structures of type $\mathsf{S}$, defined by
		$$\mathsf{S}_\Lambda[E]=\mathsf{S}[E] \times \Lambda^E\ .$$
		It is naturally a (multiplicative) minors system if $\mathsf{S}$ is. We will only treat the case $\mathsf{S}=\mathsf{Mat}$ in detail. The elements of $\mathsf{Mat}_\Lambda[E]$ are then pairs $(M,\lambda)$ with $M$ a matroid on the ground set $E$ and $\lambda:E\rightarrow\Lambda$ a coloring function. The case where $\Lambda=\{+,-\}$ is relevant to knot theory because a $\{+,-\}$-colored planar graph is essentially the same as a planar link diagram.

	\subsection{The universal Tutte character for colored matroids and the colored Tutte polynomial}
	
		The computation of the Grothendieck monoid of $\mathsf{Mat}_\Lambda$ is straightforward.
		
		\begin{prop}
		For every color $\lambda\in\Lambda$, let us denote by $u_\lambda$ \textup(\/resp.\ $v_\lambda$\/\textup) the class in $U(\mathsf{Mat}_\Lambda)$ of a coloop \textup(\/resp.\ a loop\/\textup) with the color $\lambda$. We have an isomorphism of monoids
		$$U(\mathsf{Mat}_\Lambda) \simeq \langle u_\lambda,v_\lambda \;,\; \lambda\in\Lambda\rangle \,/\, \langle u_\lambda v_\mu = u_\mu v_\lambda \; ,\; \lambda,\mu\in\Lambda\rangle \ .$$
		\end{prop}
		
		\begin{proof}
		By Theorem \ref{thm:pres X(S) second}, the elements $u_\lambda$ and $v_\lambda$ generate $U(\mathsf{Mat}_\Lambda)$ and the relations that they satisfy come from colored matroid on two-element sets. The uniform matroid $U_{1,2}$ whose elements are colored with two colors $\lambda$ and $\mu$ gives rise to the relation $u_\lambda v_\mu=u_\mu v_\lambda$, and one readily checks that these are the only non-trivial relations.
		\end{proof}
		
		It is natural to embed $U(\mathsf{Mat}_\Lambda)$ in the monoid
		$$U'(\mathsf{Mat}_\Lambda) = \langle u,v \;;\; a_\lambda \; , \; \lambda\in\Lambda\rangle = U(\mathsf{Mat}) \times \langle a_\lambda \; , \; \lambda\in\Lambda\rangle $$
		by $u_\lambda\mapsto ua_\lambda$ and $v_\lambda\mapsto va_\lambda$. With this presentation, the class of a colored matroid $(M,\lambda)$ in $U(\mathsf{Mat}_\Lambda)$ is the monomial 
		$$u^{\rk(M)}v^{\cork(M)}\prod_{e\in E(M)}a_{\lambda(e)}\ . $$
		
		The universal Tutte character
		\begin{equation*}
		\begin{split}
		T^{\mathsf{Mat}_\Lambda}:\KK\mathsf{Mat}_\Lambda  \rightarrow  \KK[ u_{\lambda,i},& v_{\lambda,i} \;,\; \lambda\in\Lambda \; , i=1,2] \,/\, \langle u_{\lambda,i} v_{\mu,i} = u_{\mu,i} v_{\lambda,i}\rangle \\
		& \hookrightarrow \KK[u_1,v_1,u_2,v_2\;;\;a_{\lambda,1},a_{\lambda,2} \; , \;\lambda\in\Lambda] 
		\end{split}
		\end{equation*}
		is thus given by 
		\begin{equation*}
		\begin{split}
		T&^{\mathsf{Mat}_\Lambda}(M,\lambda)\\
		&=\sum_{A\subseteq E(M)}\left(\prod_{e\in A}a_{\lambda(e),1}\right)\left(\prod_{e\notin A}a_{\lambda(e),2}\right) u_1^{\rk(M|A)} \, v_1^{\cork(M|A)}\, u_2^{\rk(M/A)} \, v_2^{\cork(M/A)}\ .
		\end{split}
		\end{equation*}
		
		By specializing to $(u_1,v_1,u_2,v_2)=(1,y-1,x-1,1)$ and $(a_{\lambda,1},a_{\lambda,2})=(a_{\lambda},1)$, we obtain an invariant known as the \emph{colored Tutte polynomial}:
		$$\mathfrak{T}_{(M,\lambda)}(x,y)=\sum_{A\subseteq E(M)} \left(\prod_{e\in A}a_{\lambda(e)}\right) (x-1)^{\rk(M)-\rk(A)}(y-1)^{|A|-\rk(A)}\ , $$
		We note that this is formally the same thing as the multivariate Tutte polynomial \eqref{eq:multivariateTutte}, although in a slightly different setting. 
		
	\subsection{Comparison with Bollob\'{a}s--Riordan}\label{sec:BR}
	
		In \cite{bollobasriordancoloured} Bollob\'{a}s and Riordan consider morphisms $\Phi:\mathsf{Mat}_\Lambda\rightarrow R$, with $R$ a ring, which take value $1$ on the empty matroid and satisfy a deletion-contraction recurrence formula of the form
		\begin{equation}\label{eq:deletion contraction colored}
		\Phi(M)=N_1(M\backslash e^\c)\,\Phi(M/e) + N_2(M/e^\c)\,\Phi(M\backslash e)\ ,
		\end{equation}
		where $N_1,N_2:\mathsf{Mat}_\Lambda[\{*\}]\rightarrow R$ are any two functions. For $\lambda\in\Lambda$ let us denote by $c_\lambda$ (resp. $l_\lambda$) a coloop (resp. a loop) with color $\lambda$. For $i=1,2$ we set $u_{\lambda,i}\doteq N_i(c_\lambda)$ and $v_{\lambda,i}\doteq N_i(l_{\lambda})$. Bollob\'{a}s and Riordan use a different set of variables
		 $$(x_\lambda,y_\lambda,X_\lambda,Y_\lambda)=
		 (u_{\lambda,1},v_{\lambda,2},u_{\lambda,1}+u_{\lambda,2},v_{\lambda,1}+v_{\lambda,2})\ . $$ 
		 The main result of \cite{bollobasriordancoloured}, Theorem 2, gives necessary and sufficient conditions on these variables 
		 so that the deletion-contraction recurrence formula gives a well-defined morphism $\Phi$. Here we give a different proof of this result that does not require a basis activities expansion and only uses our result on general deletion-contraction recurrences (Proposition \ref{prop:general recurrence}).
	
		\begin{thm}
		The deletion-contraction recurrence \eqref{eq:deletion contraction colored} gives a well-defined morphism $\Phi:\mathsf{Mat}_\Lambda\rightarrow R$ if and only if the following equalities hold in $R$ for all $\lambda,\mu,\nu\in\Lambda$:
		\begin{equation}\label{eq:BRfirst}
		u_{\lambda,1}v_{\mu,1}-u_{\mu,1}v_{\lambda,1}=u_{\lambda,2}v_{\mu,2}-u_{\mu,2}v_{\lambda,2} \ .
		\end{equation}
		\begin{equation}\label{eq:BRsecond}
		(u_{\lambda,1}v_{\mu,1}-u_{\mu,1}v_{\lambda,1})(v_{\nu,1}+v_{\nu,2})=0 \ .
		\end{equation}
		\begin{equation}\label{eq:BRthird}
		(u_{\lambda,2}v_{\mu,2}-u_{\mu,2}v_{\lambda,2})(u_{\nu,1}+u_{\nu,2})=0 \ .
		\end{equation}
		\end{thm}
		
		\begin{proof}
		We use the general criterion given in Proposition \ref{prop:general recurrence}. The morphism $\Phi$ is obviously well-defined in cardinality $1$ and satisfies $\Phi(c_\lambda)=u_{\lambda,1}+u_{\lambda,2}$ and $\Phi(l_\lambda)=v_{\lambda,1}+v_{\lambda,2}$. In cardinality $2$ the only colored matroid that gives a non-trivial relation is $U_{1,2}$ with colors $\lambda$ and $\mu$, and relation \eqref{eq:general recurrence} is exactly \eqref{eq:BRfirst}. In cardinality $3$, applying relation \eqref{eq:general recurrence} to the matroid $U_{1,3}$ with colors $\lambda$, $\mu$, $\nu$ gives rise to relation \eqref{eq:BRsecond}, and applying it to the matroid $U_{2,3}$ with colors $\lambda$, $\mu$, $\nu$ gives rise to relation \eqref{eq:BRthird}. 
		
		We now need to prove that these three relations ensure that $\Phi$ is well-defined in any cardinality. Let $M$ be any colored matroid on a ground set $E$ of cardinality $\geq 2$, let us fix $e,f\in E$ and let $\lambda$ and $\mu$ denote the colors of $e$ and $f$ respectively. Relation \eqref{eq:general recurrence} is trivial if $M_{e,f}$ and $M^{e,f}$ have an underlying matroid different from $U_{1,2}$. We then have three cases to consider.
		\begin{enumerate}
		\item If $M_{e,f}$ and $M^{e,f}$ have $U_{1,2}$ as their underlying matroid, then $M_{e,f}=M^{e,f}$ and $M$ is necessarily a direct sum $M=M'\oplus M''$ with $E(M')=\{e,f\}$ and $E(M'')=\{e,f\}^\c$. This implies that $M/\{e,f\}=M\backslash \{e,f\}=M''$ and relation \eqref{eq:general recurrence} reads
		$$(u_{\lambda,1}v_{\mu,1}-u_{\mu,1}v_{\lambda,1}) \Phi(M'')=(u_{\lambda,2}v_{\mu,2}-u_{\mu,2}v_{\lambda,2})\Phi(M'')\ ,$$
		which is a consequence of \eqref{eq:BRfirst}.
		\item If $M_{e,f}$ has $U_{1,2}$ as its underlying matroid and $M^{e,f}$ does not, then relation \eqref{eq:general recurrence} reads
		$$(u_{\lambda,1}v_{\mu,1}-u_{\mu,1}v_{\lambda,1}) \Phi(M/\{e,f\})=0\ .$$
		Now $M/\{e,f\}$ cannot be the empty matroid, so by construction, $\Phi(M/\{e,f\})$ is in the ideal of $R$ generated by the elements $u_{\nu,1}+u_{\nu,2}$ and $v_{\nu,1}+v_{\nu,2}$ for $\nu\in\Lambda$. Then the above relation is a consequence of relations \eqref{eq:BRfirst}, \eqref{eq:BRsecond} and \eqref{eq:BRthird}.
		\item If $M^{e,f}$ has $U_{1,2}$ as its underlying matroid and $M_{e,f}$ does not then a dual argument applies.
		\end{enumerate}
		In all three cases, the relation given by Proposition \ref{prop:general recurrence} is a consequence of relations \eqref{eq:BRfirst}, \eqref{eq:BRsecond} and \eqref{eq:BRthird}, and we are done.
		\end{proof}
		
		In the variables used by Bollob\'{a}s and Riordan, the relations from the theorem read (compare with \cite[Theorem 2]{bollobasriordancoloured}):
		$$x_\lambda Y_\mu-x_\mu Y_\lambda=X_\lambda y_\mu - X_\mu y_\lambda \ ,$$
		$$(x_\lambda Y_\mu -x_\mu Y_\lambda -x_\lambda y_\mu +x_\mu y_\lambda)Y_\nu=0\ ,$$
		$$(X_\lambda y_\mu - X_\mu y_\lambda -x_\lambda y_\mu +x_\mu y_\lambda) X_\nu =0 \ .$$
		
		The condition that $u_{\lambda,1}+u_{\lambda,2}=X_\lambda$ and $v_{\lambda,1}+v_{\lambda,2}=Y_\lambda$ are not zero divisors in $R$ for all $\lambda$ is natural, since they are the values taken by $\Phi$ on $c_\lambda$ and $l_\lambda$, respectively. If $R$ is a domain then this means that these two quantities are non-zero. If this condition is satisfied then equations \eqref{eq:BRfirst}, \eqref{eq:BRsecond} and \eqref{eq:BRthird} are equivalent to
		\begin{equation}\label{eq:uvequation norms}
		u_{\lambda,1}v_{\mu,1}-u_{\mu,1}v_{\lambda,1}=u_{\lambda,2}v_{\mu,2}-u_{\mu,2}v_{\lambda,2}=0 \ ,
		\end{equation}
		or in the variables used by Bollob\'{as} and Riordan (compare with \cite[Corollary 3]{bollobasriordancoloured}):
		$$x_\lambda Y_\mu - x_\mu Y_\lambda = X_\lambda y_\mu - X_\mu y_\lambda = x_\lambda y_\mu - x_\mu y_\lambda\ .$$
		We note that equation \eqref{eq:uvequation norms} amounts to saying that $N_1$ and $N_2$ are norms. 
		This is another argument for the naturality of our use of norms in our deletion-contraction recurrence formula.

\section{Arithmetic matroids}\label{sec:arith}
\emph{Arithmetic matroids} \cite{Md1, BM}
are one of several recently introduced kinds of decorated matroid.
Matroids encode information about the topology of hyperplane arrangements:
notably the graded dimension of the cohomology of a complex hyperplane arrangement complement is an evaluation of the Tutte polynomial.
Arithmetic matroids were introduced in pursuit of a generalization of this theory to arrangements of codimension~1 subtori in a torus,
with an \emph{arithmetic Tutte polynomial} answering to the Tutte polynomial.
Thus where matroids only record linear dependence among vectors (hyperplane normals),
arithmetic matroids furthermore record \emph{multiplicities} encoding the arithmetic relations of elements in a $\mathbb{Z}$-module (torus characters).
Beyond toric arrangements, the arithmetic Tutte polynomial finds application to vector partition functions, lattice points in zonotopes, graphs and CW-complexes.

As special cases of our universal convolution formula, 
we will recover a recent convolution formula of Backman and Lenz \cite{backmanlenz} for arithmetic matroids,
from which the Kook--Reiner--Stanton--Etienne--Las Vergnas formula for matroids is obtained by a forgetful morphism of minors systems.
Furthermore, from a slight modification of this bialgebra we obtain a more general formula, which expresses the arithmetic Tutte polynomial of the \emph{product} of two arithmetic matroids (as defined in \cite{delucchimoci}) as the convolution of the arithmetic Tutte polynomials of its factors (see Theorem~\ref{thm:arith conv}).

	\subsection{The minors system of arithmetic matroids}
		
		Our presentation of arithmetic matroids follows mostly \cite{BM},
		though our notation deviates in that we name arithmetic matroids as pairs~$(M,m)$
		rather than giving them simple names like~$M$. 

		\begin{defi} A {\em molecule} in a matroid $M=(E,\rk)$ is a triple $\alpha:=(R,F,T)$ of disjoint subsets of $E$ such that, for every $A\subseteq E $ with $R\subseteq A \subseteq R\cup F\cup T$,
	  	$$\rk(A) = \rk(R) + \vert A\cap F \vert.$$
	  	\end{defi}
	  	
		For those familiar with matroid theory, a molecule of~$M$ is the data indexing a minor $(M|R\cup F\cup T)/R$ of~$M$ consisting only of coloops and loops, together with its partition $F\sqcup T$ into the coloops and loops respectively. Notice that if $\alpha=(R,F,T)$ is a molecule, then so is the triple $(R',F',T')$ for every $F'\subseteq F$, every $T'\subseteq T$ and every $R'$ with $R\subseteq R'\subseteq (R\cup F\cup T)\setminus (F'\cup T')$.
	
		\begin{defi}
		\label{def_AM}
		 An {\em arithmetic matroid} is a pair $(M,m)$ where $M=(E,\rk)$ is a matroid, and $m: 2^E \to \mathbb Z_{>0}$ is a function (the \emph{multiplicity function}) satisfying the following axioms.
		 \begin{itemize}
 		 \item[(A1)] For all $A\subseteq E$ and all $e\in E$, if $\rk(A\cup e)> \rk(A)$, then $m(A)$ divides $m(A\cup e)$; if $\rk(A\cup e)= \rk(A)$, then $m(A\cup e)$ divides $m(A)$.
  		 \item[(A2)]  For every molecule $\alpha=(R,F,T)$ of $(E,\rk)$
 		$$m(R)\,m(R\cup F\cup T) = m(R\cup F)\,m(R\cup T).$$
 		\item[(P)] For every molecule $\alpha=(R,F,T)$ of $(E,\rk)$,
 		$$ (-1)^{T}\sum_{R\subseteq A \subseteq R\cup F\cup T} (-1)^{\vert (R\cup F\cup T) \setminus A \vert}\, m(A)\geq 0$$
  		\end{itemize}
		\end{defi}

		The \emph{restriction} of $(M,m)$ to $A\subseteq E$, denoted $(M,m)|A$, has underlying matroid the restriction $M| A$ and multiplicity function the restriction of $m$. The \emph{contraction} of $(M,m)$ by a subset $A\subseteq E$, denoted $(M,m)/A$, has underlying matroid the contraction $M/A$ and multiplicity function $\overline m$ given by $\overline m(B) \doteq m(B \cup A)$ for $B \subseteq  E(M) \setminus A$. The \emph{direct sum} of two arithmetic matroids $(M,m)$ and $(M',m')$ is the arithmetic matroid $(M\oplus M',m\oplus m')$, where $M\oplus M'$ is the direct sum of matroids and the multiplicity function is defined by 
		$$(m\oplus m')(A\sqcup A') = m(A)\,m'(A').$$
	
		\begin{rem}\label{rem:matroids are normalized}
		Matroid axiom (R1) implies that $\rk(\varnothing)=0$ in any matroid.
		This can be viewed as a ``normalization'' condition which must be taken into account in the definition of operations like contractions:
		it is the reason for the $-\rk(A)$ summand in the definition of matroid contraction.
		A na\"ive inventor of matroid theory might have chosen to omit this normalization, producing a disconnected (multiplicative) minors system.
		Had this been done, the familiar matroid Hopf algebra would be recovered from the resulting bialgebra by Remark~\ref{rem:Hopf quotient}.
		
		It is notable that the definition of arithmetic matroids lacks a similar normalization for multiplicities.
		Had one been included, say by dividing multiplicities through by a scalar to force $m(\varnothing)=1$,
		a connected minors system and Hopf algebra would be produced, 
		rather than the disconnected minors system and bialgebra below.
		But, as we will see, the gain of connectedness would incur the great cost of no longer being able to recover the arithmetic Tutte polynomial from the universal one.
		\end{rem}
		
		Equipped with these notions of restriction, contraction and direct sum, arithmetic matroids form a multiplicative minors system that we denote $\mathsf{AMat}$. An object in $\mathsf{AMat}[\varnothing]$ is the datum of an integer $m=m(\varnothing)\in\mathbb{Z}_{>0}$. Thus, the minors system of arithmetic matroids is not connected and we have an isomorphism of monoids 
		$$\mathsf{AMat}[\varnothing] \simeq (\mathbb{Z}_{>0},\times)\simeq \prod_{p \textnormal{ prime}} a_p^{\mathbb{N}}\ .$$
		We denote by $[m]=\prod_pa_p^{v_p(m)}$ the class of an integer $m\in\mathbb{Z}_{>0}$ in the monoid ring $\KK[\mathbb{Z}_{>0}]\simeq\KK[\{a_p\}]$.\medskip
		
		There is a natural inclusion $\mathsf{Mat}\hookrightarrow\mathsf{AMat}$ whose image consists of the arithmetic matroids with constant multiplicity $1$. It has a section $\mathsf{AMat}\twoheadrightarrow\mathsf{Mat}$ which forgets about the multiplicity function.
		
	\subsection{The universal Tutte character for arithmetic matroids}
	
		The Grothendieck monoid of $\mathsf{AMat}$ is easy to compute.
	
		\begin{prop}
		There is an embedding of monoids
		$$U(\mathsf{AMat}) \hookrightarrow \mathbb{Q}_{>0} \times u^{\mathbb{N}}v^{\mathbb{N}} \;\; ,\; \; [M,m]\mapsto \frac{m(E(M))}{m(\varnothing)} \, u^{\rk(M)}v^{\cork(M)} \ ,$$
		whose image is 
		\[\{1\}\cup\{au^i:a\in\mathbb Z_{\geq1},i>0\}\cup\{\tfrac{1}{a}v^j:a\in\mathbb Z_{\geq1},j>0\}\cup\{qu^iv^j:q\in\mathbb Q_{>0},i,j>0\}.\]
		\end{prop}
		
		\begin{proof}
		The morphism, that we denote by $\alpha$, is easily seen to be well-defined, and we want to prove that it is injective. We note that $U(\mathsf{AMat})$ is generated by the classes of the arithmetic matroids $c_a$, the coloop with empty multiplicity $1$ and total multiplicity $a$, and $l_a$, the loop with empty multiplicity $a$ and total multiplicity $1$, for $a\in\mathbb{Z}_{\geq 1}$. We introduce the notation
		$$\gamma_a = \alpha([c_a]) = au \;\;  , \;\; \lambda_a = \alpha([l_a]) = \tfrac{1}{a}v\ .$$
		So the image of $\alpha$ is generated by such classes.
		If $qu^iv^j$ is an element of $\im\alpha$ such that $q\not\in\mathbb Z$, then any expression for it as a product of generators must contain a factor $\lambda_a$, implying $j>0$;
		dually, if $q^{-1}\not\in\mathbb Z$, we conclude $i>0$.
		Conversely, for $a,b\in\mathbb Z_{\geq1}$ and $i,j>0$ we see that $au^i = \gamma_a(\gamma_1)^{i-1}$ and $1/b\cdot v^j=\lambda_b(\lambda_1)^{j-1}$ are in $\im\alpha$, as is their product,
		which proves the proposition's claim about the image.
		
		One readily verifies that the relations satisfied by the elements $\gamma_a$ and $\lambda_a$ in the monoid $\mathbb{Q}_{>0}\times u^{\mathbb{N}}v^{\mathbb{N}}$ are generated by the following relations:
		\begin{enumerate}
		\item $\gamma_a\lambda_a=\gamma_1\lambda_1$;
		\item $\lambda_a\lambda_b=\lambda_{ab}\lambda_1$;
		\item $\gamma_a\gamma_b=\gamma_{ab}\gamma_1$.
		\end{enumerate}
		What we have to do is to prove that the corresponding relations are satisfied by $[c_a]$ and $[l_a]$ in $U(\mathsf{AMat})$. 
		\begin{enumerate}
		\item Let $(M,m)$ be the arithmetic matroid induced by the list of vectors $e=1, f=a$ in the abelian group $\mathbb{Z}$. Then $M|e=c_1$, $M/e=l_1$, $M|f=c_a$ and $M/f=l_a$. We thus get $[c_a][l_a]=[(M,m)]=[c_1][l_1]$.
		\item Let $(M,m)$ be the arithmetic matroid induced by the list of vectors $e=1$, $f=a$ in the abelian group $\mathbb{Z}/ab\mathbb{Z}$. Then $M|e=l_{ab}$, $M/e=l_1$, $M|f=(\varnothing,a)\oplus l_b$ and $M/f=l_a$. We thus get $[l_a][l_b]=[(M,m)]=[l_{ab}][l_1]$.
		\item This is dual to the previous case.\qedhere
		\end{enumerate}
		\end{proof}
		
		The corresponding universal Tutte character
		\begin{equation*}
		\begin{split}
		T^{\mathsf{AMat}}:\KK\mathsf{AMat}\rightarrow \KK[&U(\mathsf{AMat})\times\mathsf{AMat}[\varnothing]\times U(\mathsf{AMat})] \\
		& \hookrightarrow \KK[\mathbb{Q}_{>0}\times\mathbb{Z}_{>0}\times\mathbb{Q}_{>0}][u_1,v_1,u_2,v_2]
		\end{split}
		\end{equation*}
		is thus given by:
		\begin{equation*}
		\begin{split}
		T&^{\mathsf{AMat}}(M,m)\\
		&=\sum_{A\subseteq E(M)}\left[\frac{m(A)}{m(\varnothing)}\,,m(A)\,,\frac{m(E(M))}{m(A)}\right] u_1^{\rk(M|A)}v_1^{\cork(M|A)}u_2^{\rk(M/A)}v_2^{\cork(M/A)}\ .
		\end{split}
		\end{equation*}
		The coefficient inside the square brackets contains the information of the multiplicity $m(A)$ along with the empty multiplicity $m(\varnothing)$ and the total multiplicity $m(E(M))$. It is natural to forget about these two last pieces of information which are independent of $A$ and only retain the information of $m(A)\in\mathbb{Z}_{>0}$. Regarding the variables $u_1$, $v_1$, $u_2$, $v_2$, the story is the same as in the case of matroids and to be consistent with that case we work with the same conventional specialization. In other words, we work after the specialization:
		$$ \KK[\mathbb{Q}_{>0}\times\mathbb{Z}_{>0}\times\mathbb{Q}_{>0}][u_1,v_1,u_2,v_2] \longrightarrow \KK[\mathbb{Z}_{>0}][x,y]$$
		induced by the projection $\mathbb{Q}_{>0}\times\mathbb{Z}_{>0}\times\mathbb{Q}_{>0}\rightarrow \mathbb{Z}_{>0}$ and by $(u_1,v_1,u_2,v_2)=(1,y-1,x-1,1)$. To sum things up, we are looking at the Tutte character with values in the algebra $\KK[\mathbb{Z}_{>0}][x,y]$ corresponding to the norms $N_1(M,m)=(y-1)^{\cork(M)}$ and $N_2(M,m)=(x-1)^{\rk(M)}$, and to the twist map $\tau(\varnothing,m)=[m]$. The following definition gives the explicit formula for it.
		
		\begin{defi}
		The \emph{universal arithmetic Tutte polynomial} of an arithmetic matroid $(M,m)$ is the polynomial in $\KK[\mathbb{Z}_{>0}][x,y]$ defined as
		$$\mathfrak{M}^{\,\mathrm{uni}}_{(M,m)}(x,y)=\sum_{A\subseteq E(M)} [m(A)]\,(x-1)^{\rk(M)-\rk(A)}(y-1)^{|A|-\rk(A)} \ .$$
		\end{defi}
		
		Strictly speaking, this is not a reduced Tutte character in the sense of Section \ref{ssec:reduced} since we used the twist map to forget more information on the norms side. However, the universal Tutte character $T^{\mathsf{AMat}}(M,m)$ can be recovered from the universal arithmetic Tutte polynomial $\mathfrak{M}^{\,\mathrm{uni}}_{(M,m)}(x,y)$ by a change of variables and multiplication by a pre-factor. We leave the task of writing a precise formula to the interested reader.\medskip
		
		The universal arithmetic Tutte polynomial can be computed recursively using the same deletion-contraction recurrence formula as the Tutte polynomial, with the base case $\mathfrak{M}^{\,\mathrm{uni}}_{(\varnothing,m)}(x,y)=[m]$.
		
	\subsection{Specializations}
	
		Fix $\KK=\mathbb{Z}$ for simplicity. We describe notable specializations of the universal arithmetic Tutte polynomial.

		\subsubsection{The Tutte polynomial}
		The morphism of rings $\mathbb{Z}[\mathbb{Z}_{>0}]	\rightarrow \mathbb{Z}$ induced by the morphism of monoids $\mathbb{Z}_{>0}\rightarrow 1\hookrightarrow \mathbb{Z}$ sends $a_p$ to $1$ for every prime $p$. By applying it to the coefficients of $\mathfrak{M}^{\,\mathrm{uni}}_{(M,m)}(x,y)$ one recovers the classical Tutte polynomial $\mathfrak{T}_{M}(x,y)$ of the underlying matroid.

		\subsubsection{The arithmetic Tutte polynomial}
		 The morphism of rings $\mathbb{Z}[\mathbb{Z}_{>0}]\rightarrow \mathbb{Z}$ induced by the morphism of monoids $\mathbb{Z}_{>0}\hookrightarrow \mathbb{Z}$ sends $a_p$ to $p$ for every prime $p$. By applying it to the coefficients of $\mathfrak{M}^{\,\mathrm{uni}}_{(M,m)}(x,y)$ one recovers the arithmetic Tutte polynomial in $\mathbb{Z}[x,y]$ defined in \cite{M1} and \cite{Md1}:
		$$\mathfrak{M}_{(M,m)}(x,y)=\sum_{A\subseteq E(M)} m(A)\,(x-1)^{\rk(M)-\rk(A)}(y-1)^{|A|-\rk(A)} \ .$$
		
		\subsubsection{The $p$-local arithmetic Tutte polynomial and the Tutte quasi-polynomial}
		Let us write $m=\prod_{p\textnormal{ prime}} m_p$ for the decomposition of an integer $m\in\mathbb{Z}_{>0}$ into its $p$-local parts. For a fixed prime $p$, the assignment $[m]\mapsto m_p$ induces a morphism of rings $\mathbb{Z}[\mathbb{Z}_{>0}]\rightarrow\mathbb{Z}$, which sends $a_p$ to $p$ and $a_q$ to $1$ for $q\neq p$. By applying it to the coefficients of $\mathfrak{M}^{\,\mathrm{uni}}_{(M,m)}(x,y)$ one gets a \emph{$p$-local arithmetic Tutte polynomial}
		$$\mathfrak{M}^{(p)}_{(M,m)}(x,y)=\sum_{A\subseteq E(M)} m_p(A)\, (x-1)^{\rk(M)-\rk(A)}(y-1)^{|A|-\rk(A)}\ .$$

		A \emph{Tutte quasi-polynomial} $Q(x,y)$ associated to a list $\alpha$ of elements in a finitely generated $\mathbb Z$-module was introduced in~\cite{BM} (our notation is that of \cite{FM})
		with connections to toric arrangements, Dahmen--Micchelli spaces, and zonotopes, among other topics.
		It is quasi-polynomial in the quantity $q = (x-1)(y-1)$, that is, 
		there exist an integer $m_\alpha$ and for each $i\in\mathbb Z/m_\alpha\mathbb Z$ a \emph{constituent} polynomial $Q_i(x,y)$ 
		such that for any $x,y\in\mathbb Z$ we have
		\[Q(x,y) = Q_{[q]}(x,y)\ .\]
		If $(M,m)$ is the arithmetic matroid of the same list $\alpha$, 
		then $m_\alpha$ may be chosen to be $\operatorname{lcm}\{m(A) : A\subseteq E(M)\}$,
		and \cite[Theorem~6.4]{BM} asserts that
		$Q_{[0]}(x,y) = \mathfrak{M}_{(M,m)}(x,y)$, while
		$Q_i(x,y) = \mathfrak{T}_M(x,y)$ if $i$~is a unit in~$\mathbb Z/m_\alpha\mathbb Z$.
		
		A generalization of these facts follows from \cite[Section~7.1]{FM}.
		Let $m_\alpha = (m_\alpha)_p\cdot r$.
		Then if $i\in\mathbb Z/m_\alpha\mathbb Z$ maps under the Chinese remainder isomorphism
		\[\mathbb Z/m_\alpha\mathbb Z\simeq\mathbb Z/(m_\alpha)_p\mathbb Z\times\mathbb Z/r\mathbb Z\]
		to $(0,i')$ where $i'$ is a unit, we have that $Q_i(x,y) = \mathfrak{M}^{(p)}_{(M,m)}(x,y)$.
		More generally, one can use the specialization of $\mathfrak{M}^{\,\mathrm{uni}}_{(M,m)}(x,y)$
		arising from the product of several of the morphisms $[m]\mapsto m_p$
		to compute any of the constituents $Q_i(x,y)$ where $i$ is either zero or invertible in each of the indecomposible direct factors of $\mathbb Z/m_\alpha\mathbb Z$.
		
		The constituents $Q_i(x,y)$ where $i$ is not of the above form are not invariants of an arithmetic matroid at all.
		To capture them requires a richer combinatorial object, such as the matroids over~$\mathbb Z$ of \cite{FM}.
		The requisite data also appear in the \emph{$G$-Tutte polynomial} of \cite{Yoshinaga}, and
		Theorem~5.5 of that work explains how to recover from the $G$-Tutte polynomial
		every constituent of the \emph{characteristic quasi-polynomial}, which is up to a sign the evaluation $Q(1-t,0)$ of the Tutte quasi-polynomial.

	\begin{samepage}
	\subsection{The convolution formula}
	
		\subsubsection{Product of arithmetic matroids}			
			Consider a fixed matroid $M$, two possibly different functions $m_1, m_2 : 2^{E(M)} \to \mathbb Z$ and their pointwise product $m_1m_2$ defined for $A\subseteq E(M)$ by $(m_1m_2)(A)\doteq m_1(A)\, m_2(A)$. The following theorem was proved in \cite{delucchimoci}.
	\end{samepage}

			\begin{thm}\label{Thm_prod}
			If both $(M,m_1)$ and $(M,m_2)$ are arithmetic matroids, then $(M,m_1m_2)$ is also an arithmetic matroid.
			\end{thm}

			We remark that there the key point was to prove that $(M,m_1m_2)$ satisfies axiom (P), since (A1) and (A2) hold trivially.
			\begin{defi}
			We call  $(M,m_1m_2)$ the \emph{product} of $(M,m_1)$ and $(M,m_2)$.
			\end{defi}
			This operation makes the set of arithmetic matroids over a fixed underlying matroid into a commutative monoid, with unit given by the trivial multiplicity $m(A)=1$ for all  $A\subseteq E$.
			
		\subsubsection{Bi-arithmetic matroids and the convolution formula}
		
			\begin{defi}
			A \emph{bi-arithmetic matroid} is a triple $(M,m_1,m_2)$ where $M$ is a matroid and $m_1, m_2:2^{E(M)}\rightarrow\mathbb{Z}_{>0}$ are two functions such that $(M,m_1)$ and $(M,m_2)$ are arithmetic matroids.
			\end{defi}
			
			Bi-arithmetic matroids form a multiplicative minors system denoted by $\mathsf{A^2Mat}$, for which the notions of restriction, contraction and direct sum are the obvious ones. In categorical terms this can be interpreted as the fiber product
			$$\mathsf{A^2Mat} = \mathsf{AMat}\times_{\mathsf{Mat}}\mathsf{AMat}\ .$$
			The product of arithmetic matroids should then be viewed as a morphism
			$$\mathsf{AMat}\times_{\mathsf{Mat}}\mathsf{AMat}\rightarrow\mathsf{AMat}\, , \,((M,m_1), (M,m_2)) \mapsto (M,m_1m_2)\ ,$$
			which gives $\mathsf{AMat}$ the structure of a monoid in the category of minors systems over $\mathsf{Mat}$.\medskip
			
			We now prove a convolution formula for the universal arithmetic Tutte polynomial.
			
			\begin{thm}\label{thm:arith conv}
			The universal arithmetic Tutte polynomial satisfies the following analogue of Kung's formula \eqref{eq: general convolution formula tutte matroids} in the algebra
			$\KK[\mathbb{Z}_{>0}][a,b,c,d]$:
			\begin{equation*}			
			\begin{split}
			&\mathfrak{M}^{\,\mathrm{uni}}_{(M,m_1m_2)} (1-ab,1-cd)  \\ 
			&=\sum_{A\subseteq E(M)} a^{\rk(M)-\rk(A)}d^{|A|-\rk(A)}  \,\mathfrak{M}^{\,\mathrm{uni}}_{(M,m_1)|A}(1-a,1-c)\,\mathfrak{M}^{\,\mathrm{uni}}_{(M,m_2)/A}(1-b,1-d)\ .
			\end{split}
			\end{equation*}
			\end{thm}
			
			\begin{proof}
			The proof is the same as in the case of matroids (Proposition \ref{prop:convolution matroids}) by taking care of the twist maps. For the sake of completeness, we show how it follows from the general convolution formula given in Theorem \ref{thm:general convolution formula}. Let's prove the latter formula first. We consider the norms $N_0,N_1,N_2:\KK\mathsf{A^2Mat}\rightarrow \KK[\mathbb{Z}_{>0}][a,b,c,d]$ which map a bi-arithmetic matroid $(M,m_1,m_2)$ to
			$$(-cd)^{\cork(M)} \;\; ,\;\; (-a)^{\rk(M)}d^{\cork(M)}  \;\;,\;\;(-ab)^{\rk(M)}\ ,$$
			respectively.
			We consider the twist maps $\tau_1,\tau_2:\KK\mathsf{A^2Mat}\rightarrow\KK[\mathbb{Z}_{>0}][a,b,c,d]$ defined by 
			$$\tau_1(\varnothing,m_1,m_2)=[m_1]\;\; , \;\; \tau_2(\varnothing,m_1,m_2)=[m_2]\ .$$
			We have the Tutte characters
			$$T_{N_0,\tau_1\tau_2,N_2}(M,m_1,m_2)=\mathfrak{M}^{\,\mathrm{uni}}_{(M,m_1m_2)}(1-ab,1-cd)\ ;$$
			$$T_{N_0,\tau_1,N_1}(M,m_1,m_2)=d^{\cork(M)}\mathfrak{M}^{\,\mathrm{uni}}_{(M,m_1)}(1-a,1-c)\ ;$$
			$$T_{\overline{N_1},\tau_2,N_2}(M,m_1,m_2)=a^{\rk(M)}\mathfrak{M}^{\,\mathrm{uni}}_{(M,m_2)}(1-b,1-d)\ .$$
			The claim is thus a consequence of Theorem \ref{thm:general convolution formula}. 
			\end{proof}
			
			As a specialization we obtain the following convolution formula in the spirit of the formula of Kook--Reiner--Stanton and Etienne--Las Vergnas.
			
			\begin{coro}
			The universal arithmetic Tutte polynomial satisfies the following convolution formula in the algebra $\KK[\mathbb{Z}_{>0}][x,y]$:
			\begin{equation}\label{aritconv}
			  \mathfrak{M}^{\,\mathrm{uni}}_{(M,m_1m_2)} (x,y) 
			=\sum_{A\subseteq E(M)} \mathfrak{M}^{\,\mathrm{uni}}_{(M,m_1)|A}(0,y)\,\mathfrak{M}^{\,\mathrm{uni}}_{(M,m_2)/A}(x,0). 
			\end{equation}
			\end{coro}
			
			\begin{proof}
			This follows from Theorem \ref{thm:arith conv} by specializing $(a,b,c,d)$ to $(1,1-y,1-x,1).$
			\end{proof}
			
			All the specializations and variants of the convolution formula that we discussed in the case of matroids are also available here. We only mention one that was proved by Backman and Lenz \cite{backmanlenz}. For an arithmetic matroid $(M,m)$ we have the following convolution formula for the arithmetic Tutte polynomial in the ring $\mathbb{Z}[x,y]$:
			\begin{equation*}
			\begin{split}
			\mathfrak{M}_{(M,m)}(x,y) &= \sum_{A\subseteq E(M)} \mathfrak{M}_{(M,m)|A}(0,y) \, \mathfrak{T}_{M/A}(x,0)\\
			& = \sum_{A\subseteq E(M)} \mathfrak{T}_{M|A}(0,y)\,\mathfrak{M}_{(M,m)/A}(x,0)\ .
			\end{split}
			\end{equation*}
			This follows from formula (\ref{aritconv}) in the case $(m_1,m_2)=(m,1)$ or $(1,m)$ by the projection ${\mathbb{Z}}_{>0}\rightarrow 1$.

\longthanks{
The genesis of this project took place at the 2016 Borel Seminar
in conversations between the second and third authors and
Spencer Backman and Matthias Lenz, whom we thank,
together with the seminar organisers for bringing us together. The first author thanks Camille Combe and Emeric Gioan for many interesting conversations on the themes of this paper.
Joseph Kung and Iain Moffatt had many valuable comments on a draft of this paper.
We also thank the anonymous referees for their suggestions.}

\nocite{*}
\bibliographystyle{amsplain-ac}
\bibliography{biblio1}
\end{document}